\documentclass[11pt,a4paper,reqno]{amsart}
\usepackage[applemac]{inputenc}
\usepackage[T1]{fontenc}
\usepackage{amsmath}
\usepackage{amsthm}
\usepackage{amsfonts}
\usepackage{amssymb}
\usepackage{graphicx}
\usepackage{amsbsy}
\usepackage{mathrsfs}
\usepackage{bbm}
\newtheorem{theorem}{Theorem}[section]
\newtheorem{lemma}[theorem]{Lemma}

\newtheorem{proposition}[theorem]{Proposition}

\newtheorem{definition}[theorem]{Definition}
\theoremstyle{remark}
\newtheorem{remark}[theorem]{\it \bf{Remark}\/}

\numberwithin{equation}{section}
\catcode`@=11
\def\section{\@startsection{section}{1}%
  \z@{1.5\linespacing\@plus\linespacing}{.5\linespacing}%
  {\normalfont\bfseries\large\centering}}
\catcode`@=12
\newcommand{\be}{\begin{equation}}
\newcommand{\ee}{\end{equation}}
\newcommand{\bea}{\begin{eqnarray}}
\newcommand{\eea}{\end{eqnarray}}
\newcommand{\bee}{\begin{eqnarray*}}
\newcommand{\eee}{\end{eqnarray*}}

\def\pa{\partial}

\def\RR{\mathbb{R}}

\def\fref#1{{\rm (\ref{#1})}}

\def\G{{\Gamma}}

\catcode`@=11
\def\supess{\mathop{\operator@font Sup\,ess}}
\catcode`@=12

\def\bt{\tilde{b}}

\def\RR{\mathbb{R}}

\def\e{\varepsilon}

\def\g#1{{\bf #1}}

\def\fref#1{{\rm (\ref{#1})}}

\def\R2+{\RR ^2_+}

\def\lsl{\frac{\lambda_s}{\lambda}}

\def\pa{\partial}

\def\lim{\mathop{\rm lim}}

\def\sup{\mathop{\rm sup}}

\def\l{\lambda}

\def\log{{\rm log}}

\def\et{\tilde{\e}}

\def\lsl{\frac{\lambda_s}{\lambda}}

\def\qbt{\tilde{Q}_b}

\def\tt{\tilde{T}}

\def\Sigmat{\tilde{\Sigma}}
\def\pa{\partial}

\def\et{\tilde{\e}}

\def\alphah{\hat{\alpha}}

\def\pa{\partial}
\def\alphat{\tilde{\alpha}}
\def\Psit{\tilde{\Psi}}

\def\tT{\tilde{T}}
\def\tS{\tilde{S}}
\def\matchal{\mathcal}

\title[]{Quantized slow blow up dynamics for the corotational energy critical harmonic heat flow}
\author[P. Rapha\"el]{Pierre Rapha\"el}
\address{Laboratoire J.A. Dieudonn\'e, Universit\'e de Nice Sophia Antipolis\\
 et Institut Universitaire de France}
\email{praphael@unice.fr}
\author[R. Schweyer]{Remi Schweyer}
\address{Institut de Math\'ematiques de Toulouse, Universit\'e Paul  Sabatier, Toulouse, France}
\email{remi.schweyer@math.univ-toulouse.fr}

\begin{document}
\maketitle

\begin{abstract}
We consider the energy critical harmonic heat flow from $\Bbb R^2$ into a smooth compact revolution surface  of $\Bbb R^3$. For initial data with corotational symmetry, the evolution reduces to the semilinear radially symmetric parabolic problem $$\partial_t u -\pa^2_{r} u-\frac{\pa_r u}{r} + \frac{f(u)}{r^2}=0$$ for a suitable class of functions $f$ . Given an integer $L\in \Bbb N^*$, we exhibit a set of initial data arbitrarily close to the least energy harmonic map $Q$ in the energy critical topology such that the corresponding solution blows up in finite time by concentrating its energy $$\nabla u(t,r)-\nabla Q\left(\frac{r}{\l(t)}\right)\to u^*\ \ \mbox{in}\ \  L^2$$ at a speed given by the {\it quantized} rates: $$\l(t)=c(u_0)(1+o(1))\frac{(T-t)^L}{|\log (T-t)|^{\frac{2L}{2L-1}}},$$ in accordance with the formal predictions \cite{heatflow}. The case $L=1$ corresponds to the stable regime exhibited in \cite{RSc1}, and the data for $L\ge 2$ leave on a manifold of codimension $(L-1)$ in some weak sense. Our analysis lies in the continuation of \cite{RaphRod}, \cite{MRR}, \cite{RSc1} by further exhibiting the mechanism for the existence of the excited slow blow up rates and the associated instability of these threshold dynamics.\end{abstract}


\section{Introduction}



\subsection{The parabolic heat flow}


The harmonic heat flow between two embedded Riemanian manifolds $(N,g_N),(M,g_M)$ is the gradient flow associated to the Dirichlet energy of maps from $N\to M$: 
\be
\label{hfgeneral}
\left\{\begin{array}{ll}\pa_tv={\Bbb P}_{T_vM}(\Delta_{g_N}v)\\
v_{|t=0}=v_0\end{array}\right . \ \ (t,x)\in \Bbb R\times N, \ \ v(t,x)\in M
\ee where ${\Bbb P}_{T_vM}$ is the projection onto the tangent space to $M$ at v. The special case $N=\Bbb R^2$, $M=\Bbb S^2$ corresponds to the harmonic heat flow to the 2-sphere
\be
\label{harmonicheatflow}
\pa_tv=\Delta v+|\nabla v|^2v, \ \ (t,x)\in \Bbb R\times \Bbb R^2,\ \ v(t,x)\in \Bbb S^2
\ee
and is related to the Landau Lifschitz equation of ferromagnetism, we refer to \cite{heatflow}, \cite{matano}, \cite{NT1}, \cite{NT2} and references therein for a complete introduction to this class of problems. We shall from now on restrict our discussion to the case: $$N=\Bbb R^2.$$
Smooth initial data yield unique local in time smooth solutions which dissipate the Dirichlet energy $$\frac{d}{dt}\left\{\int_{\Bbb R^2}|\nabla v|^2\right\}=-2\int_{\Bbb R^2}|\pa_tv|^2.$$ An essential feature of the problem is that the scaling symmetry $$u\mapsto u_\l(t,x)=u(\lambda^2t, \lambda x), \ \ \l>0$$ leaves the Dirichlet energy unchanged, and hence the problem is {\it energy critical}.


\subsection{Corotational flows}


We restrict our attention in this paper to flows with so called co-rotational symmetry. More precisely, let a smooth closed curve in the plane parametrized by arclength $$u\in[-\pi,\pi]\mapsto \left|\begin{array}{ll}g(u)\\z(u)\end{array}\right., \ \ (g')^2+(z')^2=1,$$ where 
\be
\label{assumtiong}
(H)\ \ \left\{\begin{array}{lll} g\in \mathcal C^{\infty}(\Bbb R) \ \ \mbox{is odd and $2\pi$ periodic},\\
g(0)=g(\pi)=0, \ \ g(u)>0\ \ \mbox{for}\ \ 0<u<\pi,\\
g'(0)=1, \ \ g'(\pi)=-1,\end{array}\right.
\ee
then the revolution surface $M$ with parametrization $$(\theta,u)\in[0,2\pi]\times[0,\pi]\mapsto \left|\begin{array}{lll} g(u)\cos\theta\\g(u)\sin\theta\\z(u)\end{array}\right.$$ is a smooth\footnote{see eg  \cite{GHL}} compact revolution surface of $\Bbb R^3$ with metric $(du)^2+g^2(u)(d\theta)^2$. Given a homotopy degree $k\in \Bbb Z^*$, the k-corotational reduction to \fref{hfgeneral} corresponds to solutions of the form 
\be
\label{corotshpoere} 
v(t,r)= \left|\begin{array}{lll} g(u(t,r))\cos(k\theta)\\g(u(t,r))\sin(k\theta)\\z(u(t,r))\end{array}\right.
\ee which leads to the semilinear parabolic equation: 
\be
\label{mapg}
\left\{\begin{array}{ll}\pa_tu-\pa^2_ru-\frac{\pa_ru}{r}+k^2\frac{f(u)}{r^2}=0,\\
u_{t=0}=u_0\end{array}\right.\ \ f=gg'.
\ee
The k-corotational Dirichlet energy becomes 
\be
\label{diricheltenergy}
E(u)=\int_0^{+\infty}\left[|\pa_ru|^2+k^2\frac{(g(u))^2}{r^2}\right] rdr
\ee
and is minimized among maps with boundary conditions 
\be
\label{boundary}
u(0)=0, \ \ \lim_{r\to+\infty}u(r)=\pi
\ee
 onto the least energy harmonic map $Q_k$ which is the unique -up to scaling- solution to
 \be
 \label{eqgroudnsttae}
 r\pa_rQ_k=kg(Q_k)
 \ee satisfying \fref{boundary}, see for example \cite{cote}. In the case of $\Bbb S^2$ target $g(u)=\sin u$, the harmonic map is explicitly given by 
\be
\label{deywgvuew}
Q_k(r)=2\tan^{-1}(r^k).
\ee


\subsection{The blow up problem}


The question of the existence of blow up solutions and the description of the associated concentration of energy scenario has attracted a considerable attention for the past thirty years. In the pioneering works of Struwe \cite{Struwe}, Ding and Tian \cite{dingtian}, Qing and Tian \cite{qingtian} (see Topping \cite{Topping} for a complete history of the problem), it is shown that if occurring, the concentration of energy implies the bubbling off of a non trivial harmonic map at a finite number of blow up points 
\be
\label{cnknneooe}
v(t_i,a_i+\lambda(t_i)x)\to  Q_i, \ \ \l(t_i)\to 0
\ee
locally in space. In particular, this shows the existence of a global in time flow on negatively curved targets where no nontrivial harmonic map exists.\\
For corotational data and homotopy number $k\geq 2$, Guan, Gustaffson,  Tsai \cite{NT1}, Gustaffson, Nakanishi, Tsai \cite{NT2} prove that the flow is globally defined near the ground state harmonic map. In fact, $Q_k$ is asymptotically stable for $k\geq 3$, and in particular no blow up will occur, and eternally oscillating solutions and infinite time grow up solutions are exhibited for $k=2$.\\
A contrario, for $k=1$, the existence of finite time blow up solutions  has been proved in various different geometrical settings using strongly the maximum principle, see in particular Chang, Ding, Ye \cite{CDY}, Coron and Ghidaglia \cite{CD}, Qing and Tian \cite{qingtian}, Topping \cite{Topping}. Despite some serious efforts and the use of the maximum principle, see in particular \cite{matano}, very little was known until recently on the description of the blow up bubble and the derivation of the blow up speed, in particular due to the critical nature of the problem.\\

We shall from now on and for the rest of the paper focus onto the degree $$k=1$$ case which generates the least energy no trivial  harmonic map $Q\equiv Q_1$. For $\Bbb D^2$ initial manifold and $\Bbb S^2$ target, Van den Berg, Hulshof and King \cite{heatflow} implement in the continuation of \cite{Velas} a formal analysis based on the matched asymptotics techniques and predict the existence of blow up solutions of the form
\be
\label{cbebnebeib}
u(t,r)\sim Q\left(\frac{r}{\lambda(t)}\right)
\ee
with blow up speed governed by the quantized rates $$\l(t)\sim \frac{(T-t)^L}{|\log (T-t)|^{\frac{2L}{2L-1}}}, \ \ L\in \Bbb N^*.$$ We will further discuss the presence of quantized rates which is reminiscent to the classification of type II blow up for the supercritical nonlinear heat equation \cite{Mizo}.\\ 

We revisited completely the blow up analysis in \cite{RSc1} by adapting the strategy developed in \cite{RaphRod}, \cite{MRR} for the study of wave and Schr\"odinger maps, with two main new input:
\begin{itemize}
\item we completely avoid the formal matched asymptotics approach and replace it by an elementary derivation of an explicit and universal system of ODE's which drives the blow up speed. A similar simplification further occurred in related critical settings, see in particular \cite{RSc2};
\item we designed a robust universal {\it energy method} to control the solution in the blow up regime, and which applies both to parabolic and dispersive problems. In particular, we aim in purpose at making no use of the maximum principle.
\end{itemize}
These techniques led in \cite{RSc1} to the construction of an {\it open set} of corotational initial data arbitrarily close to the ground state harmonic map in the energy critical topology such that the corresponding solution to \fref{mapg} bubbles off an harmonic map according to \fref{cbebnebeib} at the speed $$\l(t)\sim \frac{T-t}{|\log (T-t)|^{2}} \ \ \mbox{i.e.}\  \  L=1.$$ This is the {\it stable}\footnote{in the presence of corotational symmetry, blow up dynamics are expected to be unstable by rotation under general perturbations, see \cite{MRR}.} blow up regime.


\subsection{Statement of the result}


Our main claim in this paper is that the analysis in \cite{RSc1} can be further extended to exhibit the unstable modes which are responsible for {\it a discrete sequence of  quantized slow blow up rates}.

\begin{theorem}[Excited slow blow up dynamics for the 1-corotational heat flow]
\label{thmmain}
Let $k=1$ and $g$ satisfy \fref{assumtiong}. Let $Q$ be the least energy harmonic map. Let $L\in \Bbb N^*$. Then there exists a smooth corotational initial data $u_0(r)$ such that the corresponding solution to \fref{mapg} blows up in finite time $T=T(u_0)>0$ by bubbling off a harmonic map:
\be
\label{concnenergy}
\nabla u(t,r)-\nabla Q\left(\frac{r}{\l(t)}\right)\to \nabla u^*\ \ \mbox{in}\ \ L^2\ \ \mbox{as}\ \ t\to T
\ee
at the excited rate:
\be
\label{Pexciitedlaw}
\l(t)=c(u_0)(1+o_{t\to T}(1))\frac{(T-t)^L}{|\log (T-t)|^{\frac{2L}{2L-1}}}, \ \ c(u_0)>0.
\ee
Moreover, $u_0$ can be taken arbitrarily close to $Q$ in the energy critical topology.
\end{theorem}

{\it Comments on the result:}\\

{\it 1. Regularity of the asymptotic profile}: Arguing as in \cite{RSc1} and using the estimates of Proposition \ref{bootstrap}, one can directly relate the rate of blow up \fref{Pexciitedlaw} to the regularity of the remaining excess of energy, in the sense that $u^*$ exhibits an $H^{L+1}$ regularity is some suitable Sobolev sense, see Remark \ref{remarkreg}. See also \cite{MR5} for a related phenomenon in the dispersive setting.\\

{\it 2. Stable and excited blow up rates}: The case $L=1$ is treated in \cite{RSc1} and corresponds to stable blow up. For $L\geq 2$, the set of initial data leading to \fref{Pexciitedlaw} is of codimension $(L-1)$ in the following sense: there exist fixed directions $(\psi_i)_{2\leq i\leq L}$ such that for any suitable perturbation $\e_0$ of $Q$, there exist $(a_i(\e_0))_{2\leq i\leq L}\in \Bbb R^{L-1}$ such that the solution to \fref{mapg} with data $$Q+\e_0+\Sigma_{i=2}^La_i(\e_0)\psi_i$$ blows up in finite time with the blow up speed \fref{Pexciitedlaw}. Building a smooth manifold would require proving local uniqueness and smoothness of the flow $\e_0\mapsto a_i(\e_0))_{2\leq i\leq L}$ which is a separate problem, we refer for example to \cite{KSNLS} for an introduction to this kind of issues. The control of the unstable modes relies on a classical soft and powerful Brouwer type topological argument in the continuation of \cite{CMM}, \cite{CZ}, \cite{HR}.\\

{\it 3. On quantized blow up rates}. There is an important formal and rigorous literature on the existence of quantized blow up rates for parabolic problems. In the pioneering formal works \cite{Velas}, \cite{FHV}, Filippas, Herrero and Velasquez predicted the existence of a sequence of quantized blow up rates for the supercritical power nonlinearity heat equation $$\pa_tu=\Delta u+u^p, \ \ x\in \Bbb R^d, \ \ p>p(d), \ \ d<7,$$ and this sequence is in one to one correspondence with the spectrum of the linearized operator close the explicit singular self similar solution. After this formal work, and using the a priori bounds on radial type II blow up solutions of Matano and Merle \cite{MaM1}, \cite{MaM2}, Mizogushi completely classified the radial data type II blow up according to these quantized rates. Note that Mizogushi finishes the classification using the Matano-Merle a priori estimates on threshold dynamics which heavily rely on the maximum principle, but the argument is not constructive. One of the main input of our work is to revisit the formal derivation of the sequence of blow up rates and to relate it not to a spectral problem, but to the structure of the resonances of the linearized operator $H$ close to $Q$ {\it and of its iterates}, i.e. the growing solutions to $$H^kT_k=0,  \ \ k\in \Bbb N^*.$$ In particular, we show how the {\it dynamics of tails} as initiated in \cite{RaphRod}, \cite{MRR} leads to a universal dynamical system driving the blow up speed which admits unstable solutions \fref{Pexciitedlaw} corresponding to a codimension $(L-1)$ set of initial data. Another by product of this analysis is the first explicit construction of type II blow up for the energy critical nonlinear heat equation, \cite{Sc}\\

{\it 4. Classification of the flow near $Q$}. The question of the classification of the flow near the harmonic map, and more generally near the ground state solitary wave in nonlinear evolution problems, has attracted a considerable attention recently, see for example \cite{textzurich}. This program has been concluded for the mass critical (gKdV) equation in the series of papers \cite{MMR1}, \cite{MMR2}, \cite{MMR3} where it is shown that provided the data is taken close enough to the ground state {\it in a suitable topology} which is strictly smaller than the energy norm, then the blow up dynamics are completely classified.  A contrario, arbitrarily slow blow up can be achieved for large deformations of the ground state in this restricted sense. The existence of such slow blow up regimes remains however open in many important instances, in particular for the mass critical (NLS), see \cite{MRS2010} for a further introduction to this delicate problem. For energy critical problems like wave or Schr\"odinger maps, the recent construction \cite{KST} shows that arbitrarily slow blow up can be achieved, but the known examples so far are never $\mathcal C^{\infty}$ smooth. The structure of the flow near $Q$ is also somewhat mysterious and various new kind of global dynamics have been constructed, see \cite{DK}, \cite{BT}. One of the new input of our analysis in this paper is to show the essential role played by he control of higher order Sobolev norms which provide a new topology to measure the distance to the solitary wave which is sharp enough to see all the blow up regimes \fref{Pexciitedlaw}. The control of these norms acts in the energy method as a replacement of the counting of the number of intersections of the solution with the ground state  which in the parabolic setting plays an essential role for the classification of the blow up dynamics, \cite{Mizo}, but relies in an essential way on maximum principle techniques. We believe that the blow up solutions we construct in this paper are the building blocks to classify the blow up dynamics near the ground state in a suitable topology.\\

{\it 5. Extension to dispersive problems}. We treat in this paper the parabolic problem, but the works \cite{RaphRod}, \cite{MRR} for the dispersive wave and Schr\"odinger maps with $\Bbb S^2$ target show the robustness of the approach. We expect that similar constructions can be performed there as well to produce arbitrarily slow $\mathcal C^{\infty}$ blow up solutions with quantized rate, hence completing the analysis of these excited regimes started in the seminal work \cite{KST}.\\

\noindent {\bf Aknowledgments:} P.R would like to thank Frank Merle and Igor Rodnianski for stimulating discussions on this problem. P.R and R.S. are supported by the junior ERC/ANR project SWAP. P.R is also party supported by the senior ERC grant BLOWDISOL.\\

\noindent{\bf Notations:} We introduce the differential operator $$\Lambda f=y\cdot \nabla f\ \ (\mbox{energy critical scaling}).$$ Given a parameter $\lambda>0$, we let $$u_\lambda(r)=u(y)\ \ \mbox{with} \ \ y=\frac{r}{\lambda}.$$ Given a positive number $b_1>0$, we let 
\be
\label{defbnot}
B_0=\frac{1}{\sqrt{b_1}}, \ \ B_1=\frac{|\log b_1|}{\sqrt{b_1}}.
\ee
 We let $\chi$ be a positive nonincreasing smooth cut off function with $$\chi(y)=\left\{\begin{array}{ll}1\ \ \mbox{for}\ \ y\leq 1,\\ 0\ \ \mbox{for} \ \ y\geq 2.\end{array}\right.$$ Given a parameter $B>0$, we will denote: 
\be
\label{defchib}
\chi_B(y)=\chi\left(\frac{y}{B}\right).
\ee
We shall systematically omit the measure in all radial two dimensional integrals and note: $$\int f=\int_0^{+\infty}f(r)rdr.$$ Given a $p$-uplet $J=(j_1,\dots,j_p)\in \Bbb N^p$, we introduce the norms: 
\be
\label{normsonpupelts}
|J|_1=\Sigma_{k=1}^pj_k, \ \ |J|_2=\Sigma_{k=1}^p kj_k.
\ee
We note $$\mathcal B_d(R)=\{x\in \Bbb R^d, \ \ |x|=\left(\Sigma_{i=1}^dx_i^2\right)^{\frac 12}\leq R\}.$$

\subsection{Strategy of the proof}


Let us give a brief insight into the strategy of the proof of Theorem \ref{thmmain}.\\

\noindent \emph{(i) Renormalized flow and iterated resonnances}. 

\noindent Let us look for a modulated solution $u(t,r)$ of \fref{mapg} in renormalized form 
\be
\label{renofineo}
u(t,r)=v(s,y), \ \ y=\frac{r}{\lambda(t)},\ \ \frac{ds}{dt}=\frac{1}{\lambda^2(t)}
\ee which leads to the self similar equation: 
\be
\label{eqselfsimilar}
\pa_sv-\Delta v+b_1\Lambda v+\frac{f(v)}{y^2}=0, \ \ b_1=-\lsl.
\ee
We know from theoretical ground that if blow up occurs, $v(s,y)=Q(y)+\e(s,y)$ for some small $\e(s,y)$, and hence the linear part of the $\e$ flow is governed by the Schr\"odinger operator $$H=-\Delta+\frac{f'(Q)}{y^2}.$$ The energy critical structure of the problem induces an explicit resonance: $$H(\Lambda Q)=0$$ where from explicit computation:
\be
\label{slowedecay}
 \Lambda Q\sim \frac{2}{y}\ \ \mbox{as}\ \ y\to \infty.
 \ee
 More generally, the iterates of the kernel of $H$ computed iteratively through the scheme 
 \be
 \label{deft}
 HT_{k+1}=-T_k, \ \ T_0=\Lambda Q,
 \ee display a non trivial tail at infinity:
 \be
 \label{tailtk}
 T_k(y)\sim y^{2k-1}(c_k \log y+d_k)\ \ \mbox{for}\ \ y\gg 1.
 \ee 
 
\noindent \emph{(ii) Tail dynamics}. 
 We now generalize the approach developed in \cite{RaphRod}, \cite{MRR} and claim that $(T_k)_{k\geq 1}$ correspond to unstable directions which can be excited in a universal way. To see this, let us look for a slowly modulated solution to \fref{eqselfsimilar} of the form $v(s,y)=Q_{b(s)}(y)$ with 
 \be
 \label{defqbintro}
 b=(b_1,\dots,b_L),\ \ Q_b=Q(y)+\Sigma_{i=1}^Lb_iT_i(y)+\Sigma_{i=2}^{L+2}S_i(y)
 \ee
 and with a priori bounds $$b_i\sim b_1^i, \ \ |S_i(y)|\lesssim b_1^iy^{C_i},$$ so that $S_i$ is in some sense homogeneous of degree $i$ in $b_1$. Our strategy is the following: choose the universal dynamical system driving the modes $(b_i)_{1\leq i\leq L}$ which generates the {\it least growing in space solution $S_i$}. Let us illustrate the procedure.\\
 
 \noindent \underline{$O(b_1)$}. We do not adjust the law of $b_1$ for the first term\footnote{if $(b_1)_s=-c_1b_1$, then $-\lsl\sim b_1\sim e^{-c_1s}$ and hence after integration in time $|\log \l|\lesssim 1$ and there is no blow up.}. We therefore obtain from \fref{eqselfsimilar} the equation $$b_1(HT_1+\Lambda Q)=0.$$
 \noindent \underline{$O(b_1^2,b_2)$}. We obtain: $$(b_1)_sT_1+b^2_1\Lambda  T_1+b_2HT_2+HS_2=b_1^2NL(T_1,Q)$$  where $NL(T_1,Q)$ corresponds to nonlinear interaction terms. When considering the far away tail \fref{tailtk}, we have for $y$ large, $$\Lambda T_1\sim T_1, \ \ HT_2=-T_1$$ and thus $$(b_1)_sT_1+b^2_1\Lambda T_1+b_2HT_2\sim ((b_1)_s+b_1^2-b_2)T_1,$$ and hence the leading order growth is cancelled by the choice 
 \be
 \label{chocioh}
 (b_1)_s+b_1^2-b_2=0.
 \ee
  We then solve for $$HS_2=-b^2_1(\Lambda T_1-T_1)+ NL(T_1,Q)$$ and check that $S_2\ll b_1^2T_1$ for $y$ large.\\
  \noindent \underline{$O(b_1^{k+1},b_{k+1})$}. At the $k$-th iteration, we obtain an elliptic equation of the form: $$(b_k)_sT_k+b_1b_k\Lambda  T_k+b_{k+1}HT_{k+1}+HS_1=b_1^{k+1}NL_k(T_1,\dots, T_{k},Q).$$ We have from \fref{tailtk} for tails: $$\Lambda T_k\sim (2k-1)T_k$$ and therefore:
  $$(b_{k})_sT_k+b_1b_{k}\Lambda T_k+b_{k+1}HT_{k+1}\sim ((b_k)_s+(2k-1)b_1b_k-b_{k+1})T_k.$$ The cancellation of the leading order growth occurs for $$
  (b_{k})_s+(2k-1)b_1b_{k}-b_{k+1}=0.$$ We then solve for the remaining $S_{k+1}$ term and check that $S_{k+1}\ll b_1^{k+1}T_{k+1}$ for y large.\\
  
 \noindent \emph{(iii) The universal system of ODE's}. The above approach leads to the universal system of ODE's which we stop after the $L$-th iterate:
 \be
\label{systdynfundone}
(b_k)_s+\left(2k-1\right)b_1b_k-b_{k+1}=0,  \ \ 1\leq k\leq L, \ \ b_{L+1}\equiv 0, \ \ -\lsl=b_1.
\ee
It turns out, and this is classical for critical problems, that an additional logarithmic gain related to the growth \fref{tailtk} can be captured, and this turns out to be essential for the analysis\footnote{see for example \cite{RaphRod} for a further discussion.}. This leads to the sharp dynamical system:
\be
\label{systdynfundintro}
\left\{\begin{array}{lll}(b_k)_s+\left(2k-1+\frac{2}{|\log b_1|}\right)b_1b_k-b_{k+1}=0,  \ \ 1\leq k\leq L, \ \ b_{L+1}\equiv 0,\\
-\lsl=b_1,\\
\frac{ds}{dt}=\frac1{\l^2}.
\end{array}\right.
\ee
It is easily seen -Lemma \ref{lemmaexplicitsol}- that \fref{systdynfundintro} rewritten in the original $t$ time variable admits solutions such that $\l(t)$ touches 0 in finite time $T$ with the asymptotic \fref{Pexciitedlaw}, equivalently in renormalized variables:
\be
\label{odelrejropw}
\l(s)\sim \frac{(\log s)^{|d_1|}}{s^{c_1}}, \ \ b(s)\sim \frac{c_1}{s} \ \ \mbox{with}\ \ c_1=\frac{L}{2L-1},  \ \ d_1=\frac{-2L}{(2L-1)^2}.
\ee
 Moreoever -Lemma \ref{lemmalinear}-, the corresponding solution is stable for $L=1$, this is the stable blow up regime, and unstable with $(L-1)$ directions of instabilities for $L\geq 2$.\\
 
\noindent \emph{(iv). Decomposition of the flow and modulation equations}.

\noindent  Let then the approximate solution $Q_b$ be given by \fref{defqbintro} which by construction generates an approximate solution to the renomalized flow \fref{eqselfsimilar}:
$$\Psi_b = \pa_sQ_b -\Delta Q_b+ b \Lambda Q_b+ \frac{f(Q_b)}{y^2}=\rm{Mod}(t)+O(b^{2L+2})$$ where roughly 
$$\rm{Mod}(t) = \Sigma_{i=1}^{L}\left[(b_i)_s+(2i-1+\frac{2}{|\log b_1|})b_1b_i-b_{i+1}\right]T_i.$$ We localize $Q_b$ in the zone $y\leq B_1$ to avoid the irrelevant growing tails for $y\gg \frac{1}{\sqrt{b_1}}$. We then pick an initial data of the form $$u_0(y)=Q_{b}(y)+\e_0(y), \ \ |\e_0(y)|\ll 1$$ in some suitable sense where $b(0)$ is chosen initially close to the exact excited solution to \fref{systdynfundone}. From standard modulation argument, we dynamically introduce a modulated decomposition of the flow 
\be
\label{decomntoro}
u(t,r)=(Q_{b(t)}+\e)\left(t,\frac{r}{\lambda(t)}\right)=(Q_{b(t)})\left(t,\frac{r}{\lambda(t)}\right)+w(t,r)
\ee
where the $L+1$ modulation parameters $(b(t),\l(t))$ are chosen in order to manufacture the orthogonality conditions: 
\be
\label{vnekoenono}
(\e(t),H^k\Phi_M)=0, \ \ 0\leq k\leq M.
\ee
Here $\Phi_M(y)$ is some fixed direction depending on some large constant $M$ which generates an approximation of the kernel of the iterates of $H$, see \fref{defdirection}. This orthogonal decomposition, which for each fixed time t directly follows from the implicit function theorem, now allows us to compute the modulation equations governing the parameters $(b(t),\l(t))$. The $Q_b$ construction is precisely manufactured to produce the expected ODE's\footnote{see Lemma \ref{modulationequations}.}:
\be
\label{cnbecbnoenoe}
\left|\lsl+b_1\right|+\Sigma_{i=1}^L\left|(b_i)_s+(2i-1+\frac{2}{|\log b_1|})b_1b_i-b_{i+1}\right|\lesssim \|\e\|_{loc}+b_1^{L+\frac32}
\ee where $\|\e\|_{loc}$ measures a {\it local in space} interaction with the harmonic map.\\

 \noindent \emph{(iii). Control of the radiation and monotonicity formula.}
 
 According to \fref{cnbecbnoenoe}, the core of our analysis is now to show that local norms of $\e$ are under control and do not perturb the dynamical system \fref{systdynfundone}. This is achieved using high order Sobolev norms adapted to the linear flow, and in particular we claim that the orthogonality conditions \fref{vnekoenono} ensure the Hardy type coercivity of the iterated operator:
 $$\mathcal E_{2k+2}=\int |H^{k+1}\e|^2\gtrsim \int \frac{|\e|^2}{(1+y^{4k+4})(1+|\log y|^2)}, \ \ 0\leq k\leq L.$$
 We now claim the we can control theses norms thanks to an energy estimate {\it seen on the linearized equation in original variables} that is by working with $w$ in \fref{decomntoro} and not $\e$,  as initiated in \cite{RaphRod}, \cite{MRR}. Here the parabolic structure of the problem simplifies a bit the analysis to display some repulsively property of the renormalized linearized operator, see the proof of Proposition \ref{monoenoiencle}. The outcome is an estimate of the form 
 \be
 \label{vnknvornror}
 \frac{d}{ds}\left\{\frac{\matchal E_{2k+2}}{\l^{4k+2}}\right\}\lesssim \frac{b_1^{2k+3}}{\l^{4k+2}}|\log b_1|^{c_k}
 \ee where the right hand side is controlled by the size of the error in the construction of the approximate blow up profile. Integrating this in time yields two contributions, one from data and one from the error:
 $$\matchal E_{2k+2}(s)\lesssim \l^{4k+2}(s)\mathcal E_{2k+2}(0)+\l^{4k+2}(s)\int_{s_0}^s\frac{b_1^{2k+3}}{\l^{4k+2}}|\log b_1|^{c_k}d\sigma.$$ The second contribution is estimated in the regime \fref{odelrejropw} using the fundamental algebra: 
 \be
\label{comprsion}
(2k+3)-c_1(4k+2)=1+\frac{2(L-k-1)}{2L-1}\left\{\begin{array}{ll}\geq 1 \ \ \mbox{for}\ \ k\leq L-1,\\ <1\ \ \mbox{for}\ \ k=L.\end{array}\right.
\ee
Hence data dominates for $k\leq L-1$ up to a logarithmic error: 
$$\l^{4k+2}(s)\int_{s_0}^s\frac{b_1^{2k+3}}{\l^{4k+2}}|\log b_1|^{c_k}d\sigma\sim  \l^{4k+2}(\log s)^C\int_{s_0}^s\frac{d\sigma}{\sigma^{2k+3-c_1(4k+2)}}\sim \l^{4k+2}(\log s)^C$$ which yields the bound
\be
\label{lossylog}
\mathcal E_{2k+2}\lesssim \l^{4k+2} |\log s|^C, \ \ 0\leq k\leq L-1
\ee
which simply expresses the boundedness up to a log of $w$ in some Sobolev type $H^{k+1}$ norm. On the contrary, for $k=L$, first of all we can derive a {\it sharp logarithmic gain} in \fref{vnknvornror}:
\be
\label{neneneone}
 \frac{d}{ds}\left\{\frac{\matchal E_{2L+2}}{\l^{4k+2}}\right\}\lesssim \frac{b_1^{2L+3}}{\l^{4L+2}|\log b_1|^2}
 \ee and then the integral diverges from \fref{comprsion} and:
 $$\l^{4k+2}(s)\int_{s_0}^s\frac{b_1^{2L+3}}{\l^{4L+2}|\log b_1|^2}d\sigma\sim \l^{4k+2}(s)\int_{s_0}^s\frac{1}{\sigma}\frac{b_1^{2L+2}}{\l^{4L+2}|\log b_1|^2}d\sigma\sim \frac{b_1^{2L+2}}{|\log b_1|^2}\gg \l^{4k+2}.$$ We therefore obtain 
 \be
 \label{cnekoncnenoe}
 \mathcal E_{2L+2}\lesssim \frac{b_1^{2L+2}}{|\log b_1|^2}.
 \ee 
 The difference between the controls \fref{lossylog} for $0\le k\leq L-1$ and the sharp control \fref{cnekoncnenoe} is an essential feature of the analysis and explains the introduction of an exactly order $L+1$ Sobolev energy.\\ We can now reinject this bound into \fref{cnbecbnoenoe} and show thanks to the logarithmic gain in \fref{neneneone} that $\e$ does not perturb the system \fref{systdynfundintro}, modulo the control of the associated unstable $L-1$ modes by a further adjusted choice of the initial data. This concludes the proof of Theorem \ref{thmmain}.\\

This paper is organized as follows. In section \ref{sectiontwo}, we construct the aprroximate self similar solutions $Q_b$ and obtain sharp estimates on the error term $\Psi_b$. We also exhibit an explicit solution to the dynamical system \fref{systdynfundintro} and show that it displays $(L-1)$ directions of instability. In section \ref{sectionthree}, we set up the bootstrap argument, Proposition \ref{bootstrap}, and derive the fundamental monotonicity of the Sobolev type norm $\|H^{L+1}\e\|_{L^2}^2$, Proposition \ref{AEI2}, which is the heart of the analysis. In section \ref{sectionfour}, we close the bootstrap bounds which easily imply the blow up statement of Theorem \ref{thmmain}.


\section{Construction of the approximate profile}
\label{sectiontwo}

This section is devoted to the construction of the approximate $Q_b$ blow up profile and the study of the associated dynamical system for $b=(b_1,\dots,b_L)$.


\subsection{The linearized Hamiltonian}


Let us start with recalling the structure of the harmonic map $Q$ which is the unique -up to scaling- solution to 
\be
\label{harmonicmapequation} 
\Lambda Q=g(Q), \ \ Q(0)=0, \ \ \lim_{r\to +\infty}Q(r)=\pi.
\ee
This equation can be integrated explicitely\footnote{see \cite{RSc1} for more details}. $Q$ is smooth $Q\in \mathcal C^{\infty}([0,+\infty),[0,\pi))$ and admits using \fref{assumtiong} a Taylor expansion\footnote{up to scaling} to all order at the origin:
\be
\label{origin}
Q(y)=\Sigma_{i=0}^pc_{i}y^{2i+1}+O(y^{2p+3})\ \ \mbox{as}\ \ y\to 0,
\ee 
and at infinity:
\be
\label{infinity}
Q(y)=\pi-\frac{2}{y}-\Sigma_{i=1}^p\frac{d_i}{y^{2i+1}}+O\left(\frac{1}{y^{2p+3}}\right)\ \ \mbox{as}\ \ y\to +\infty.
\ee
The linearized operator close to $Q$ displays a remarkable structure. Indeed, let the potentials 
\be
\label{defpotential}
Z=g'(Q), \ \ V=Z^2+\Lambda Z=f'(Q), \ \ \widetilde{V}=(1+Z)^2-\Lambda Z,
\ee
which satisfy from \fref{origin},\fref{infinity} the following behavior at $0,+\infty$: 
\be
\label{comportementz} 
Z(y)=\left\{\begin{array}{ll}1+\Sigma_{i=1}^pc_iy^{2i}+O(y^{2p+2})\ \ \mbox{as}\ \ y\to 0,\\
 				-1+\Sigma_{i=1}^p\frac{c_i}{y^{2i}}+O\left(\frac{1}{y^{2p+2}}\right) \ \ \mbox{as}\ \  y\to+\infty,\end{array}\right.
\ee
\be
\label{comportementv} 
V(y)=\left\{\begin{array}{ll}1+\Sigma_{i=1}^pc_iy^{2i}+O(y^{2p+2})\ \ \mbox{as}\ \ y\to 0,\\
 				1+\Sigma_{i=1}^p\frac{c_i}{y^{2i}}+O\left(\frac{1}{y^{2p+2}}\right) \ \ \mbox{as}\ \  y\to+\infty,\end{array}\right.
\ee
\be
\label{comportementvtilde} 
\widetilde{V}(y)=\left\{\begin{array}{ll}4+\Sigma_{i=1}^pc_iy^{2i}+O(y^{2p+2})\ \ \mbox{as}\ \ y\to 0,\\
 				\Sigma_{i=1}^p\frac{c_i}{y^{2i}}+O\left(\frac{1}{y^{2p+2}}\right) \ \ \mbox{as}\ \  y\to+\infty,\end{array}\right.
\ee
where $(c_i)_{i\geq 1}$ stands for some generic sequence of constants which depend on the Taylor expansion of $g$ at $(0,\pi)$. The linearized operator close to $Q$ is the Schr\"odinger operator:
\be
\label{defh}
H=-\Delta +\frac{V}{y^2}.
\ee and admits the factorization:
\be
\label{facotrization}
H=A^*A
\ee with $$A= -\partial_y + \frac{Z}{y}, \ \ \ A^*= \partial_y + \frac{1+Z}{y},\ \ Z(y)=g'(Q).$$ Observe that equivalently:
\be
\label{otherformula}
Au=-\Lambda Q\frac{\partial}{\partial y}\left(\frac{u}{\Lambda Q}\right), \ \ A^*u=\frac{1}{y\Lambda Q}\frac{\partial}{\pa y}\left(uy\Lambda Q\right)
\ee
and thus the kernels of $A$ and $A^*$ on $\Bbb R^*_+$ are explicit: 
\be
\label{efinitei}
Au=0\ \ \mbox{iff}\ \  u\in \mbox{Span}(\Lambda Q),\ \ A^*u=0\ \ \mbox{iff}\ \ u\in \mbox{Span}\left(\frac{1}{y\Lambda Q}\right).
\ee
Hence the kernel of $H$ on $\Bbb R^*_+$ is: 
\be
\label{kernelh}
Hu=0\ \ \mbox{iff}\ \ u\in \mbox{Span}(\Lambda Q,\Gamma)
\ee 
with 
\be
\label{Gamma} \Gamma(y)=\Lambda \phi\int_1^y\frac{dx}{x(\Lambda\phi(x))^2}=\left\{\begin{array}{ll} O(\frac1y) \ \ \mbox{as} \ \ y\to 0,\\ \frac{y}{4}+O\left(\frac{\log y}{y}\right)\ \  \mbox{as} \ \ y\to +\infty.\end{array}\right . 
\ee
In particular, $H$ is a positive operator on $\dot{H}^1_{rad}$ with a {\it resonnance} $\Lambda Q$ at the origin induced by the energy critical scaling invariance. We also introduce the conjuguate Hamiltonian
\be
\label{computationhtilde}
\tilde{H}=AA^*=-\Delta +\frac{\widetilde{V}}{y^2}
\ee
which is definite positive by construction and \fref{efinitei}, see Lemma \ref{coerchtilde}.


\subsection{Admissible functions}


The explicit knowledge of the Green's functions allows us to introduce the formal inverse
\be
\label{definvesr}
H^{-1}f=-\G(y)\int_0^y f\Lambda Qxdx +\Lambda Q(y)\int_{0}^y f\G xdx.
\ee 
Given a function $f$, we introduce the suitable derivatives of $f$ by considering the sequence:
\be
\label{suitablederivatives}
 f_0=f, \ \ f_{k+1}=\left\{\begin{array}{ll} A^*f_k\ \ \mbox{for k odd}\\ A f_k\ \ \mbox{for k even}\end{array}\right ., \ \ k\geq 0.
 \ee 
 We shall introduce the formal notation $$f_k=\matchal A^kf.$$ We define a first class of admissible functions which display a suitable behavior both at the origin and infinity:
 
\begin{definition}[Admissible functions]
\label{defadmissible}
We say a smooth function $f\in \mathcal C^{\infty}(\Bbb R_+,\Bbb R)$ is admissible of degree $(p_1,p_2)\in \Bbb N \times \Bbb Z$  if:\\
(i) $f$ admits a Taylor expansion at the origin to all order:
\be
\label{assumptoinfsbis}
f(y)=\Sigma_{k=p_1}^pc_ky^{2k+1}+O(y^{2p+3}).
\ee
(ii) $f$ and its suitable derivatives admit a bound for $y\geq 2$:
\be
\label{taylorexpansionoriginbis}
\forall k\geq 0, \ \ \left|f_k(y) \right|\lesssim   \left\{\begin{array}{ll}y^{2p_2-k-1}(1+|\log y|) \ \ \mbox{for}\ \ 2p_2-k\geq 1,\\ y^{2p_2-k-1} \ \ \mbox{for}\ \ 2p_2-k\leq 0\end{array}\right..
\ee

\end{definition}

$H$ naturally acts on the class of admissible functions in the following way:

\begin{lemma}[Action of $H$ and $H^{-1}$ on admissible functions]
\label{lemmapropinverse}
Let $f$ be an admissible function of degree $(p_1,p_2)$, then:\\
(i) $\forall l\geq 1$, $H^lf$ is admissible of degree:
\be
\label{dergeeiterate}
(\max(p_1-l,0),p_2-l).
\ee
(ii) $\forall l,p_2\geq 0$, $H^{-l}f$ is admissible of degree:
\be
\label{knehatof}
(p_1+l,p_2+l).
\ee
\end{lemma}

\begin{proof}[Proof of Lemma \ref{lemmapropinverse}]  This a simple consequence of the expansions \fref{origin}, \fref{infinity}.\\
Let us  first show that $Hf$ is admissible of degree at least $(\max(p_1-1,1),p_2-1)$ which yields \fref{dergeeiterate} by induction. We inject the Taylor expansions \fref{assumptoinfsbis}, \fref{taylorexpansionoriginbis} into \fref{defh}. Near the origin, the claim directly follows from the Taylor expansion \fref{comportementv}  and the cancellation $H(y)=cy+O(y^3)$ at the origin. The claim at infinity directly follows from the relation $u_k=f_{k+2}$ by definition.\\
Let now $p_2\geq 0$ and $u=H^{-1}f$ be given by \fref{definvesr}, and let us show that $u$ is admissible of degree at least $(p_1+1,p_2+1)$ which yields \fref{knehatof} by induction. From the relation $u_k=f_{k-2}$ for $k\geq 2$, we need only consider $k=0,1$. We first observe from the Wronskian relation $\Gamma'(\Lambda Q)-(\Lambda Q)'\Gamma=\frac1y$ that $$A\Gamma=-\Gamma'+\frac{Z}{y}\Gamma=-\Gamma'+\frac{(\Lambda Q)'}{\Lambda Q}\Gamma=-\frac{1}{y\Lambda Q}.$$ We thus compute using the cancellation $A\Lambda Q=0$: 
\be
\label{formulaau}
Au=-A\Gamma\int_0^yf\Lambda Qxdx=\frac{1}{y\Lambda Q}\int_0^yf\Lambda Qxdx.
\ee
Moreover, we may invert $A$ using \fref{otherformula} and the boundary condition $u=O(y^3)$ from \fref{definvesr} which yields: 
\be
\label{inveriosnu}
u=-\Lambda Q\int_0^y\frac{Au}{\Lambda Q}dx=-\Lambda Q(y)\int_0^y\frac{dx}{x(\Lambda Q(x))^2}\int_0^xf(z)\Lambda Q(z)zdz.
\ee
This yields using \fref{origin} the Taylor expansion near the origin:
$$Au=\Sigma_{k=p_1}^pc_k^{(1)}y^{2k+2}+O(y^{2p+4}), \ \ u=-\Lambda Q\int_0^y\frac{Au}{\Lambda Q}dx=\Sigma_{k=p_1}^pc^{(2)}_ky^{2k+3}+O(y^{2p+5})$$ and hence $u$ is of degree at least $p_1+1$ near the origin. For $y\geq 1$, we estimate in brute force from \fref{formulaau}, \fref{inveriosnu}, \fref{taylorexpansionoriginbis} for $p_2\geq 1$:
$$|Au|=|u_1|\lesssim \int_0^y\tau^{2p_2-1}(1+|\log \tau|)d\tau\lesssim y^{2p_2}(1+|\log y|),$$
$$|u|\lesssim \frac{1}{y}\int_0^y \tau^{2p_2}(1+|\log \tau|)\tau d\tau\lesssim y^{2p_2+1}(1+|\log y|),$$
and for $p_2=0$:
$$|Au|=|u_1|\lesssim \int_1^y\tau^{-1}d\tau\lesssim 1+|\log y|,$$
$$|u|\lesssim \frac{1}{y}\int_0^y (1+|\log \tau|)\tau d\tau\lesssim y(1+|\log y|).$$
Hence $u$ satisfies \fref{taylorexpansionoriginbis} with $p_2\to p_2+1$ and $k=0,1$.
\end{proof}

Let us give an explicit example of admissible functions which will be essential for the analysis. From \fref{origin} and the cancellation $A\Lambda Q=0$, $\Lambda Q$ is admissible of degree $(0,0)$, and hence Lemma \ref{lemmapropinverse} ensures:

\begin{lemma}[Generators of the kernel of $H^i$]
\label{lemmaradiation}
Let the sequence of profiles for $i\geq 1$
\be
\label{deftk}
T_i=(-1)^{i}H^{-i}\Lambda Q,
\ee
then $T_i$ is admissible of degree $(i,i)$. 
\end{lemma}


\subsection{$b_1$ admissible functions}


We will need an extended notion of admissible function for the construction of the blow up profile. In the sequel, we consider a small enough $0<b_1\ll1$ and let $B_0, \chi_{B_0}$ be given by \fref{defbnot}, \fref{defchib}. Given $l\in \Bbb Z$, we let:
\be
\label{defgby}
g_l(b_1,y)=\left\{\begin{array}{ll}\frac{1+|\log (\sqrt{b_1}y)|}{|\log b_1|}{\bf 1}_{y\leq 3B_0}\ \ \mbox{for}\ \ l\geq 1\\ \frac{{\bf 1}_{y\leq 3B_0}}{|\log b_1|}\ \ \mbox{for}\ \ l\leq 0\end{array}\right.
\ee
and similarily:
\be
\label{defgbybis}
\tilde{g}_l(b_1,y)=\left\{\begin{array}{ll}\frac{1+|\log y|}{|\log b_1|}{\bf 1}_{y\leq 3B_0}\ \ \mbox{for}\ \ l\geq 1\\ \frac{{\bf 1}_{y\leq 3B_0}}{|\log b_1|}\ \ \mbox{for}\ \ l\leq 0\end{array}\right..
\ee

We then define the extended class of $b_1$ admissible functions:

\begin{definition}[$b_1$ admissible functions]
\label{defadmissiblebone}
We say a smooth function $f\in \mathcal C^{\infty}(\Bbb R^*_+\times \Bbb R_+,\Bbb R)$ is  $b_1$-admissible of degree $(p_1,p_2)\in \Bbb N \times \Bbb Z$  if:\\
(i) For $y\leq 1$, $f$ admits a representation 
\be
\label{separationvariable}
f(b_1,y)=\Sigma_{j=1}^Jh_j(b_1)\tilde{f}_j(y)
\ee 
for some finite order $J\in \Bbb N^*$, some smooth functions $\tilde{f}_j(y)$ with a Taylor expansion at the origin to all order: $\forall y\le 1$,
\be
\label{assumptoinfs}
\tilde{f}_j(y)=\Sigma_{k=p_1}^pc_{k,j}y^{2k+1}+O(y^{2p+3})
\ee
and some smooth functions $h_j(b_1)$ away from the origin with 
\be
\label{boudnorigin}
\forall l\geq 0, \ \ \left|\frac{\partial^l h_j}{\partial b_1^l}\right|\lesssim \frac{1}{b_1^l}.
\ee
(ii) $f$ and its suitable derivatives \fref{suitablederivatives} satisfy a uniform bound for some constant $c_{p_2}>0$: $\forall y\geq 2$, $\forall k\geq0$:
\bea
\label{boundinfitybis}
\nonumber |f_k(b_1,y)|& \lesssim &  y^{2p_2-k-1}g_{2p_2-k}(b_1,y)+y^{2p_2-k-3}|\log y|^{c_{p_2}}\\
& + & F_{p_2,k,0}(b_1)y^{2p_2-k-3}{\bf 1}_{y\geq 3B_0},
\eea
and $\forall l\geq 1$:
\bea
\label{boundinfity}
\nonumber \left|\frac{\pa^{l}}{\pa b_1^l}f_k(b_1,y)\right|&\lesssim &\frac{1}{b^l_1|\log b_1|}\left\{  y^{2p_2-k-1}\tilde{g}_{2p_2-k}(b_1,y)+y^{2p_2-k-3}|\log y|^{c_{p_2}}\right\}\\
 & + & F_{p_2,k,l}(b_1) y^{2p_2-k-3}{\bf 1}_{y\geq 3B_0}
\eea
with
 \be
\label{defffbone}
\forall l\geq 0, \ \  F_{p_2,k,l}(b_1)=\left\{\begin{array}{ll}0\ \mbox{for}\ \ 2p_2-k-3\leq -1,\\ \frac{1}{b_1^{l+1}|\log b_1|}\ \ \mbox{for}\ \ 2p_2-k-3\geq 0\end{array}\right.
\ee
\end{definition}

\begin{remark} Let us consider the solution $T_1$ to $$HT_1=-\Lambda Q,$$ then an explicit computation reveals the growth for $y$ large $$\Lambda Q\sim \frac 1y, \ \ T_1(y)\sim y \log y.$$ The $b_1$ admissibility corresponds to a $\log b_1$ gain on the growth at $\infty$ which is an essential feature of the slowly growing tails in the construction of the modulated blow up profile in Proposition \ref{consprofapproch}. Observe for example that \fref{boundinfitybis}, \fref{defffbone} imply the rough bound:
\be
\label{roughbound}
|f_k|\lesssim (1+y)^{2p_2-1-k}, \ \ \left|\frac{\pa^l f_k}{\pa b_1^l}\right|\lesssim \frac{(1+y)^{2p_2-1-k}}{|\log b_1|}, \ \ k\geq 0,\ \ l\geq 1, 
\ee 
and hence a logarithmic improvement with respect to \fref{taylorexpansionoriginbis}. This gain will be measured in a sharp way through the computation of suitable weighted Sobolev bounds, see  Lemma \ref{lemmaestimate}. 
\end{remark}

We claim that $H,H^{-1}$ and the scaling operators naturally act on the class of $b_1$ admissible functions in the following way:

\begin{lemma}[Action of $H$,$H^{-1}$ and scaling operators on $b_1$-admissible functions]
\label{lemmapropinversebis}
Let $f$ be a $b_1$-admissible function of degree $(p_1,p_2)$, then:\\
(i) $\forall l\geq 1$, $H^lf$ is $b_1$-admissible of degree:
\be
\label{dergeeiteratebis}
(\max(p_1-l,0),p_2-l).
\ee
(ii) $\forall l,p_2\geq 1$, $H^{-l}f$ is $b_1$-admissible of degree:
\be
\label{knehatofbis}
(p_1+l,p_2+l).
\ee
(iii) $\Lambda f=y\pa_y f$ is admissible of degree $(p_1,p_2)$.\\
(iv) $b_1\frac{\pa f}{\pa b_1}$ is admissible of degree $(p_1,p_2)$.
\end{lemma}

\begin{proof}[Proof of Lemma \ref{lemmapropinversebis}]

{\bf step 1} Proof of (i).  Let us show that $u=Hf$ is $b_1$-admissible of degree $(\max(p_1-1,0),p_2-1)$ which yields \fref{dergeeiteratebis} by induction. Near the origin, the claim directly follows from the Taylor expansion \fref{assumptoinfs} with \fref{separationvariable} and the cancellation $H(y)=cy+O(y^3)$ at the origin. For $y\geq 1$, $H$ is independent of $b_1$ so that by definition 
$$\forall l\geq 0, \ \ \frac{\pa^l u_k}{\pa b_1^l} = \frac{\pa^lf_{k+2}}{\pa b_1^l}$$
which satisfies \fref{boundinfitybis}, \fref{boundinfity}, \fref{defffbone} with $p_2\to p_2-1$ and $F_{p_2-1,k,l}(b_1)=F_{p_2,k+2,l}(b_1)$, and \fref{dergeeiteratebis} follows.\\

{\bf step 2} Proof of (ii). Let now $p_2\geq 1$ and let us show that $u=H^{-1}f$ is admissible of degree $(p_1+1,p_2+1)$ which yields \fref{knehatofbis} by induction. Observe that for $k\geq 2$, $$\forall l\geq 0, \ \ \frac{\pa^l u_k}{\pa b_1^l} = \frac{\pa^lf_{k-2}}{\pa b_1^l}$$ which satisfies \fref{boundinfitybis}, \fref{boundinfity}, \fref{defffbone} with $p_2\to p_2+1$ and $F_{p_2+1,k,l}(b_1)=F_{p_2,k-2,l}(b_1)$. It thus only remains to estimate $u,Au$ and their derivatives in $b_1$. \\
{\it Estimate for $u$ near the origin}: The inversion formulas \fref{formulaau}, \fref{inveriosnu}  ensure the decomposition of variables near the origin $$u(b_1,y)=\Sigma_{j=1}^Jh_j(b_1)\tilde{u}_j(y)$$ with using \fref{origin} the Taylor expansion near the origin:
$$A\tilde{u}_j=\Sigma_{k=p_1}^pc_{k,j}^{(1)}y^{2k+2}+O(y^{2p+4}), \ \ \tilde{u}_j=-\Lambda Q\int_0^y\frac{A\tilde{u}_j}{\Lambda Q}dx=\Sigma_{k=p_1}^pc^{(2)}_{k,j}y^{2k+3}+O(y^{2p+5})$$ and hence $u$ is of degree at least $p_1+1$ near the origin.\\
{\it Estimate for $u_1=Au$ for $y\geq 1$}: We use the formula \fref{formulaau} and the assumption $p_2\geq 1$ to estimate for $1\leq y\leq 3B_0$:
\bee
|Au|& \lesssim & \int_0^y |f|d\tau\lesssim \int_0^y  \left[\tau^{2p_2-1}g_{2p_2-1}(b_1,\tau)+\tau^{2p_2-3}|\log \tau|^{c_{p_2}}\right]d\tau\\
& \lesssim & \frac{1}{b_1^{p_2}|\log b_1|}\int_{0}^{\sqrt{b_1}y}\sigma^{2p_2-1}(1+|\log \sigma|)d\sigma+O(y^{2p_2-2}|\log y|^{1+c_{p_2}})\\
& \lesssim & y^{2p_2}\frac{1+|\log(\sqrt{b_1 y})|}{|\log b_1|}+y^{2p_2-2}|\log y|^{1+c_{p_2}}\\
& = & y^{2(p_2+1)-2}g_{2(p_2+1)-1}(b_1,y)+y^{2(p_2+1)-4}|\log y|^{1+c_{p_2}}
\eee
and for $y\geq 3B_0$:
\bee
|Au|& \lesssim & \int_0^y |f|d\tau\lesssim \int_0^{3B_0}\tau^{2p_2-1}g_1(b_1,\tau)d\tau+\int_{3B_0}^yF_{p_2,0,0}(b_1)\tau^{2p_2-3}d\tau+O(y^{2p_2-2}|\log y|^{1+c_{p_2}})\\
& \lesssim & \frac{1}{b_1^{p_2}|\log b_1|}+\int_{3B_0}^yF_{p_2,0,0}(b_1)\tau^{2p_2-3}d\tau+O(y^{2p_2-2}|\log y|^{1+c_{p_2}}).
\eee
If $p_2=1$ which is the borderline case $2p_2-3=-1$, then $F_{p_2,0,0}=0$, and we thus get the bound for all $p_2\geq 1$, $y\geq 3B_0$:
\bee
|Au| &\lesssim &y^{2p_2-2}\left(\frac{1}{b_1|\log b_1|}+F_{p_2,0,0}(b_1)\right)+y^{2p_2-2}|\log y|^{1+c_{p_2}}\\
& \lesssim& \frac 1{b_1|\log b_1|}y^{2p_2-2} +y^{2(p_2+1)-4}|\log y|^{1+c_{p_2}}
\eee
 and \fref{defffbone} is satisfied for $(p_2\to p_2+1, k=1)$  thanks to $2(p_2+1)-1-3\geq 0.$\\
We now pick $l\geq 1$. $H$ is independent of $b_1$ so $$H\left(\frac{\pa^lu}{\pa b_1^l}\right)=\frac{\pa^lf}{\pa b_1^l}$$ and we therefore compute from \fref{formulaau}$$\frac{\pa^lu_1}{\pa b_1^l} =   \frac{1}{y\Lambda Q}\int_0^y\Lambda Q \frac{\pa^lf}{\pa b_1^l}xdx.$$
This yields the bound for $|y|\leq 3B_0$:
\bee
\left|\frac{\pa^lu_1}{\pa b_1^l}\right|& \lesssim & \int_0^y \frac{1}{b^l_1|\log b_1|}\left\{  y^{2p_2-1}\tilde{g}_{2p_2-1}(b_1,y)+y^{2p_2-3}|\log y|^{c_{p_2}}\right\}dy\\
& \lesssim & \frac{1}{b^l_1|\log b_1|}\left[y^{2p_2}\tilde{g}_1(b_1,y)+y^{2p_2-2}|\log y|^{c_{p_2}+1}\right]\\
& = &  \frac{1}{b^l_1|\log b_1|}\left[y^{2(p_2+1)-2}\tilde{g}_{2(p_2+1)-1}(b_1,y)+y^{2(p_2+1)-4}|\log y|^{c_{p_2}+1}\right]
\eee
and for $|y|\geq 3B_0$:
$$
\left|\frac{\pa^lu_1}{\pa b_1^l}\right| \lesssim  \frac{1}{b_1^l|\log b_1|}\left[\frac{1}{b_1^{p_2}}+y^{2p_2-2}|\log y|^{c_{p_2}+1}\right]+\int_{3B_0}^y  F_{p_2,0,l}(b_1)y^{2p_2-3}dy.
$$
Again, if $p_2=1$, $F_{p_2,0,l}=0$ and we therefore obtain the bound for all $p_2\geq 1$:
\bee
\left|\frac{\pa^lu_1}{\pa b_1^l}\right|&\lesssim & y^{2p_2-2}\left[\frac{1}{b_1^{l+1}|\log b_1|}+F_{p_2,0,l}(b_1)\right]+\frac{y^{2p_2-2}|\log y|^{c_{p_2}+1}}{b_1^{l}|\log b_1|}\\
& \lesssim & \frac{y^{2(p_2+1)-1-3}}{b_1^{l+1}|\log b_1|}+\frac{1}{b_1^l|\log b_1|}y^{2(p_2+1)-1-3}|\log y|^{c_{p_2}+1}
\eee
and \fref{defffbone} is satisfied for $(p_2\to p_2+1, k=1)$  thanks to $2(p_2+1)-1-3\geq 0.$\\
{\it Estimate for $u$}: We now estimate from the above bounds and \fref{inveriosnu}: for $1\leq y\leq 3B_0$,
\bee
|u|&\lesssim& \frac{1}{y}\int_0^y|Au|\tau d\tau\lesssim \frac{1}{y}\int_0^y\left[\tau^{2p_2+1}g_{2p_2+1}(b_1,\tau)+\tau^{2p_2-1}|\log \tau|^{1+c_{p_2}}\right]d\tau\\
& \lesssim &  y^{2p_2+1}\frac{1+|\log(\sqrt{b_1 y})|}{|\log b_1|}+y^{2p_2-1}|\log y|^{2+c_{p_2}}\\
& = & y^{2(p_2+1)-1}g_{2(p_2+1)}(b_1,y)+y^{2(p_2+1)-3}|\log y|^{2+c_{p_2}}
\eee
and for $y\geq 3B_0$:
\bee
|u|&\lesssim & \frac{1}{y}\left[\int_0^{3B_0}\tau^{2p_2+1}g(b_1,\tau)d\tau+\int_{3B_0}^yF_{p_2+1,1,0}(b_1)\tau^{2p_2-1}d\tau\right]+y^{2(p_2+1)-3}|\log y|^{2+c_{p_2}}\\
& \lesssim & y^{2p_2-1}\left[\frac{1}{b_1|\log b_1|}\right]+y^{2(p_2+1)-3}|\log y|^{2+c_{p_2}}
\eee
which satisfies \fref{boundinfitybis} for $(p_2\to p_2+1, k=0)$  thanks to $2(p_2+1)-3-1\geq 0$. Finally, for $l\geq 1$, $1\leq y\leq 3B_0$: 
\bee
\left|\frac{\pa^lu}{\pa b_1^l}\right|&\lesssim& \frac{1}{y}\int_0^y\left|\frac{\pa^l u_1}{\pa b_1^l}\right|\tau d\tau\lesssim \frac{1}{yb_1^l|\log b_1|}\int_0^y\left[\tau^{2p_2+1}\tilde{g}_{2p_2+1}(b_1,\tau)+\tau^{2p_2-1}|\log \tau|^{1+c_{p_2}}\right]d\tau\\
& \lesssim & \frac{1}{b^l_1|\log b_1|}\left[y^{2(p_2+1)-1}\tilde{g}_{2(p_2+1)}(b_1,y)+y^{2(p_2+1)-3}|\log y|^{c_{p_2}+1}\right]
\eee
and for $y\geq 3B_0$:
\bee
\left|\frac{\pa^lu}{\pa b_1^l}\right|&\lesssim & \frac{1}{y}\left[\int_0^{3B_0}\frac{\tau^{2p_2+1}\tilde{g}_1(b_1,\tau)}{b_1^l|\log b_1|}d\tau+\int_{3B_0}^yF_{p_2+1,1,l}(b_1)\tau^{2p_2-1}d\tau\right]+\frac{y^{2(p_2+1)-3}|\log y|^{2+c_{p_2}}}{b_1^l|\log b_1|}\\
& \lesssim & \frac{y^{2(p_2+1)-3}}{b_1^{l+1}|\log b_1|}+\frac{y^{2(p_2+1)-3}|\log y|^{2+c_{p_2}}}{b_1^l|\log b_1|}
\eee
and hence $u$ is $b_1$ admissible of degree $(p_1+1,p_2+1)$.\\

{\bf step 3} Proof of (iii), (iv). The property (iv) is a direct consequence of the Definition \ref{defadmissiblebone} of $b_1$ admisssible functions and the trivial bound $$\frac{\tilde{g}_l(b_1,y)}{|\log b_1|}|\lesssim g_l(b_1,y).$$. We now turn to the proof of (iii) and first rewrite the scaling operator as $$\Lambda =y\pa_y=-Id-yA+(1+Z).$$ Near the origin, the existence of the decompositon \fref{separationvariable} follows directly from the even parity of the Taylor expansion of $Z$ at the origin \fref{comportementz}. Far out, let $$\Lambda f=-f-yf_1+(1+Z)f.$$ 
A simple induction argument similar to Lemma \ref{leibnizrule} yields the expansion for $k\geq 1$: 
\be
\label{induction}
(yf_1)_k=c_{k+1}yf_{k+1}+c_{k+2}f_k+\Sigma_{i=1}^{k}P_{k,i}(y)f_i
\ee
with the improved decay:
\be\label{improvedecay}
|\pa_y^lP_{k,i}(y)|\lesssim \frac{1}{1+y^{2+l+k-i}}, \ \ \forall l\geq 0, \ \ y\geq 1.
\ee
We therefore obtain from \fref{roughbound}, \fref{induction}, \fref{improvedecay}, \fref{comportementz} the bound:
\bee
 |(\Lambda f)_k|& \lesssim & |yf_{k+1}|+|f_k|+\Sigma_{i=0}^{k}\frac{1}{y^{2+k-i}} y^{2p_2-i-1}\\
 & \lesssim & y^{2p_2-k-1}(g_{2p_2-(k+1)}+g_{2p_2-k})+y^{2p_2-k-3}|\log y|^{c_{p_2}}\\
 & + & (F_{p_2,k,0}+F_{p_2,k+1,0})y^{2p_2-k-3}{\bf 1}_{y\geq 3B_0}
\eee 
We now observe the monotonicity $g_{2p_2-k-1}\lesssim g_{2p_2-k}$ from \fref{defgby} and $F_{p_2,k+1,0}\lesssim F_{p_2,k,0}$ from \fref{defffbone}, and thus  $(\Lambda f)_k$ satisfies \fref{boundinfity}, \fref{defffbone} for $l=0$. Similarily, for $k\geq0$, $l\geq 1$, we use the bound for $y\gtrsim B_0$:
$$y^{2p_2-k-5}|F_{p_2,i,l}(b_1)|\lesssim \frac{y^{2p_2-k-5}}{b^{l+1}_1|\log b_1|}\lesssim \frac{y^{2p_2-k-3}}{b_1^l|\log b_1|}$$ to estimate:
 \bee
  && \left|\frac{\pa^l(\Lambda f)_k}{\pa b_1^l}\right|\lesssim \left|y\frac{\pa^l f_{k+1}}{\pa b_1^l}\right|+\left|\frac{\pa^l f_k}{\pa b_1^l}\right|\\
  & + &  \Sigma_{i=0}^k\frac{1}{y^{k-i+2}}\left\{\frac{1}{b_1^l|\log b_1|}\left[y^{2p_2-i-1}+y^{2p_2-i-3}|\log y|^{c_{p_2}}\right]+ F_{p_2,i,l}(b_1)y^{2p_2-i-3}{\bf 1}_{y\geq 3B_0}\right\}\\
   & \lesssim & \frac{1}{b_1^l|\log b_1|}\left[y^{2p_2-k-1}\tilde{g}_{2p_2-(k+1)}+\tilde{g}_{2p_2-k})+y^{2p_2-k-3}|\log y|^{c_{p_2}}\right]\\
   & +&  (F_{p_2,k,l}+F_{p_2,k+1,l})y^{2p_2-k-3}{\bf 1}_{y\geq 3B_0}
\eee
and the bounds $\tilde{g}_{2p_2-k-1}\lesssim \tilde{g}_{2p_2-k}$, $F_{p_2,k+1,l}\lesssim F_{p_2,k,l}$ now ensure \fref{boundinfity}, \fref{defffbone} for $l\geq 1$.
\end{proof}


\subsection{Slowly growing tails}


Let us give an example of admissible profiles wich will be central in the construction of the leading order slowly modulated blow up profile. Given $b_1>0$ small enough, we let the radiation be: 
\bea
\label{defsigmab}
\Sigma_{b_1} = H^{-1}\bigg\{-c_{b_1}\chi_{\frac{B_0}{4}}\Lambda Q +d_{b_1} H[(1 - \chi_{B_0}) \Lambda Q]\bigg\}
\eea
with:
\be
\label{cb}
c_{b_1} = \frac{4}{\int \chi_{\frac{B_0}{4}} (\Lambda Q)^2}, \ \ d_{b_1} = c_{b_1}\int_0^{B_0} \chi_{\frac{B_0}{4}} \Lambda Q \G ydy.
\ee

We claim:

\begin{lemma}[Slowly growing tails]
\label{lemmaradiationbis}
Let $(T_i)_{i\geq 1}$ be given by \fref{deftk}, then the sequence of profiles for $i\geq 1$
\be
\label{defthetak}
 \Theta_{i}=\Lambda T_i-(2i-1)T_i-(-1)^{i+1}H^{-i+1}\Sigma_{b_1},
\ee
 is $b_1$-admissible of degree $(i,i)$. 
\end{lemma}

\begin{proof}[Proof of Lemma \ref{lemmaradiationbis}]
{\bf step 1} Structure of $T_1$. Let us consider $T_1=-H^{-1}\Lambda Q$ which is admissible of degree $(1,1)$ from Lemma \ref{lemmaradiation}. For $y\geq 1$, and explicit computation using the expansion \fref{infinity} into  \fref{definvesr} yields:
\be
\label{esttone}
T_1(y)=y\log y+e_0y+O\left(\frac{|\log y|^2}{y}\right), \ \ \Lambda T_1=y\log y+(1+e_0)y+O\left(\frac{|\log y|^2}{y}\right)
\ee
for some universal constant $e_0$, and hence the essential cancellation: 
\be
\label{cancellation}
\Lambda T_1-T_1=y+O\left(\frac{|\log y|^2}{y}\right).
\ee
We now prove that $\Theta_i$ is of order $(i,i)$ by induction on $i$.\\

{\bf step 2} $i=1$. By definition: 
\bea
\label{formulasigmbone}
\nonumber \Sigma_{b_1} & = &  \G(y)\int_0^y c_{b_1}\chi_{\frac{B_0}{4}}(\Lambda Q)^2xdx -\Lambda Q(y)\int_{0}^y c_{b_1}\chi_{\frac{B_0}{4}}\G\Lambda Qxdx \\
& + & d_{b_1} (1 - \chi_{B_0}) \Lambda Q(y)
\eea
and thus by definiton of $c_{b_1},d_{b_1}$ given by \fref{cb}:
\bea
\label{sigmab0}
\Sigma_{b_1} = \left\{\begin{array}{ll} c_{b_1}T_1\ \ \mbox{for} \ \ y \leq \frac{B_0}{4}\\ \\
4\G \ \ \mbox{for} \ \ y \geq 3B_0
\end{array} \right. 
\eea
In particular, $\Sigma_{b_1}$ admits a representation \fref{separationvariable} near the origin with $J=1$, $h_1(b_1)=c_{b_1}$ and $\tilde{f}_1(y)=T_1(y)$, and thus an expansion \fref{assumptoinfs} of order $p_1=1$ from the first step. A direct computation on the formula \fref{cb} yields the bounds: 
\be
\label{cbdb}
c_{b_1}=\frac{2}{|\log b_1|}\left[1+O\left(\frac{1}{|\log b_1|}\right)\right], \ \ |d_{b_1}|\lesssim \frac{1}{b_1|\log b_1|},
\ee
\be
\label{boudnderivtaovecb}
\forall l\geq 1,  \ \ \left|\frac{\pa^l c_{b_1}}{\pa b^l_1}\right|\lesssim \frac{1}{b_1^l|\log b_1|^{2}}, \ \ \left|\frac{\pa^l d_{b_1}}{\pa b^l_1}\right|\lesssim \frac{1}{b_1^{l+1}|\log b_1|}
\ee
which imply \fref{boudnorigin}.\\
 For $y \geq 3B_0$, we estimate from \fref{Gamma}, \fref{sigmab0}:
\be
\label{sigmabinfty}
\Sigma_{b_1}(y) = y + O\left(\frac{\log y}{y}\right)
\ee
and for $2\leq y \leq 3B_0$:
\bea
\label{sigmabmilieu}
\nonumber \Sigma_{b_1}(y) & = &c_{b_1}\left(\frac{y}{4}+O\left(\frac{\log y}{y}\right)\right)\left[\int_0^y \chi_{\frac{B_0}{4}}(\Lambda Q)^2xdx \right]-c_{b_1}\Lambda Q(y)\int_1^{y}O(1)xdx\\
& = & y\frac{\int_0^y\chi_{\frac{B_0}{4}}(\Lambda Q)^2 }{\int \chi_{\frac{B_0}{4}}(\Lambda Q)^2}+O\left(\frac{1+y}{|\log b_1|}\right).
\eea
We thus conclude from \fref{cancellation}, \fref{sigmabmilieu} that for $y\leq 3B_0$: 
\bee
\Theta_1(y)& =& y- y\frac{\int_0^y\chi_{\frac{B_0}{4}}(\Lambda Q)^2 }{\int \chi_{\frac{B_0}{4}}(\Lambda Q)^2}+O\left(\frac{1+y}{|\log b_1|}\right)+O\left(\frac{|\log y|^2}{1+y}\right)\\
& = & O\left(\frac{1+y}{|\log b_1|}(1+|\log (y\sqrt{b_1})|\right)
\eee
which together with the bounds \fref{esttone}, \fref{sigmabinfty} for $y\geq 3B_0$ yields the bound for $y\geq 2$:
\be
\label{estzerothetaone}
|\Theta_1(y)|\lesssim yg_2(b_1,y)+O\left(\frac{|\log y|^2}{y}\right).
\ee
We now compute from \fref{formulaau}, \fref{defsigmab}:
$$A\Sigma_{b_1}=\frac{1}{y\Lambda Q}\int_0^y\Lambda Q\left[-c_{b_1}\chi_{\frac{B_0}{4}}\Lambda Q +d_{b_1} H[(1 - \chi_{B_0}) \Lambda Q]\right]xdx$$ and estimate from \fref{cbdb} for $y\leq 3B_0$:
\bea
\label{estonefjo}
\nonumber A\Sigma_{b_1} & = & -\frac{4}{y\Lambda Q}+\frac{c_{b_1}}{y\Lambda Q}\int_y^{B_0} (\Lambda Q)^2xdx+O\left(\frac{d_{b_1}}{B_0^2}{\bf 1}_{B_0\leq y\leq 3B_0}\right)\\
& = & -2+O\left(g_1(b_1,y)\right)
\eea
and for $y\geq 3B_0$: 
\be
\label{estonefjotwo}
A\Sigma_{b_1}=-\frac{4}{y\Lambda Q}=-2+O\left(\frac{1}{y^2}\right).
\ee
Moreover, a simple rescaling argument yields the formula: $$A(\Lambda u)=Au+\Lambda Au-\frac{\Lambda Z}{y}u$$ and thus using \fref{esttone}, \fref{comportementz} : $$A(\Lambda T_1-T_1)=\Lambda AT_1-\frac{\Lambda Z}{y}T_1=\Lambda AT_1+O\left(\frac{\log y}{y^2}\right).$$ We now estimate from \fref{formulaau}, \fref{infinity}:
\bee
AT_1= -\left[\frac{1}{y\Lambda Q}\int_0^y(\Lambda Q)^2xdx\right]=-2\log y+O\left(\frac{\log y}{y^2}\right)
\eee
and similarily: 
$$
\Lambda AT_1=-2+O\left(\frac{\log y}{y^2}\right)
$$
from which 
\be
\label{estonefjothree} 
A(\Lambda T_1-T_1)=-2+O\left(\frac{\log y}{y^2}\right).
\ee
We thus conclude from \fref{estzerothetaone}, \fref{estonefjo}, \fref{estonefjotwo}, \fref{estonefjothree} that $$|A\Theta_1|\lesssim g_1(b_1,y)+O\left(\frac{\log y}{y^2}\right).$$ We now turn to the control of $H\Theta_{1}$. We first compute from a simple rescaling argument:
\be
\label{hresacledlambda}
H(\Lambda u)=2Hu+\Lambda Hu-\frac{\Lambda V}{y^2}u
\ee
which implies:
$$H(\Lambda T_1-T_1)=-\Lambda Q-\Lambda ^2Q+O\left(\frac{\log y}{y^3}\right)=O\left(\frac{\log y}{y^3}\right).$$ Hence the desired cancellation according to \fref{defgby}:
$$|H\Theta_1|\lesssim |H(\Lambda T_1-T_1)|+|H\Sigma_{b_1}|\lesssim \frac{1}{(1+y)|\log b_1|}{\bf 1}_{y\leq 3B_0}+O\left(\frac{\log y}{y^3}\right).$$ The control of higher order suitable derivatives in $y$ now follows by iteration using \fref{infinity}, \fref{comportementv}. Hence $\Theta_1$ satisfies the bound \fref{boundinfitybis} with $p_2=1$, $l=0$.\\
We now take derivatives in $b_1$ in which case from \fref{formulasigmbone} for $l\geq 1$:
\bee
\frac{\pa^l \Theta_1}{\pa^l b_1} & = & -\frac{\partial^l \Sigma_{b_1}}{\partial b^l_1}=  \G(y)\int_0^y \frac{\pa^l}{\pa b^l_1}\left\{c_{b_1}\chi_{\frac{B_0}{4}}\right\}(\Lambda Q)^2xdx \\
& - & \Lambda Q(y)\int_{0}^y \frac{\pa^l}{\pa b^l_1}\left\{c_{b_1}\chi_{\frac{B_0}{4}}\right\}\G\Lambda Qxdx+  \frac{\pa^l}{\pa b^l_1}\left\{d_{b_1} (1 - \chi_{B_0})\right\} \Lambda Q(y)
\eee
and from \fref{sigmab0}: $$\frac{\pa^l \Theta_1}{\pa^l b_1} =-\frac{\pa^l \Sigma_{b_1}}{\pa b^l_1}(y)=0 \ \ \mbox{for}\ \ y\geq 3B_0.$$
We estimate in brute force from \fref{defbnot}: $$ \left|\frac{\pa^l \chi_{B_0}}{\pa b^l_1}\right|\lesssim\frac{{\bf 1}_{B_0\leq y\leq 2B_0}}{b_1^l}$$ and thus obtain from the Leibniz rule and \fref{boudnderivtaovecb} for $y\leq 3B_0$:
\bee
\left|\frac{\pa^l \Theta_1}{\pa b^l_1}\right|& \lesssim & \frac{y}{b_1^l|\log b_1|^2}\left(1+|\log y|\right) +  \left[\Sigma_{k=1}^l\frac{1}{b_1^{l-k}b_1^k|\log b_1|^2}\right]y{\bf 1}_{\frac{B_0}{2}\leq y\leq 3B_0}\\
& + & \left[\Sigma_{k=0}^l\frac{1}{b_1^{l-k}b_1^{k+1}|\log b_1|}\right]\frac{{\bf 1}_{\frac{B_0}{2}\leq y\leq 3B_0}}{y}\\
& \lesssim & \frac{y(1+|\log y|)}{b_1^l|\log b_1|^2}\lesssim \frac{y\tilde{g}_1}{b^l_1|\log b_1|}.
\eee
The control of higher suitable derivatives  $\left(\frac{\pa^l \mathcal A^k\Theta_1}{\pa b^l_1}\right)_{l,k\geq 1}$ follows similarly using the explicit formula \fref{defsigmab}.
This concludes the proof of the estimate \fref{boundinfity} with $p_2=1$, and thus $\Theta_1$ is $b_1$-admissible of degree $(1,1)$.\\

{\bf step 3} $i\to i+1$. We assume the claim for $\Theta_i$ and prove it for $\Theta_{i+1}$. From \fref{deftk}, \fref{defthetak}, \fref{hresacledlambda}:
\bee
H\Theta_{i+1} & = & H(\Lambda T_{i+1})-(2i+1)HT_{i+1}-(-1)^{i}H^{-i+1}\Sigma_{b_1}\\
& = & \Lambda HT_{i+1}-(2i-1)HT_{i+1}+(-1)^{i+1}H^{-i+1}\Sigma_{b_1}-\frac{\Lambda V}{y^2}T_{i+1}\\
& = & -\left[(\Lambda T_i-(2i-1)T_i-(-1)^{i+1}H^{-i+1}\Sigma_{b_1}\right]-\frac{\Lambda V}{y^2}T_{i+1}\\
& = & -\Theta_i-\frac{\Lambda V}{y^2}T_{i+1}.
\eee
The induction hypothesis ensures that $\Theta_i$ is $b_1$-admissible of order $(i,i)$. Moreoever, near the origin, $T_{i+1}$ is from Lemma \ref{lemmaradiation} of degree $i+1$ and hence the development  \fref{comportementv}  ensures that  $\frac{\Lambda V}{y^2}T_{i+1}$ is of degree $i+1$ near the origin. For $y\geq 1$, \fref{comportementv} ensures the improved bound $$\left|\frac{\pa^p}{\pa y^p}\left(\frac{\Lambda V}{y^2}\right)\right|\lesssim \frac{1}{y^{p+4}},\ \ p\geq 0$$  and since $T_{i+1}$ is of degree $i+1$, we obtain from Leibniz rule the rough bound: $\forall k\geq 0$, $$\left|\mathcal A^k\left[\frac{\Lambda V}{y^2}T_{i+1}\right]\right|\lesssim \Sigma_{p=0}^k\frac{1}{y^{k-p+4}}y^{2(i+1)-p-1}|\log y|^{c_i}\lesssim y^{2i-k-3}|\log y|^{c_i},$$ and hence $\frac{\Lambda V}{y^2}T_{i+1}$, which is independent of $b_1$, satisfies \fref{boundinfitybis} and is $b_1$ admissible of degree $(i,i)$.  We conclude from Lemma \ref{lemmapropinversebis} that $\Theta_{i+1}$ is admissible of order $(i+1,i+1)$.
 \end{proof}
 

\subsection{Sobolev bounds on $b_1$ admissible functions}


The property of $b_1$-admissibility leads to simple Sobolev bounds with sharp logartihmic gains. We let $B_1$ be given by \fref{defbnot}.

\begin{lemma}[Estimate of $b_1$ admissible function]
\label{lemmaestimate}
Let $i\geq 1$ and $f$ be a $b_1$-admissible function of degree $(i,i)$. Then:
\be
\label{estunhk}
\int_{y\leq 2B_1}|H^k f|^2\lesssim\frac{|\log b_1|^{4(i-k-1)}}{b_1^{2(i-k)}|\log b_1|^2}\ \ \mbox{for} \ \ 0\leq k\leq i-1,
\ee 
\be
\label{esthooeh}
\int_{y\leq 2B_1}|H^k f|^2\lesssim \ \ 1\ \ \mbox{for} \ \ k\geq i.
\ee
\be
\label{estwitghte}
\int_{y\leq 2B_1}\frac{1+|\log y|^2}{1+y^4}|H^{k}f|^2+\int_{y\leq 2B_1}\frac{1+|\log y|^2}{1+y^2}|AH^{k}f|^2\lesssim |\log b_1|^3\ \ \mbox{for}\ \ k\geq i-1.
\ee
\end{lemma}

\begin{remark} The boundedness of the Sobolev norm \fref{esthooeh} in the borderline case $k=i$ is a consequence of the definition \fref{defgby}. Indeed, $$\int_{y\leq 3B_0}\left|\frac{1+|\log\sqrt{b_1 y}|}{(1+y)|\log b_1|}\right|^2\sim |\log b_1|$$ but
\be
\label{eoieoheoh}
 \int_{y\leq 3B_0}\left|\frac{1}{(1+y)|\log b_1|}\right|^2\lesssim 1.
\ee
\end{remark}

\begin{proof}[Proof of Lemma \ref{lemmaestimate}]. Let $k\geq 0$. Near the origin, the cancelation $A(y)=y^2+O(y)$ and the Taylor expansion \fref{assumptoinfs} ensure that $H^kf$ is bounded uniformly in $y\leq 1$, $|b_1|\le \frac12$. For $y\geq 1$, we estimate from \fref{boundinfitybis}:
\bee
\int_{y\leq 2B_1}|H^k f|^2 & =&  \int|f_{2k}|^2\lesssim \int_{3B_0\leq y\leq 2B_1}|F_{i,2k,0}(b_1)y^{2i-2k-3}{\bf 1}_{y\geq 3B_0}|^2\\
& + & \int_{1\leq y\leq 2B_1}  \left|y^{2i-2k-1}g_{2i-2k}(b_,y)+y^{2i-2k-3}|\log y|^{c_{i}}\right|^2
\eee
For $k\geq i$, $F_{i,2k,0}=0$ and we estimate from \fref{defgby}, \fref{eoieoheoh}:
$$
\int_{y\leq 2B_1}|H^k f|^2 \lesssim 1+ \int_{1\leq y\leq 2B_1}\left|y^{2(i-k)-1}\frac{{\bf 1}_{y\leq 3B_0}}{|\log b_1|}+y^{2i-2k-3}|\log y|^{c_i}\right|^2\lesssim 1.$$
 For $k\leq i-1$, the growth can be controlled in a sharp way. Indeed, we estimate -using $F_{i,2k,0}=0$ for $k=i-1$ precisely to avoid an additional logarithmic error-:
\bee
\int_{y\leq 2B_1}|H^k f|^2& \lesssim& \frac{B_1^{4i-4k-4}}{b_1^2|\log b_1|^2}\\
& + &\frac{1}{|\log b_1|^2}\int_{y\leq 3B_0}  y^{4(i-k)-2}(1+|\log \sqrt{b_1}y|^2)+B_1^{4(i-k)-4}|\log b_1|^{2c_{p_2}+1}\\
& \lesssim & 1+\frac{B_0^{4(i-k)}}{|\log b_1|^2}+\frac{|\log b_1|^{4(i-k-1)}}{b_1^{2(i-k)}|\log b_1|^2}\lesssim \frac{|\log b_1|^{4(i-k-1)}}{b_1^{2(i-k)}|\log b_1|^2}.
\eee
Finally, for $k\geq i-1$, we estimate using the rough bound \fref{roughbound}:
\bee
 &&\int_{y\leq 2B_1}\frac{1+|\log y|^2}{1+y^4}|H^{k}f|^2+\int_{y\leq 2B_1}\frac{1+|\log y|^2}{1+y^2}|AH^{k}f|^2\\
 & \lesssim &\int_{y\leq 2B_1}\frac{1+|\log y|^2}{1+y^4}\left|(1+y)^{2i-2k-1}\right|^2+\int_{y\leq 2B_1}\frac{1+|\log y|^2}{1+y^2}\left|(1+y)^{2i-2(k+1)-1}\right|^2\\
& \lesssim &|\log b_1|^3.
\eee
\end{proof}


\subsection{Slowly modulated blow up profiles}


We proceed in this section to the construction of the approximate modulated blow up profile. Let us start with introducing the notion of homogeneous admissible functions:

\begin{definition}[Homogeneous functions]
\label{defadmissiblehomog}
Given parameters $b=(b_k)_{1\leq k\leq L}$ and $(p_1,p_2,p_3)\in \Bbb N\times\Bbb Z\times \Bbb N$, we say a function $S(b,y)$ is homogeneous of degree $(p_1,p_2,p_3)$ if of the form $$S(b,y)=\Sigma_{J=(j_1,\dots, j_L),\ \ |J|_2=p_3}\left[c_J \left(\Pi_{k=1}^Lb^{j_k}_{k}\right)\tilde{S}_J(b_1,y)\right]$$ where $$J=(j_1,\dots,j_L)\in \Bbb Z\times \Bbb N^{L-1}, \ \ |J|_2=\Sigma_{k=1}^Lkj_k,$$ and for some $b_1$-admissible profiles $\tilde{S}_J$ of degree $(p_1,p_2)$ in the sense of Definition \ref{defadmissiblebone}. We note: $$deg(S)=(p_1,p_2,p_3).$$
\end{definition}

\begin{remark}
\label{remarkfun}
We allow for negative powers of $b_1$ only in the above definition. This ensures from Lemma \ref{lemmapropinversebis} that the space of homogeneous functions of a given degree is stable by application of the operator $ b_1\frac{\pa }{\pa b_1}$. It is also stable by multiplication by $c_{b_1}$ from \fref{cbdb}, \fref{boudnderivtaovecb}.
\end{remark}

We may now proceed to the construction of the slowly modulated blow up profiles.

\begin{proposition}[Construction of the approximate profile]
\label{consprofapproch}
Let $M>0$ be a large enough universal constant, then there exists a small enough universal constant $b^*(M)>0$ such that the following holds true. Let a $\mathcal C^1$ map $$b=(b_k)_{1\leq k\leq L}:[s_0,s_1]\mapsto (-b^*(M),b^*(M))^L $$ with a priori bounds on $[s_0,s_1]$: 
\be
\label{aprioirbound}
0<b_1<b^*(M), \ \ |b_k|\lesssim b_1^k\ \ \mbox{for}\ \ 2\leq k\leq L.
\ee
Let $B_1$ be given by \fref{defbnot} and $(T_i)_{1\leq i\leq L}$ be given by \fref{deftk}. Then there exist homogeneous profiles $$\left\{\begin{array}{l}S_i(b,y), \ \ 2\leq i\leq L+2\\ S_1=0
\end{array}\right .$$ with 
\be
\label{propprofiles}
\left\{\begin{array}{ll}deg(S_i)= (i,i,i)\\ \frac{\partial S_i}{\partial b_j}=0\ \ \mbox{for}\ \ 2\leq i\leq j\leq L\end{array}\right .
\ee  
such that
\be
\label{decompqb}
Q_{b(s)}(y)= Q(y) + \alpha_{b(s)}(y), \ \ \alpha_b(y)=\Sigma_{i=1}^Lb_iT_i(y)+\Sigma_{i=2}^{L+2}S_i(y)
\ee
generates an approximate solution to the renormalized flow:
\be
\label{deferreur}
\pa_sQ_b  -\Delta Q_b+b_1 \Lambda Q_b+ \frac{f(Q_b)}{y^2}=\Psi_b+Mod(t)
\ee
with 
\be
\label{defmodtun}
Mod(t) =    \Sigma_{i=1}^{L}\left[(b_i)_s+(2i-1+c_{b_1})b_1b_i-b_{i+1}\right]\left[T_i+\Sigma_{j=i+1}^{L+2}\frac{\partial S_j}{\partial b_i}\right].
\ee
Here we used the convention $$b_{L+1}=0, \ \ T_0=\Lambda Q,$$ and $\Psi_b$ satisfies the bounds:\\
(i) Global weighted bounds:
\be
\label{controleh4erreur}
\forall 1\leq k\leq L, \ \ \int_{y\leq 2B_1}|H^{k}\Psi_b|^2 \lesssim b_1^{2k+2}|\log b_1|^C,
\ee
\be
\label{estimathehtwo}
\int_{y\leq 2B_1}\frac{1+|\log y|^2}{1+y^4}|H^L\Psi_b|^2+\int_{y\leq 2B_1}\frac{1+|\log y|^2}{1+y^2}|AH^L\Psi_b|^2\lesssim \frac{b_1^{2L+3}}{|\log b_1|^2},
\ee
\be
\label{controleh4erreursharp}
\int_{y\leq 2B_1}|H^{L+1}\Psi_b|^2 \lesssim \frac{b_1^{2L+4}}{|\log b_1|^2}.
\ee

(ii) Improved local control: there holds for some universal constant $C=C(L)>0$:
\be
\label{fluxcomputationone}
\forall 0\leq k\leq L+1, \ \ \int_{y\leq 2M}|H^k\Psi_b|^2\lesssim M^Cb_1^{2L+6}.
\ee
\end{proposition}

{\bf Proof of Proposition \ref{consprofapproch}}\\

{\bf step 1} Computation of the error.\\

We compute from \fref{decompqb}, \fref{deferreur}:
$$
\pa_sQ_b  -\Delta Q_b+b_1\Lambda Q_b+ \frac{f(Q_b)}{y^2} =  A_1+A_2$$ with $$A_1=b_1\Lambda Q+\Sigma_{i=1}^L\left[(b_i)_sT_i+b_iHT_i+b_1b_i\Lambda T_i\right]+\Sigma_{i=2}^{L+2}\left[\pa_sS_i+HS_i+b_1\Lambda S_i\right],$$ $$ A_2=\frac{1}{y^2}\left[f(Q+\alpha_b)-f(Q)-f'(Q)\alpha_b\right].
$$
Let us rearrange the first sum using the definition \fref{deftk}:
\bee
A_1 & = & b_1\Lambda Q+\pa_sS_{L+2}+b_1\Lambda S_{L+2}\\
& + & \Sigma_{i=1}^L\left[(b_i)_sT_i-b_iT_{i-1}+b_1b_i\Lambda T_i\right]+\Sigma_{i=2}^{L+1}\left[\pa_sS_i+b_1\Lambda S_i\right]+\Sigma_{i=1}^{L+1}HS_{i+1}\\
& = & \left[\pa_sS_{L+2}+b_1\Lambda S_{L+2}\right]+\left[HS_{L+2}+\pa_sS_{L+1}+b_1\Lambda S_{L+1}\right]\\
& + &  \Sigma_{i=1}^{L}\left[(b_i)_s+(2i-1+c_{b_1})b_1b_i-b_{i+1}\right]T_i\\
& + & \Sigma_{i=1}^{L}\left[HS_{i+1}+\pa_s S_i+b_1b_i(\Lambda T_i-(2i-1+c_{b_1})T_i)+b_1\Lambda S_i\right]
\eee
where $c_{b_1}$ is given by \fref{cb}. We now treat the time dependence using the anticipated approximate modulation equation:
\bee
\pa_sS_i & = & \Sigma_{j=1}^L(b_j)_s\frac{\partial S_i}{\partial b_j}\\
& = & \Sigma_{j=1}^L((b_j)_s+(2j-1+c_{b_1})b_1b_j-b_{j+1})\frac{\partial S_i}{\partial b_j}-\Sigma_{j=1}^L((2j-1+c_{b_1})b_1b_j-b_{j+1})\frac{\partial S_i}{\partial b_j}
\eee and thus using \fref{propprofiles}:
\bee
A_1 & = & \left\{ b_1\Lambda S_{L+2}-\Sigma_{i=1}^L((2i-1+c_{b_1})b_1b_i-b_{i+1})\frac{\partial S_{L+2}}{\partial b_i}\right\}\\
& + & \left\{HS_{L+2}+b_1\Lambda S_{L+1}-\Sigma_{i=1}^L((2i-1+c_{b_1})b_1b_i-b_{i+1})\frac{\partial S_{L+1}}{\partial b_i}\right\}\\
& + & \Sigma_{i=1}^{L}\left[HS_{i+1}+b_1b_i(\Lambda T_i-(2i-1+c_{b_1})T_i)+b_1\Lambda S_i-\Sigma_{j=1}^{i-1}((2j-1+c_{b_1})b_1b_j-b_{j+1})\frac{\partial S_i}{\partial b_j}\right]\\
& + &  \Sigma_{i=1}^{L}\left[(b_i)_s+(2i-1+c_{b_1})b_1b_i-b_{i+1}\right]\left[T_i+\Sigma_{j={i+1}}^{L+2}\frac{\partial S_j}{\partial b_i}\right].
\eee
We now expand $A_2$ using a Taylor expansion: 
$$
A_2  =  \frac{1}{y^2}\left\{\Sigma_{j=2}^{L+2}\frac{f^{(j)}(Q)}{j!}\alpha_b^j+R_2\right\}
$$
with
\be
\label{defrtwo}
R_2=\frac{\alpha_b^{L+3}}{(L+2)!}\int_0^1(1-\tau)^{L+2}f^{(L+3)}(Q+\tau \alpha_b)d\tau.
\ee
We sort the Taylor polynomial\footnote{remember that $b_i$ is or order $b_1^i$ from \fref{aprioirbound}} using the notation \fref{normsonpupelts} for the $2L+1$ uplet $J=(i_1,\dots,i_L,j_2,\dots,j_{L+2})\in \Bbb N^{2L+1}$:
$$|J|_1=\Sigma_{k=1}^Li_k+\Sigma_{k=2}^{L+2}j_k,  \ \ |J|_2=\Sigma_{k=1}^Lki_k+\Sigma_{k=2}^{L+2}kj_k$$
and thus:
\bee
\Sigma_{j=2}^{L+2}\frac{f^{(j)}(Q)}{j!}\alpha_b^j & = & \Sigma_{j=2}^{L+2}\frac{f^{(j)}(Q)}{j!}\Sigma_{|J|_1=j}c_{J}\Pi_{k=1}^Lb^{i_k}_kT_k^{i_k}\Pi_{k=2}^{L+2}S_k^{j_k}\\
& = & \Sigma_{i=2}^{L+2}P_i+R_1
\eee
where 
\be
\label{defpi}
P_i= \Sigma_{j=2}^{L+2}\frac{f^{(j)}(Q)}{j!}\Sigma_{|J|_1=j,|J|_2=i}c_{J}\Pi_{k=1}^Lb^{i_k}_kT_k^{i_k}\Pi_{k=2}^{L+2}S_k^{j_k},
\ee
\be
\label{defrun}
R_1=\Sigma_{j=2}^{L+2}\frac{f^{(j)}(Q)}{j!}\Sigma_{|J|_1=j, |J|_2\geq L+3}c_{J}\Pi_{k=1}^Lb^{i_k}_kT_k^{i_k}\Pi_{k=2}^{L+2}S_k^{j_k}.
\ee
We finally use the definitions \fref{deftk}, \fref{defthetak}, \fref{defmodtun} to rewrite: $$\Lambda T_i-(2i-1+c_{b_1})T_i=\Theta_i+(-1)^{i+1}H^{-i+1}\Sigma_{b_1}-c_{b_1}T_i=\Theta_i+(-1)^{i+1}H^{-i+1}(\Sigma_{b_1}-c_{b_1}T_1)$$ which together with \fref{deferreur} yields the following expression for the error:
\bea
\label{expressionpsibone}
&& \Psi_b=\Sigma_{i=1}^L(-1)^{i+1}b_1b_iH^{-i+1}(\Sigma_{b_1}-c_{b_1}T_1)\\
\nonumber & + &  \left\{ b_1\Lambda S_{L+2}-\Sigma_{i=1}^L((2i-1+c_{b_1})b_1b_i-b_{i+1})\frac{\partial S_{L+2}}{\partial b_i}+\frac{1}{y^2}[R_1+R_2]\right\}\\
\nonumber & + & \left\{HS_{L+2}+b_1\Lambda S_{L+1}+\frac{P_{L+2}}{y^2}-\Sigma_{i=1}^L((2i-1+c_{b_1})b_1b_i-b_{i+1})\frac{\partial S_{L+1}}{\partial b_i}\right\}\\
\nonumber & + & \Sigma_{i=1}^{L}\left[HS_{i+1}+b_1b_i\Theta_i+b_1\Lambda S_i+\frac{P_{i+1}}{y^2}-\Sigma_{j=1}^{i-1}((2j-1+c_{b_1})b_1b_j-b_{j+1})\frac{\partial S_i}{\partial b_j}\right].
\eea
We now construct iteratively the sequence of profiles $(S_i)_{1\leq i\leq L+2}$ through the scheme: 
\be
\label{defsi}
\left\{\begin{array}{ll}S_1=0,\\
S_{i}=-H^{-1}\Phi_{i}, \ \ 2\leq i\leq L+2
\end{array}\right .
\ee
with for $1\leq i\leq L$:
\be
\label{defphii}
\Phi_{i+1} = b_1b_i\Theta_i+b_1\Lambda S_i+\frac{P_{i+1}}{y^2}-\Sigma_{j=1}^{i-1}((2j-1+c_{b_1})b_1b_j-b_{j+1})\frac{\partial S_i}{\partial b_j},
\ee
\be
\label{defphilplusun}
\Phi_{L+2}=b_1\Lambda S_{L+1}+\frac{P_{L+2}}{y^2}-\Sigma_{i=1}^L((2i-1+c_{b_1})b_1b_i-b_{i+1})\frac{\partial S_{L+1}}{\partial b_i}.
\ee

{\bf step 2} Control of $\Phi_i,S_i$.\\

We claim by induction on $i$ that $\Phi_i$ is homogeneous with
\be
\label{homoegephii}
deg(\Phi_i)= (i-1,i-1,i) \ \ \mbox{for} \ \ 2\leq i\leq L+2
\ee
and 
\be
\label{dependanceparameters}
 \frac{\partial \Phi_i}{\partial b_j}=0\ \ \mbox{for}\ \ 2\leq i\leq j\leq L+2.
 \ee
This implies from Lemma \ref{lemmapropinversebis} that $S_i$ given by \fref{defsi} is homogeneous and satisfies \fref{propprofiles} for $2\leq i\leq L+2$.\\
\underline{{\em $i=1$}}: We compute explicitely: $$\Phi_2=b_1^2\Theta_1+b_1^2\frac{f''(Q)}{2y^2}T^2_1$$ which satisfies \fref{dependanceparameters}. Recall from \fref{assumtiong} that $f=gg'$ is odd and $\pi$ periodic so that the expansions \fref{origin}, \fref{infinity} yield at the origin: 
\be
\label{estforigin}
\frac{f^{(j)}(Q)}{y^2}=\left\{\begin{array}{ll}\Sigma_{k= -1}^py^{2k+1}+O(y^{2p+3})\ \ \mbox{for j even},\\ \Sigma_{k= -1}^py^{2k}+O(y^{2p+2})\ \ \mbox{for j odd}\end{array}\right.
\ee
and at infinity:
\be
\label{estfinfinity}
\frac{f^{(j)}(Q)}{y^2}=\left\{\begin{array}{ll}\Sigma_{k= 1}^py^{-2k-1}+O(y^{-2p-3})\ \ \mbox{for j even},\\ \Sigma_{k= 1}^py^{-2k}+O(y^{-2p-2})\ \ \mbox{for j odd}\end{array}\right.
\ee 
From Lemma \ref{lemmaradiation} and Lemma \ref{lemmaradiationbis}, $T_1$ and $\Theta_1$ are respectiviely admissible and $b_1$-admissible of order $(1,1)$. In particular, we have the Taylor expansion near the origin $$\frac{f''(Q)}{2y^2}T^2_1=\Sigma_{k=1}^pc_ky^{2k+1}+O(y^{2p+3}), \ \ p\geq 1$$ and the bound at infinity $$\left|\mathcal A^k\left(\frac{f''(Q)}{2y^2}T^2_1\right)\right|\lesssim \frac{1}{y^{3+k}}y^2 |\log y|^2\lesssim y^{2-k-3}|\log y|^2,\ \ k\geq 0$$ and hence $\frac{f''(Q)}{2y^2}T^2_1$ is $b_1$-admissible of degre $(1,1)$. We conclude that $\Phi_2$ is homogeneous with $$deg(\Phi_2)=(1,1,2).$$
\underline{{\em $i\to i+1$}}: We estimate all terms in \fref{defphii}. \fref{dependanceparameters} holds by direct inspection. From Lemma \ref{lemmaradiationbis}, $b_1b_i\Theta_i$ is homogeneous of degree $(i,i,i+1)$. From Lemma \ref{lemmapropinversebis}, $b_1\Lambda S_i$ is by induction homogeneous of degree $(i,i,i+1)$. For $j\geq 2$, we have by definition and induction that $$((2j-1+c_{b_1})b_1b_j-b_{j+1})\frac{\partial S_i}{\partial b_j}$$ is homogeneous of degree $(i,i,i+1)$. For $j=1$, we rewrite this term $$((1+c_{b_1})b_1^2-b_{2})\frac{\partial S_i}{\partial b_1}=\left((1+c_{b_1})b_1-\frac{b_2}{b_1}\right)\left(b_1\frac{\partial S_i}{\partial b_1}\right)$$ and conclude recalling Remark \ref{remarkfun} that this term is also homogeneous of degree $(i,i,i+1)$. It thus remains to estimate the nonlinear term $\frac{P_{i+1}}{y^2}$ in \fref{defphii} which from \fref{defpi} is a linear combination of monomials of the form\footnote{Observe that terms involving $k\geq i+1$ are indeed forbidden in the last product are forbidden from the constraint $|J|_1\geq 2$, $|J|_2=i+1$.}: 
$$M_J(y)=\frac{f^{(j)}(Q)}{y^2}\Pi_{k=1}^ib^{i_k}_kT_k^{i_k}\Pi_{k=2}^{i}S_k^{j_k}, \ \ |J|_1=j,\ \ |J|_2=i+1, \ \ 2\leq j\leq i+1.
$$ We conclude using \fref{estforigin}, \fref{estfinfinity} that $M_J$ is admissible with at the origin the development: for $j=2l$,
\bee
M_J(y) & = & y^{-1}y^{\Sigma_{k\geq 1}i_k(2k+1)+j_k(2k+1)}(c_0+c_2y^2+\dots+c_py^{2p}+o(y^{2p+1}))\\
& = & y^{2|J|_2+j-1}(c_0+c_2y^2+\dots+c_py^{2p}+o(y^{2p+1}))\\
& = & y^{2(i+l)+1}(c_0+c_2y^2+\dots+c_py^{2p}+o(y^{2p+1}))
\eee
and for $j=2l+1$:
\bee 
M_J(y)& = &  y^{-2}y^{\Sigma_{k\geq 1}i_k(2k+1)+j_k(2k+1)}(c_0+c_2y^2+\dots+c_py^{2p}+o(y^{2p+1}))\\
& = & y^{2|J|_2+j-2}(c_0+c_2y^2+\dots+c_py^{2p}+o(y^{2p+1}))\\
& = & y^{2(i+l)+1}(c_0+c_2y^2+\dots+c_py^{2p}+o(y^{2p+1})).
\eee
Now $j\ge 2$ ensures $l\ge 1$ and hence $M_J$ admits a Taylor expansion \fref{assumptoinfs} at the origin with $p_1=i+1$. For $y\geq 1$, the rough bound \fref{roughbound} and \fref{taylorexpansionoriginbis} imply $$|S_j|\lesssim b_1^jy^{2j-1}, \ \ |T_j|\lesssim y^{2j-1}|\log y|^C$$ which together with \fref{estfinfinity} yields the control:
\bea
\label{estmjbis}
 M_J(y)&\lesssim&|\log y|^C \left\{\begin{array}{ll} y^{2|J|_2-j-3}=y^{2(i-l)-1}\ \ \mbox{for}\ \ j=2l\geq 2\\
y^{2|J|_2-j-2}=y^{2(i-l)-1}\ \ \mbox{for}\ \ j=2l+1\geq 3\\
\end{array}\right.\\
\nonumber & \lesssim & y^{2i-3}|\log y|^C.
\eea
which is compatible with the degree $i$ control at infinity \fref{boundinfitybis}. The control of further derivatives in $(y,b_1)$ follows from \fref{roughbound} and Leibniz rule.  This concludes the proof of \fref{homoegephii}.\\
 
 {\bf step 3} Estimate on the error.\\
 
 We compute from \fref{expressionpsibone}:
 \be
 \label{formulareste}
\Psi_b=\Psi_b^{(0)}+\Psi^{(1)}_b, 
\ee
\be
\label{formulapisbzero}
\Psi_{b}^{(0)}=\Sigma_{i=1}^L(-1)^{i+1}b_1b_iH^{-i+1}\Sigmat_{b_1}\ \ \mbox{with}\ \ \Sigmat_{b_1}=\Sigma_{b_1}-c_{b_1}T_1,
\ee
\be
\label{defpsibnot}
\Psi^{(1)}_b=  b_1\Lambda S_{L+2}-\Sigma_{i=1}^L((2i-1+c_{b_1})b_1b_i-b_{i+1})\frac{\partial S_{L+2}}{\partial b_i}+\frac{1}{y^2}[R_1+R_2].
\ee
{\it Estimates for $\Psi_b^{(0)}$}: First observe from \fref{sigmab0}, \fref{formulapisbzero} that 
\be
\label{supporttilde}
\mbox{Supp}\Sigmat_{b_1}\subset\{y\geq \frac{B_0}{4}\}.
\ee
We extract from \fref{formulasigmbone} the rough bound for $k\geq 0$ and $\frac{B_0}{4}\leq y\leq 2B_1$: $$|H^{-k}\Sigmat_{b_1}|\lesssim 1+y^{2k+1}$$ and thus $$\int_{y\leq 2B_1}|H^{-k}\Sigmat_{b_1}|^2\lesssim b_1^{-2k-2}|\log b_1|^C,  \ \ 0\leq k\leq L.$$
On the other hand, from \fref{defsigmab} and the cancellation $H\Lambda Q=0$:
\be
\label{firstesitmate}
|H\Sigmat_{b_1}|\lesssim \frac{1}{|\log b_1|}\left(\frac{1}{1+y}\right){\bf 1}_{y\ge \frac{B_0}{4}}, 
\ee
\be
\label{firstesitmatebis}
 |H^k\Sigmat_{b_1}|\lesssim \frac{1}{B_0^{2(k-1)}|\log b_1|}\left(\frac{1}{1+y}\right){\bf 1}_{B_0\leq y\leq 3B_0}\ \ \mbox{for}\ \ k\geq 2.
 \ee
 This leads to the bound: $$\int_{y\leq 2B_1}|H\Sigmat_{b_1}|^2\lesssim \frac{1}{|\log b_1|}, \ \ \int |H^{k}\Sigmat_{b_1}|^2\lesssim \frac{b_1^{2k-2}}{|\log b_1|^2}\ \ \mbox{for}\ \ k\geq 2.$$ We thus estimate from \fref{aprioirbound}: for $0\leq k\leq L$, $$\int_{y\leq 2B_1}|H^k\Psi_b^{(0)}|^2\lesssim |\log b_1|^C\Sigma_{i=1}^Lb_1^{2+2i}b_1^{2(k-i+1)-2}\lesssim b_1^{2k+2}|\log b_1|^C$$ and the sharp logarithmic gain:
\bee
\int|H^{L+1}\Psi_b^{(0)}|^2 & \lesssim&  \Sigma_{i=1}^{L}b_1^{2+2i}\|H^{L+2-i}\Sigmat_{b_1}\|_{L^2}^2\lesssim \frac{1}{|\log b_1|^2}\Sigma_{i=1}^Lb_1^{2+2i}b_1^{2(L+1-i+1)-2}\\
& \lesssim & \frac{b_1^{2L+4}}{|\log b_1|^2}.
\eee 
Similarily, using \fref{firstesitmate}, \fref{firstesitmatebis}:
\bee
&&\int_{y\leq 2B_1}\frac{1+|\log y|^2}{1+y^4}|H^L\Psi^{(0)}_b|^2 \lesssim  \Sigma_{i=1}^{L}b_1^{2+2i}\int_{y\leq 2B_1}\frac{1+|\log y|^2}{1+y^4}|H^{L-i+1}\Sigmat_{b_1}|^2\\
& \lesssim &\Sigma_{i=1}^{L}b_1^{2+2i}\int_{y\geq \frac{B_0}{4}}\frac{1+|\log y|^2}{1+y^4}\frac{b_1^{2(L-i+1-1)}}{|\log b_1|^2(1+y^2)}\lesssim b_1^{2L+4},
 \eee
\bee
\int_{y\leq 2B_1}\frac{1+|\log y|^2}{1+y^2}|AH^L\Psi^{(0)}_b|^2& \lesssim  & \Sigma_{i=1}^{L}b_1^{2+2i}\int_{y\leq 2B_1}\frac{1+|\log y|^2}{1+y^2}|AH^{L-i+1}\Sigmat_{b_1}|^2\\
& \lesssim &\Sigma_{i=1}^{L}b_1^{2+2i}b_1^{2(L-i)}\int_{y\geq \frac{B_0}{4}} \frac{1+|\log y|^2}{y^4(1+y^2)}\\
& \lesssim & b_1^{2L+4}|\log b_1|^C.
\eee

{\it Estimates for $\Psi_b^{(1)}$}: By construction, $S_{L+2}$ is homogeneous of degree $(L+2,L+2,L+2)$ and thus so is $\Lambda S_{L+2}$. We therefore estimate from \fref{estunhk}, \fref{aprioirbound}: $\forall 0\leq k\leq L+1$,
$$\int_{y\leq 2B_1}|b_1H^{k}\Lambda S_{L+2}|^2  \lesssim \frac{b_1^2b_1^{2L+4}|\log b_1|^{4(L+2-k-1)}}{b_1^{2(L+2-k)}|\log b_1|^2}=\frac{b_1^{2k+2}}{|\log b_1|^2}|\log b_1|^{4(L+1-k)},$$
and using the rough bound \fref{roughbound}:
\bee
&&\int_{y\leq 2B_1}(1+|\log y|^2)\left[\frac{|b_1H^{L}\Lambda S_{L+2}|^2}{1+y^4}+\frac{|b_1AH^{L}\Lambda S_{L+2}|^2}{1+y^2}\right]\\
& \lesssim &  b^2_1b_1^{2L+4}\int_{y\leq 2B_1}\frac{1+|\log y|^2}{1+y^4}(1+y^2)^{2(L+2)-1-2L}\lesssim b_1^{2L+4}|\log b_1|^C.
 \eee

We now turn to the control of $R_1$ which from \fref{defrun} is a linear combination of terms of the form:
$$\tilde{M}_J=\frac{f^{(j)}(Q)}{y^2}\Pi_{k=1}^Lb^{i_k}_kT_k^{i_k}\Pi_{k=2}^{L+2}S_k^{j_k}, \ \ |J|_1=j,\ \ |J|_2\geq L+3, \ \ 2\leq j\leq L+2. $$ At the origin, the homogeneity of $S_i$ and the admissibility of $T_i$ ensure the bound for $y\le 1$: $$|\tilde{M}_J(y)|\lesssim b_1^{L+3}\left\{\begin{array}{ll} y^{2|J|_2+j-1}=y^{2(|J|_2+l)-1}\ \ \mbox{for}\ \ j=2l\\
y^{2|J|_2+j-2}=y^{2(|J|_2+l)-1} \ \mbox{for}\ \ j=2l+1\\
\end{array}\right.\lesssim b_1^{L+3}y^{2L+6},
$$
and similarily like for \fref{estmjbis} for $1\leq y\leq 2B_1$:
$$|\tilde{M}_J(y)|\lesssim b_1^{|J|_2}|\log b_1|^C\left\{\begin{array}{ll} y^{2|J|_2-j-3}=y^{2(|J|_2-l)-3}\ \ \mbox{for}\ \ j=2l\\
y^{2|J|_2-j-2}=y^{2(|J|_2-l)-3}\ \ \mbox{for}\ \ j=2l+1\\
\end{array}\right.\lesssim b_1^{|J|_2}y^{2|J|_2-5}|\log b_1|^C
$$
where we used $j\geq 2$, and similarily for higher derivatives. This ensures the control at the origin: $$|H^{k}\tilde{M}_J(y)|\lesssim b_1^{L+3}\ \ \mbox{for}\ \ 0\leq k\leq L+1, \ \ y\leq 1$$ and for $y\geq 1$: $$|H^k\tilde{M}_J(y)|\lesssim b_1^{|J|_2}y^{2(|J|_2-k)-5}, \ \ 0\leq k\leq L+1$$ and thus: $\forall 0\leq k\leq L+1$, 
\bee
&&\int_{y\leq 2B_1}|H^k\tilde{M}_J|^2\\
 & \lesssim & b_1^{2L+6}+b_1^{2|J|_2}|\log b_1|^C\int_{y\leq 2B_1}y^{4(|J|_2-k)-10}\lesssim b_1^{2L+6}+b_1^{2|J|_2}B_1^{4(|J|_2-k)-8}|\log b_1|^C\\
& \lesssim & b_1^{2L+6}+b_1^{2k+4}|\log b_1|^C\lesssim b_1^{2k+3}.
\eee
Similarily,
\bee
&&\int_{y\leq 2B_1}(1+|\log y|^2)\left[\frac{|H^{L}\tilde{M}_J|^2}{1+y^4}+\frac{|AH^{L}\tilde{M}_J|^2}{1+y^2}\right]\\
& \lesssim & b_1^{2L+6}+|\log b_1|^C\int_{1\leq y\leq 2B_1}\frac{1+|\log y|^2}{1+y^4}\left|b_1^{|J|_2}y^{2(|J|_2-L)-5}\right|^2\\
& \lesssim & b_1^{2L+6}+b_1^{2|J|_2}B_1^{4(|J|_2-L)-12}|\log b_1|^C\lesssim  b_1^{2L+6}|\log b_1|^C.
\eee
It remains to estimate the $R_2$ term given by \fref{defrtwo}. Near the origin $y\leq 1$, we have by construction $|\alpha_b|\lesssim b_1y^3$ and thus for $0\leq k\leq L+1$, $y\leq 1$: $$\left|H^k\left(\frac{R_2}{y^2}\right)\right|\lesssim b_1^{L+3}y^{3(L+3)-2-2k}\lesssim b_1^{L+3}.$$ For $y\geq 1$, we use the rough bound by construction for $1\leq y\leq 2B_1$: $$|\alpha_b|\lesssim b_1y|\log y|^C$$ which yields the bound for $0\leq k\leq L+1$, $1\leq y\leq 2B_1$: $$\left|H^k\left(\frac{R_2}{y^2}\right)\right|\lesssim b_1^{L+3}|\log b_1|^Cy^{L+3-2-2k}$$ from which for $0\leq k\leq L+1$: 
\bee
\int_{y\leq 2B_1}|H^k\left(\frac{R_2}{y^2}\right)|^2&\lesssim & b_1^{2L+6} +b_1^{2L+6}|\log b_1|^C\int_{1\leq y\leq 2B_1} y^{2L+2-4k}\\
& \lesssim & \left\{\begin{array}{ll} b_1^{2L+6} \ \ \mbox{for}\ \ 2L+2-4k<-1\\
b_1^{2L+6}B_1^{2L+4-4k}|\log b_1|^C=b_1^{2k+L+4}|\log b_1|^C
\end{array}\right.\\
& \lesssim & b_1^{2k+5}|\log b_1|^C.
\eee
Similarily, $$\int_{y\leq 2B_1}(1+|\log y|^2)\left[\frac{1}{1+y^4}|H^{L}(\frac{R_2}{y^2})|^2+\frac{1}{1+y^2}|AH^{L}(\frac{R_2}{y^2})|^2\right]\lesssim  b_1^{2L+5}.$$
The collection of above estimates yields \fref{controleh4erreur}, \fref{controleh4erreursharp}.\\
Finally, the local control \fref{fluxcomputationone} is a simple consequence of the support localization \fref{supporttilde} and the fact that $\Psi_b^{(1)}$ given by \fref{defpsibnot} satisfies by construction a bound on compact sets: $$\forall y\leq 2M\ll B_0, \ \ \forall 0\leq k\leq L+1, \ \ |H^k\Psi_b^{(1)}(y)|\lesssim M^Cb_1^{L+3}.$$
This concludes the proof of Proposition \ref{consprofapproch}.


\subsection{Localization of the profile}

 
 We now proceed to a simple localization procedure of the profile $Q_b$ to avoid some irrelevant growth in the region $y\geq 2B_1$.
 
 \begin{proposition}[Localization]
\label{consprofapprochloc}
Under the assumptions of Proposition \ref{consprofapproch}, assume moreover the a priori bound 
\be
\label{aprioribound}
|(b_1)_s|\lesssim b_1^{2}
\ee
Let the localized profile 
\be
\label{decompqbt}
\tilde{Q}_{b(s)}(y)= Q(y) + \alphat_{b(s)}(y), \ \ \alphat_b(y)=\Sigma_{i=1}^Lb_i\tT_i(y)+\Sigma_{i=2}^{L+2}\tS_i(y)
\ee
with 
\be
\label{defts}
\tilde{T}_i=\chi_{B_1}T_i, \ \ \tS_i=\chi_{B_1}S_i.
\ee
Then
 \bea
\label{deferreutilder}
&&\pa_s\qbt  -\Delta \qbt+b_1\Lambda \qbt+ \frac{f(\qbt)}{y^2}\\
\nonumber & = & \Psit_b+ \Sigma_{i=1}^{L}\left[(b_i)_s+(2i-1+c_{b_1})b_1b_i-b_{i+1}\right]\left[\tT_i+\chi_{B_1}\Sigma_{j=i+1}^{L+2}\frac{\partial S_j}{\partial b_i}\right]
\eea
where $\Psit_b$ satisfies the bounds:\\
(i) Weighted bounds: 
\be
\label{controleh4erreurtilde}
\forall 1\leq k\leq L, \ \ \int |H^{k}\Psit_b|^2 \lesssim b_1^{2k+2}|\log b_1|^C
\ee
\be
\label{estimathehtwobis}
\int_{y\leq 2B_1}\frac{1+|\log y|^2}{1+y^4}|H^L\Psi_b|^2+\int\frac{1+|\log y|^2}{1+y^2}|AH^L\Psit_b|^2\lesssim \frac{b_1^{2L+3}}{|\log b_1|^2},
\ee
\be
\label{controleh4erreursharptildebis}
\int|H^{L+1}\Psit_b|^2 \lesssim \frac{b_1^{2L+4}}{|\log b_1|^2}.
\ee
(ii) Improved local control: there holds for some universal constant $C=C(L)>0$:
\be
\label{fluxcomputationonebis}
\forall 0\leq k\leq L+1, \ \ \int_{y\leq 2M}|H^k\Psit_b|^2\lesssim M^Cb_1^{2L+6}.
\ee

\end{proposition}

\begin{proof}[Proof of Proposition \ref{consprofapprochloc}]

We compute from localization:
\bee
&&\pa_s\qbt  -\Delta \qbt+b_1 \Lambda \qbt+ \frac{f(\qbt)}{y^2}=\chi_{B_1}\left\{\pa_sQ_b-\Delta Q_b+b_1\Lambda Q_b+\frac{f(Q_b)}{y^2}\right\}\\
& + & (\pa_s\chi_{B_1})\alpha_b-2\pa_y\chi_{B_1}\pa_y\alpha_b-\alpha_b\Delta \chi_{B_1}+b_1\alpha_b\Lambda \chi_{B_1}+b_1 (1-\chi_{B_1})\Lambda Q\\
&+ & \frac{1}{y^2}\left\{f(\qbt)-f(Q)-\chi_{B_1}(f(Q_b)-f(Q))\right\}
\eee
so that:
$$
\Psit_b=\chi_{B_1}\Psi_b+\Psit_b^{(0)}$$ with 
\bea
\label{defpdtigwzero}
\Psit_b^{(0)} & = & \frac{1}{y^2}\left\{f(\qbt)-f(Q)-\chi_{B_1}(f(Q_b)-f(Q))\right\}\\
\nonumber & + & (\pa_s\chi_{B_1})\alpha_b-2\pa_y\chi_{B_1}\pa_y\alpha_b-\alpha_b\Delta \chi_{B_1}+b_1\alpha_b\Lambda \chi_{B_1}+b_1 (1-\chi_{B_1})\Lambda Q.
\eea
 Note that all terms in the above RHS are localized in $B_1\leq y\leq 2B_1$ except the last one for which $\mbox{Supp}((1-\chi_{B_1})\Lambda Q)\subset\{y\geq B_1\}.$ Hence \fref{fluxcomputationone} implies \fref{fluxcomputationonebis}. The bounds \fref{controleh4erreurtilde}, \fref{estimathehtwobis}, \fref{controleh4erreursharptildebis} for $\chi_{B_1}\Psi_b$ follow verbatim like for the proof of \fref{controleh4erreur}, \fref{estimathehtwo}, \fref{controleh4erreursharp}.\\
To estimate the second error induced by localization in \fref{defpdtigwzero}, first observe from \fref{aprioribound} the bound: $$|\pa_s\chi_{B_1}|\lesssim \frac{|(b_1)_s|}{b_1}|y\chi'_{B_1}|\lesssim b_1{\bf 1}_{B_1\leq y\leq 2B_1}.$$ Moreover, from the admissibilty of $T_i$ and the $b_1$ admissibility of $S_i$, $T_i$ terms dominate for $y\sim B_1$ in $\alpha_b$, and we estimate from \fref{taylorexpansionoriginbis}: $\forall k\geq 0$ and $B_1\leq y\leq 2B_1$,
 \be
 \label{estaklphaba}
 \left|\frac{\pa^k}{\pa y^k}\alpha_b\right|\lesssim \Sigma_{i=1}^Lb_1^iy^{2i-k-1}(1+|\log b_1|)\lesssim \frac{|\log b_1|}{B_1^{k+1}}.
 \ee
 This yields: $\forall 1\leq k\leq L$, 
 \bea
 \label{covohohoe}
 \nonumber &&\int\left|H^k\left((\pa_s\chi_{B_1})\alpha_b-2\pa_y\chi_{B_1}\pa_y\alpha_b-\alpha_b\Delta \chi_{B_1}+b_1\alpha_b\Lambda \chi_{B_1}\right)\right|^2\\
\nonumber  & \lesssim & \int_{B_1\leq y\leq 2B_1}\left|\frac{b_1|\log b_1|}{B_1^{2k+1}}+\frac{|\log b_1|}{B_1^{2k+1+2}}\right|^2\lesssim  \frac{b_1^2}{B_1^{4k}}|\log b_1|^C\\
&  \lesssim & b_1^{2k+2}|\log b_1|^C
 \eea
 and 
\bea
\label{covohohoebis}
\nonumber  &&\int\left|H^{L+1}\left((\pa_s\chi_{B_1})\alpha_b-2\pa_y\chi_{B_1}\pa_y\alpha_b-\alpha_b\Delta \chi_{B_1}+b_1\alpha_b\Lambda \chi_{B_1}\right)\right|^2\\
 & \lesssim &  \int_{B_1\leq y\leq 2B_1}\left|\frac{b_1|\log b_1|}{B_1^{2(L+1)+1}}+\frac{|\log b_1|}{B_1^{2(L+1)+1+2}}\right|^2\lesssim \frac{b_1^{2L+4}}{|\log b_1|^2}.
 \eea
 We next estimate in brute force: $$\left|\frac{d^k}{dy^k}\left[(1-\chi_{B_1})\Lambda Q\right]\right|\lesssim \frac{1}{y^{k+1}}{\bf 1}_{y\geq B_1}$$ from which: $\forall 1\leq k\leq L+1$: $$\int\left||H^k\left(b_1 (1-\chi_{B_1})\Lambda Q\right)\right|^2\lesssim b_1^2\int_{B_1\leq y\leq 2B_1}\frac{1}{y^{4k+2}}\lesssim \frac{b_1^{2k+2}}{|\log b_1|^{4k}}.$$ It remains to estimate the nonlinear term for which we estimate using \fref{estaklphaba} and $|f'|\lesssim 1$:
 $$\left|\frac{\pa^k}{\pa y^k}\left\{\frac{1}{y^2}[f(\qbt)-f(Q)-\chi_{B_1}(f(Q_b)-f(Q))]\right\}\right|\lesssim \frac{|\log b_1|}{B_1^{k+1}}
 $$
  and the corresponding terms are estimated like for \fref{covohohoe}, \fref{covohohoebis}.
  \end{proof} 
  

\subsection{Study of the dynamical system for $b=(b_1,\dots,b_L)$}


The essence of the construction of the $Q_b$ profile is to generate according to \fref{defmodtun} the finite dimensional dynamical system \fref{systdynfundintro} for $b=(b_1,\dots,b_{L})$: 
\be
\label{systdynfund}
 (b_k)_s+\left(2k-1+\frac{2}{\log s}\right)b_1b_k-b_{k+1}=0,  \ \ 1\leq k\leq L, \ \ b_{L+1}\equiv 0.
 \ee 
 We show in this section that \fref{systdynfund} admits exceptional solutions, and that the linearized operator close to these solutions is explicit.

\begin{lemma}[Approximate solution for the $b$  system]
\label{lemmaexplicitsol}
Let $L\geq 2$ and $s_0\gg 1$ be a large enough universal constant. Let the sequences 
\bea
\label{defalphai}
&&\left\{\begin{array}{ll}c_1=\frac{L}{2L-1},\\
c_{k+1}=-\frac{L-k}{2L-1}c_k, \ \ 1\leq k\leq L-1,
\end{array}\right.\\
\label{defdk}
&&\left\{\begin{array}{ll}d_1=-\frac{2L}{(2L-1)^2},\\
d_{k+1}=-\frac{L-k}{2L-1}d_{k}+\frac{4L(L-k)}{(2L-1)^2}c_k, \ \ 1\leq k\leq L-1,
\end{array}\right.,
\eea
then the explicit choice 
\be
\label{approzimatesolution}
b_k^{e}(s)=\frac{c_k}{s^k}+\frac{d_k}{s^k\log s}, \ \ 1\leq k\leq L, \ \ b_{L+1}^{e}\equiv 0
\ee
generates an approximate solution to \fref{systdynfund} in the sense that:
\be
\label{approximatesolution}
(b_k^{e})_s+\left(2k-1+\frac{2}{\log s}\right)b^e_1b^e_k-b^e_{k+1}=O\left(\frac{1}{s^{k+1}(\log s)^2}\right), \ \ 1\leq k\leq L.
\ee
\end{lemma}

The proof of Lemma \ref{lemmaexplicitsol} is an explicit computation which is left to the reader. We now claim that this solution corresponds to a codimension $(L-1)$ exceptional manifold:

\begin{lemma}[Linearization]
\label{lemmalinear}
{\em 1. Computation of the linearized system}: Let 
\be
\label{defuk}
b_k(s)=b_k^e(s)+\frac{U_k(s)}{s^k(\log s)^{\frac 54}},\ \ 1\leq k\leq L, \ \ b_{L+1}=U_{L+1}\equiv 0, 
\ee
and note $U=(U_1,\dots,U_L)$.  Then:
\bea
\label{reformulaion}
\nonumber &&(b_k)_s+\left(2k-1+\frac{2}{\log s}\right)b_1b_k-b_{k+1}\\
& = & \frac{1}{s^{k+1}(\log s)^{\frac 54}}\left[s(U_k)_s-(A_LU)_k+O\left(\frac{1}{\sqrt{\log s}}+\frac{|U|+|U|^2}{\log s}\right)\right]
\eea
where 
\be
\label{defa}
A_L=(a_{i,j})_{1\leq i,j,\leq L} \ \ \mbox{with}\ \ \left\{\begin{array}{lllll}
a_{11}=-\frac{1}{2L-1}\\
a_{i,i+1}=1, \ \ 1\leq i\leq L-1\\
a_{1,i}=-(2i-1)c_{i}, \ \ 2\leq i\leq  L\\
a_{i,i}=\frac{L-i}{2L-1}, \ \ 2\leq i\leq L\\
a_{i,j}=0\ \ \mbox{otherwise}
\end{array}\right. .
\ee
{\em 2. Diagonalization of the linearized matrix}: $A_L$ is diagonalizable:
\be
\label{specta}
A_L=P_L^{-1}D_LP_L, \ \ D_L=\mbox{diag}\left\{-1,\frac{2}{2L-1},\frac{3}{2L-1}\dots..., \frac{L}{2L-1}\right\}.
\ee
\end{lemma}

\begin{proof}[Proof of Lemma \ref{lemmalinear}]
{\bf step 1} Linearization. A simple computation from \fref{approzimatesolution} ensures:
\bee
&&(b_k)_s+\left(2k-1+\frac{2}{\log s}\right)b_1b_k-b_{k+1}\\
& = & \frac{1}{s^{k+1}(\log s)^{\frac 54}}\left[s(U_k)_s-kU_k+O\left(\frac{|U|}{\log s}\right)\right]+O\left(\frac{1}{s^{k+1}(\log s)^2}\right)\\
& + & \frac{1}{s^{k+1}(\log s)^{\frac 54}}\left[(2k-1)c_kU_1+(2k-1)c_1U_k-U_{k+1}+O\left(\frac{|U|+|U|^2}{\log s}\right)\right],
\eee
and then the relation $$(2k-1)c_1-k=\frac{(2k-1)L}{2L-1}-k=-\frac{L-k}{2L-1}$$ ensures
\bee
&&(b_k)_s+\left(2k-1+\frac{2}{\log s}\right)b_1b_k-b_{k+1}\\
& = &   \frac{1}{s^{k+1}(\log s)^{\frac 54}}\left[s(U_k)_s+(2k-1)c_kU_1-\frac{L-k}{2L-1}U_k-U_{k+1}+O\left(\frac{1}{\sqrt{\log s}}+\frac{|U|+|U|^2}{\log s}\right)\right]
\eee
which is equivalent to \fref{reformulaion}, \fref{defa}.\\

{\bf step 2} Diagonalization. The proof follows by computing the characteristic polynomial. The cases $L=2,3$ are done by direct inspection. Let us assume $L\geq 4$, we compute $$P_L(X)=\det(A_L-X\mbox{Id})$$ by developing on the last row. This yields:
\bee
& & P_L(x) =  (-1)^{L+1}(-1)(2L-1)c_L+(-X)\bigg\{(-1)^{L}(-1)(2L-3)c_{L-1}\\
& + & \left(\frac{1}{2L-1}-X\right)\left[(-1)^{L-1}(-1)(2L-5)c_{L-2}+\left(\frac{2}{2L-1}-X\right)\left[\dots...\right]\right]\bigg\}.
\eee
We use the recurrence relation \fref{defalphai} to compute explicitly:
\bee
&&(-1)^{L+1}(-1)(2L-1)c_L\\
& + & (-X)\bigg\{(-1)^{L}(-1)(2L-3)c_{L-1}+ \left(\frac{1}{2L-1}-X\right)\left[(-1)^{L-1}(-1)(2L-5)c_{L-2}\right]\bigg\}\\
& = & (-1)^L\bigg\{(2L-3)c_{L-1}\left(X-\frac{1}{2L-3}\right)+(2L-5)c_{L-2}\left(X-\frac{1}{2L-1}\right)X\bigg\}.
\eee
We now compute from \fref{defalphai} for $1\leq k\leq L-2$:
\bea
\label{recurelation}
\nonumber &&(2L-(2k+1))c_{L-k}\left(X-\frac{1}{2L-(2k+1)}\right)+(2L-(2k+3))c_{L-(k+1)}X\left(X-\frac1{2L-1}\right)\\
\nonumber & = & (2L-(2k+3))c_{L-(k+1)}\left[X\left(X-\frac{1}{2L-1}\right)-\frac{2L-(2k+1)}{2L-(2k+3)}\frac{k+1}{2L-1}\left(X-\frac{1}{2L-(2k+1)}\right)\right]\\
& = &  (2L-(2k+3))c_{L-(k+1)}\left(X-\frac{k+1}{2L-1}\right)\left(X-\frac{1}{2L-(2k+3)}\right).
\eea
We therefore obtain inductively:
\bee
&&P_L(X)= (-1)^L\bigg\{(2L-3)c_{L-1}\left(X-\frac{1}{2L-3}\right)+(2L-5)c_{L-2}\left(X-\frac{1}{2L-1}\right)X\bigg\}\\
& + & (-X)\left(\frac{1}{2L-1}-X\right)\left(\frac{2}{2L-1}-X\right)\left[(-1)^{L-2}(-1)(2L-7)c_{L-3}+\left(\frac{3}{2L-1}-X\right)[\dots]\right]\\
&= & (-1)^L\left(X-\frac{2}{2L-1}\right)\bigg\{(2L-5)c_{L-2}\left(X-\frac{1}{2L-5}\right)+(2L-7)c_{L-3}X\left(X-\frac{1}{2L-1}\right)\bigg\}\\
& + & (-X)\left(\frac{1}{2L-1}-X\right)\left(\frac{2}{2L-1}-X\right)\left(\frac{3}{2L-1}-X\right)[(-1)^{L-3}(-1)(2L-9)c_{L-4}\dots]\\
& = & (-1)^L\left(X-\frac{2}{2L-1}\right)\dots \left(X-\frac{L-2}{2L-1}\right)\\
&& \times\bigg\{3c_2\left(X-\frac 13\right)+X\left(X-\frac{1}{2L-1}\right)\left(c_1+X-\frac{L-1}{2L-1}\right)\bigg\}.
\eee
We use \fref{recurelation} with $k=L-2$ to compute the last polynomial:
\bee
&& 3c_2\left(X-\frac 13\right)+X\left(X-\frac{1}{2L-1}\right)\left(c_1+X-\frac{L-1}{2L-1}\right)\\
& =& \bigg\{3c_2\left(X-\frac 13\right)+c_1X\left(X-\frac{1}{2L-1}\right)\bigg\}+X\left(X-\frac{1}{2L-1}\right)\left(X-\frac{L-1}{2L-1}\right)\\
& = & c_1\left(X-\frac{L-1}{2L-1}\right)(X-1)+X\left(X-\frac{1}{2L-1}\right)\left(X-\frac{L-1}{2L-1}\right)\\
& = & \left(X-\frac{L-1}{2L-1}\right)\left[\frac{L}{2L-1}(X-1)+X\left(X-\frac{1}{2L-1}\right)\right]\\
& = & \left(X-\frac{L-1}{2L-1}\right)\left(X-\frac{L}{2L-1}\right)\left(X+1\right).
\eee
We have therefore computed:
$$P_L(x)=(-1)^L\left(X-\frac{2}{2L-1}\right)\dots\left(X-\frac{L-2}{2L-1}\right)\left(X-\frac{L-1}{2L-1}\right)\left(X-\frac{L}{2L-1}\right)\left(X+1\right)$$ and \fref{specta} is proved.
\end{proof}
 

\section{The trapped regime}
\label{sectionthree}

In this section, we introduce the main dynamical tools at the heart of the proof of Theorem \ref{thmmain}. We start with describing the bootstrap regime in which the blow up solutions of Theorem \ref{thmmain} will be trapped. We then exhibit the Lyapounov type control of $H^k$ norms which is the heart of our analysis.


\subsection{Modulation}


We describe in this section the set of initial data leading to the blow up scenario of Theorem  \ref{thmmain}. Let a smooth 1-corotational initial data 
\be
\label{ineingonei}
v(0,x)=\left|\begin{array}{lll} g(u_0(r))\cos\theta\\g(u_0(r))\sin\theta\\z(u_0(r))\end{array}\right. \ \ \mbox{with}\ \ \|\nabla u_0-\nabla Q\|_{L^2}\ll 1,
\ee and let $v(t,x)$ be the corresponding smooth solution to \fref{hfgeneral} with life time $0<T<+\infty$. From Lemma \ref{corotv}, we may decompose on a small time interval 
\be
\label{fullmap}
v(t,x)=\left|\begin{array}{lll} g(u(t,r))\cos\theta\\g(u(t,r))\sin\theta\\z(u(t,r))\end{array}\right.
\ee where 
\be
\label{ehoineofn}
\tilde{\e}(t,r)=u(t,r)-Q(r)\ \ \mbox{satisfies}\ \ \fref{estlinfty}
\ee Moreover from standard argument:
\be
\label{regularitiyt}
T<+\infty\ \ \mbox{implies}\ \ \|\Delta v(t)\|_{L^2}\to +\infty\ \ \mbox{as} \ \ t\to T.
\ee 
 We now modulate the solution and introduce from standard argument\footnote{see for example \cite{MMJMPA}, \cite{MR1}, \cite{RaphRod} for a further introduction to modulation.} using the initial smallness \fref{ineingonei} the unique decomposition of the flow defined on a small time $t\in[0,t_1]$ :
\be
\label{decompu}
u(t,r)=(\tilde{Q}_{b(t)}+\e(t,r))_{\lambda(t)}, \ \ \lambda(t)>0, \ \ b=(b_1,\dots,b_L)
\ee where $\e(t)$ satisfies the $L+1$ orthogonality conditions: 
\be
\label{ortho}
 (\e,H^{k}\Phi_M)=0, \ \ 0\leq k\leq L
\ee
and the smallness $$ \|\nabla \e(t)\|_{L^2}+\|\frac{\e(t)}{y}\|_{L^2}+|b(t)|\ll 1.$$ Here given $M>0$ large enough, we defined
\be
  \label{defdirection}
  \Phi_M = \Sigma_{p=0}^Lc_{p,M} H^p(\chi_M\Lambda Q)
    \ee
     where 
    $$c_{0,M}=1, \ \  c_{k,M}=(-1)^{k+1}\frac{\Sigma_{p=0}^{k-1}c_{p,M}(\chi_MH^p(\chi_M\Lambda Q),T_k)}{(\chi_M\Lambda Q,\Lambda Q)}, \ \ 1\leq k\leq L,
    $$
    is manufactured to ensure the nondegeneracy 
    \be
    \label{nondgeener}
    (\Phi_M,\Lambda Q)=(\chi_M\Lambda Q,\Lambda Q)=4\log M (1+o(1))\ \ \mbox{as}\ \ M\to +\infty
    \ee
    and the cancellation: $\forall 1\leq k\leq L$,
  \be
  \label{orthohiM}
  (\Phi_M,T_k)=\Sigma_{p=0}^{k-1}c_{p,M}(H^p(\chi_M\Lambda Q),T_k)+c_{k,M}(-1)^k(\chi_M\Lambda Q,\Lambda Q)=0.
  \ee 
  In particular, 
  \be
  \label{vnkoeoueo}
  (H^iT_j,\Phi_M)=(-1)^j(\chi_M\Lambda Q,\Lambda Q)\delta_{i,j}, \ \ 0\leq i,j\leq L.
  \ee
  Observe also by induction
  \bea
  \label{sizephimltwo}
  \forall 1\leq p\leq L, \ \  |c_{p,M}|\lesssim M^{2p}
  \eea
  from which:
  \be
  \label{estphimlto}
 \int|\Phi_M|^2\lesssim \int|\chi_M\Lambda Q|^2+\Sigma_{p=1}^Lc_{p,M}^2\int|H^p(\chi_M\Lambda Q)|^2\lesssim \log M.
 \ee
 The existence of the decomposition \fref{decompu} is a standard consequence of the implicit function theorem and the explicit relations $$\left(\frac{\pa}{\pa\l}(\qbt)_\lambda,\frac{\pa}{\pa b_1}(\qbt)_\lambda, \dots, \frac{\pa}{\pa b_L}(\qbt)_\lambda\right)|_{\l=1,b=0}=(\Lambda Q,T_1, \dots,T_L)$$ which using \fref{orthohiM} imply the non degeneracy of the Jacobian:
 \bee
\left|\left(\frac{\pa}{\pa (\l,b_j)}({\qbt})_\l,H^i\Phi_M\right)_{1\leq j\leq L,0\leq i\leq L} \right|_{\l=1, b=0}=(\chi_M\Lambda Q,\Lambda Q)^{L+1}\neq 0.
\eee 
The decomposition \fref{decompu} exists as long as $t<T$ and $\e(t,r)$ remains small in the energy topology. Observe also from \fref{ehoineofn}, \fref{decompu} and the explicit structure of $\qbt$ that $\e$ satisfies \fref{estlinfty}, and in particular Lemma \ref{propcorc} applies. In other words, we may measure the regularity of the map through the following coercive norms of $\e$: the energy norm
\be
\label{defenergyspcar}
\|\e\|_{\mathcal H}^2=\int|\pa_y\e|^2+\int\frac{|\e|^2}{y^2},
\ee
and higher order Sobolev norms adapated to the linearized operator 
\be
\label{defnorme}
\mathcal E_{2k}=\int|H^k\e|^2, \ \ 1\leq k\leq L+1.
\ee


\subsection{Setting up the bootstrap}


We now choose our set of initial data in a more restricted way. More precisely, pick a large enough time $s_0\gg1 $ and rewrite the decomposition \fref{decompu}:
\be
\label{neneonneo}
u(t,r)=(\tilde{Q}_{b(s)}+\e)(s,y)
\ee
where we introduced the renormalized variables:
\be
\label{resacledtime}
y=\frac{r}{\l(t)}, \ \ s(t)=s_0+\int_0^{t}\frac{d\tau}{\l^2(\tau)}
\ee
and now measure time in $s$ which will be proved to be a global time. We introduce a decomposition \fref{defuk}:
\be
\label{defukbis}
b_k=b_k^e+\frac{U_k}{s^k(\log s)^{\frac 54}},\ \ 1\leq k\leq L, \ \ b_{k+1}=U_{k+1}\equiv 0.
\ee  
and consider the variable 
\be
\label{devjj}
V=P_L U
\ee
where $P_L$ refers to the diagonalization \fref{specta} of $A_L$. We assume that initially:
\be
\label{estintial}
|V_1(0)|\leq 1, \left(V_2(0),\dots,V_L(0)\right)\in \mathcal B_{L-1}(2).
\ee
We also assume the explicit initial smallness of the data:
\be
\label{init2energy}
\int | \nabla \varepsilon (0)|^2 + \int \left|\dfrac{ \varepsilon (0)}{y}\right|^2 \leq b_1^2(0),
\ee
\be
\label{init2}
|\mathcal E_{2k}(0)| \leq [b_1(0)]^{10L+4}, \ \ 1\leq k\leq L+1.
\ee
Note also that up to a fixed rescaling, we may always assume:
\be
\label{intilzero}
\l(0)=1.
\ee
We then claim the following:

\begin{proposition}[Bootstrap]
\label{bootstrap}
There exists $$\left(V_2(0),\dots,V_L(0)\right)\in \mathcal B_{L-1}(2)$$ such that the following bounds hold: for all $s\geq s_0$,
\begin{itemize}
\item Control of the radiation:
\be
\label{init2h}
\int | \nabla \varepsilon (s)|^2 + \int \left|\dfrac{ \varepsilon (s)}{y}\right|^2 \leq 10(b_1(0))^{\frac 14},
\ee
\be
\label{init3h}
|\mathcal E_{2k}(s)| \leq b_1^{(2k-1)\frac{2L}{2L-1}}(s)|\log b_1(s)|^K,\ \ 1\leq k\leq L,
\ee
\be
\label{init3hb}
|\mathcal E_{2L+2}(s)| \leq K \frac{b_1^{2L+2}(s)}{|\log b_1(s)|^2}.
\ee
\item Control of the unstable modes:
\be
\label{controlunstable}
|V_1(s)|\leq 2, \ \ \left(V_2(s),\dots,V_L(s)\right)\in \mathcal B_{L-1}(2).
\ee
\end{itemize} 
\end{proposition}

\begin{remark}
\label{remarkcauchy}
Note that the bounds \fref{init3h} easily imply\footnote{see \cite{RSc1} for the full computation.} the control of the $H^2$ norm of the full map \fref{fullmap} $$\int |\Delta v(s)|^2<C(s)<+\infty, \ \ s<s^*$$ and therefore the blow up criterion \fref{regularitiyt} ensures that the map is well defined on $[s,s^*)$.
\end{remark}

Equivalently, given $(\e(0),V(0))$ as above, we introduce the time \bee
s^*& = & s^*(\e(0),V(0))\\
\nonumber & = & \sup\{s\geq s_0\ \ \mbox{such that} \ \ \fref{init2h}, \fref{init3h}, \fref{init3hb}, \fref{controlunstable}\ \ \mbox{hold on }\ \ [s_0,s]\}.
\eee
Observe that the continuity of the flow and the initial smallness \fref{init2energy}, \fref{init2} ensure that $s^*>0$. We then assume by contradiction:
\be
\label{hypcontr}
\forall \left(V_2(0),\dots,V_L(0)\right)\in \mathcal B_{L-1}(2), \ \ s^*<+\infty,
\ee
 and look for a contradiction. Our main claim is that the a priori control of the unstable modes \fref{controlunstable} is enough to improve the bounds  \fref{init2h}, \fref{init3h}, \fref{init3hb}, and then the claim follows from the $(L-1)$ codimensional instability \fref{specta} of the system \fref{systdynfund}  near the exceptional solution $b^e$ through a standard topological argument \`a la Brouwer.\\
 
 The rest of this section is devoted to the derivation of the key Lemmas for the proof of Proposition \ref{bootstrap} which is completed in section \ref{bootstrap}. We will make a systematic implicit use of the interpolation bounds of Lemma \ref{lemmainterpolation} which are a consequence of the coercivity of the $\mathcal E_{2k+2}$ energy given by Lemma \ref{propcorc}.


\subsection{Equation for the radiation}


Recall the decomposition of the flow:
  \bee
  u(t,r) = (\tilde{Q}_{b(t)} + \varepsilon) (s,y) = (Q + \tilde{\alpha}_{b(t)})_{\lambda(s)} + w(t,r).
  \eee
 We use the rescaling formulas
  \bea
  \nonumber u(t,r) = v(s,y)\mbox{,} \ \ \ y = \frac{r}{\lambda(t)} \mbox{,}  \ \ \ \ \ \partial_t u = \frac{1}{\lambda^2(t)}(\partial_s v - \frac{\lambda_s}{\lambda}\Lambda v)_{\lambda}
  \eea  
  to derive the equation for $\e$ in renormalized variables:
  \be
  \label{eqepsilon}
  \partial_s \varepsilon  - \frac{\lambda_s}{\lambda} \Lambda \varepsilon + H \varepsilon = F - \widetilde{Mod} = \mathcal F.
  \ee
  Here  $H$ is the linearized operator given by \fref{defh}, $\widetilde{Mod}(t)$ is given by 
  \bea
\label{defmodtuntilde}
\widetilde{Mod}(t)& = & -\left(\lsl+b_1\right)\Lambda \qbt\\
\nonumber & + &   \Sigma_{i=1}^{L}\left[(b_i)_s+(2i-1+c_{b_1})b_1b_i-b_{i+1}\right]\left[\tT_i+\chi_{B_1}\Sigma_{j=i+1}^{L+2}\frac{\partial S_j}{\partial b_i}\right],
\eea
and
  \be
  \label{defF}
  F = -\Psit_b +L(\varepsilon) - N(\varepsilon)
  \ee
where $L$ is the linear operator corresponding to the error in the linearized operator from $Q$ to $\tilde{Q}_b$:
  \be
  \label{defLe}
 L(\varepsilon) =  \frac{f'(Q)-f'(\qbt)}{y^2}\e
  \ee
  and the remainder term is the purely nonlinear term:
  \be
  \label{defNe}
  N(\varepsilon) = \frac{f(\tilde{Q}_b+ \varepsilon) - f(\tilde{Q}_b)-\e f'(\qbt)}{y^2}. 
  \ee
We also need to write the flow \fref{eqepsilon} in original variables. For this, let the rescaled operators
$$A_{\lambda} = -\partial_r + \frac{Z_{\lambda}}{r}, \ \ \ A^*_{\lambda} = \partial_r + \frac{1+Z_{\lambda}}{r}\ \ $$
\be
\label{defhtilde}
H_{\lambda} = A^*_{\lambda}A_{\lambda} = -\Delta + \frac{V_{\lambda}}{r^2}, \ \ \tilde{H}_{\lambda} = A_{\lambda}A^*_{\lambda} = -\Delta + \frac{\tilde{V}_{\lambda}}{r^2},
\ee
and the renormalized function $$w(t,r)=\e(s,y),$$
then \fref{eqepsilon} becomes:
  \be
  \label{eqenwini}
  \partial_t w + H_{\lambda}w = \frac 1{\lambda^2} \mathcal F_{\lambda}.
  \ee
Observe from \fref{approzimatesolution} that for $s<s^*$, 
 \be
 \label{controlbkbk}
 |b_k|\lesssim b_1^k, \ \ 0<b_1\ll 1
 \ee
 and hence the a priori bound \fref{aprioirbound} holds.


    \subsection{Modulation equations}
  

Let us now compute the modulation equations for $(b,\l)$ as a consequence of the choice of orthogonality conditions \fref{ortho}.

  \begin{lemma}[Modulation equations]
\label{modulationequations}
There holds the bound on the modulation parameters :
\be
\label{parameters}
\left|\frac{\lambda_s}{\lambda} + b_1\right| +\Sigma_{k=1}^{L-1}|(b_k)_s+(2k-1+c_{b_1})b_1b_k-b_{k+1} | \lesssim b_1^{L+\frac32},
\ee
\be
\label{parameterspresicely}
\left| (b_L)_s + (2L-1+c_{b_1})b_1b_L \right| \lesssim \frac{1}{\sqrt{\log M}} \left( \sqrt{ \mathcal E_{2L+2}} +  \frac{b_1^{L+1}}{|\log b_1|}  \right).
\ee
\end{lemma}

\begin{remark} Note that this implies in the bootstrap the rough bound:
\be
\label{rougboundpope}
|(b_1)_s|\leq 2b_1^2.
\ee
and in particular \fref{aprioribound} holds. 
\end{remark}

\begin{proof}[Proof of Lemma \ref{modulationequations}]

{\bf step 1} Law for $b_L$. Let 
\be
 \label{defut}
 D(t) = \left|\frac{\lambda_s}{\lambda} + b_1\right|+\Sigma_{k=1}^L|(b_k)_s + (2k-1+c_{b_1})b_1b_k-b_{k+1}|.
 \ee
 We take the inner product of \fref{eqepsilon} with $H^L\Phi_M$ and obtain using the orthogonality \fref{ortho}:
\bea
\label{modequone}
(\widetilde{\mbox{Mod}}(t),H^L\Phi_M)& = & -(\Psit_b,H^L\Phi_M)-(H^L\e,H\Phi_M)\\
\nonumber & - & \left(-\lsl\Lambda \e-L(\e)+N(\e),H^L\Phi_M\right).
\eea
 We first compute from the construction of the profile, \fref{defmodtuntilde}, the localization $\mbox{Supp}(\Phi_M)\subset[0,2M]$ from \fref{defdirection} and the identities \fref{nondgeener}, \fref{orthohiM}, \fref{vnkoeoueo}:
 \bee
&&\nonumber \left(H^L\left(\widetilde{\mbox{Mod}}(t)\right),\Phi_M\right) = -\left(b_1+\frac{\lambda_s}{\lambda}\right) \left(H^L\Lambda \qbt , \Phi_M\right)\\
& + &  \Sigma_{i=1}^{L}\left[(b_i)_s+(2i-1+c_{b_1})b_1b_i-b_{i+1}\right]\left(\tT_i+\chi_{B_1}\Sigma_{j=i+1}^{L+2}\frac{\partial S_j}{\partial b_i},H^L\Phi_M\right)\\
& = &  (-1)^L(\Lambda Q, \Phi_M)((b_L)_s + (2L-1+c_{b_1})b_1b_L)+O\left(M^Cb_1| D(t)|\right).
\eee
The linear term in \fref{modequone} is estimated\footnote{Observe that we do not use the interpolated bounds of Lemma \ref{lemmainterpolation} but directly the definition \fref{defnorme} of $\mathcal E_{2L+2}$, and hence the dependence of the constant in $M$ is explicit, and this will be crucial for the analysis.} from \fref{init3h}, \fref{estphimlto}
$$
\nonumber \left| (H^L \varepsilon,H\Phi_M) \right| \lesssim \| H^{L+1}\varepsilon \|_{L^2}\sqrt{\log M}=\sqrt{\log M\mathcal E_{2L+2}}
$$
and the remaining nonlinear term is estimated using the Hardy bounds of Appendix A:
$$\left|\left(-\lsl\Lambda \e+L(\e)+N(\e),H^L\Phi_M\right)\right|\lesssim M^Cb_1(\sqrt{\mathcal E_{2L+2}}+|D(t)|).$$
We inject these estimates into \fref{modequone} and conclude from \fref{nondgeener} and the local estimate \fref{fluxcomputationonebis}:
\bea
\label{estblone}
\nonumber&& \left|(b_L)_s+(2L-1+c_{b_1})b_1b_L\right|=  \frac{\sqrt{\log M\mathcal E_{2L+2}}}{\log M}+M^Cb_1|D(t)|+M^Cb_1^{L+\frac 32}\\
& \lesssim &  \frac{1}{\sqrt{\log M}} \left( \sqrt{ \mathcal E_{2L+2}} +  \frac{b_1^{L+1}}{|\log b_1|}  \right)+M^Cb_1|D(t)|.
\eea

{\bf step 2} Degeneracy of the law for $\lambda$ and $(b_k)_{1\leq k\leq L-1}$. We now take the inner product of \fref{eqepsilon} with $H^k\Phi_M$, $0\leq k\leq L-1$ and obtain:
\bea
\label{veovoheoehe}
\nonumber (\widetilde{\mbox{Mod}}(t),H^k\Phi_M)  & = &   -(\Psit_b,H^k\Phi_M)-(H^{k+1}\e,H\Phi_M)\\
& - &  \left(-\lsl\Lambda \e-L(\e)+N(\e),H^k\Phi_M\right).
\eea
Note first that the choice of orthogonality conditions \fref{ortho} gets rid of the linear term in $\varepsilon$: $$\forall 0\leq k\leq L-1, \ \ (H^{k+1}\e,\Phi_M)=0.$$ 
Next, we compute from \fref{defmodtuntilde}, the localization $\mbox{Supp}(\Phi_M)\subset[0,2M]$ from \fref{defdirection} and the identities \fref{nondgeener}, \fref{orthohiM}, \fref{vnkoeoueo}:
 \bee
&&\nonumber \left(H^k\left(\widetilde{\mbox{Mod}}(t)\right),\Phi_M\right) = -\left(b+\frac{\lambda_s}{\lambda}\right) \left(H^k\Lambda \qbt , \Phi_M\right)\\
& + &  \Sigma_{i=1}^{L}\left[(b_i)_s+(2i-1+c_{b_1})b_1b_i-b_{i+1}\right]\left(\tilde{T}_i+\chi_{B_1}\Sigma_{j=i+1}^{L+2}\frac{\partial S_j}{\partial b_i},H^k\Phi_M\right)\\
& = & (\Lambda Q,\Phi_M)\left\{\begin{array}{ll}-(\lsl+b_1)\ \ \mbox{for}\ \ k=0\\ 
(-1)^k((b_k)_s+(2k-1+c_{b_1})b_1b_k-b_{k+1})\ \ \mbox{for}\ \ 1\leq k\leq L-1 \end{array}\right .\\
& + & O\left(M^Cb_1| D(t)|\right).
\eee
Nonlinear terms are easily estimated using the Hardy bounds:
$$\left|\left(-\lsl\Lambda \e+L(\e)+N(\e),H^k\Phi_M\right)\right|\lesssim M^Cb_1(\sqrt{\mathcal E_{2L+2}}+|D(t)|)\lesssim b_1^{L+\frac32}+b_1M^C|D(t)|.$$
Injecting this bound into \fref{veovoheoehe} together with the local bound \fref{fluxcomputationonebis} yields the first bound: 
\be
\label{ceheohoeh}
D(t) \lesssim b_1^{L+\frac32}
\ee
and \fref{parameters} is proved. Injecting this bound into \fref{estblone} yields \fref{parameterspresicely}.
\end{proof}


    \subsection{Improved modulation equation fo $b_L$}
  

Observe that \fref{parameterspresicely}, \fref{init3hb} yield the pointwise bound $$\left| (b_L)_s + (2L-1+c_{b_1})b_1b_L \right| \lesssim \frac{1}{\sqrt{\log M}} \left( \sqrt{ \mathcal E_{2L+2}} +  \frac{b_1^{L+1}}{|\log b_1|}  \right)\lesssim \frac{b_1^{L+1}}{|\log b_1|}$$ which is worse than \fref{parameters} and critical to close \fref{controlunstable}. We claim that a $|\log b_1|$ is easily gained up to an oscillation in time.

\begin{lemma}[Improved control of $b_L$]
\label{improvedbl}
Let  $B_\delta=B_0^{\delta}$ and
\be
\label{defbtildel}
\tilde{b}_L=b_L+\frac{(-1)^L(H^L\e,\chi_{B_\delta}\Lambda Q)}{4\delta|\log b_1|},
\ee
then 
\be
\label{poitwidediff}
|\tilde{b}_L-b_L|\lesssim b_1^{L+\frac 12}
\ee
and $\tilde{b}_L$ satisfies the pointwise differential equation:
\be
\label{esteqbrildl}
|(\tilde{b}_{L})_s+(2L-1+c_{b_1})b_1\tilde{b}_L|\lesssim \frac{C(M)}{\sqrt{|\log b_1|}}\left[\sqrt{\mathcal E_{2L+2}}+\frac{b_1^{L+1}}{|\log b_1|}\right].
\ee
\end{lemma}
  
\begin{proof}[Proof of Lemma \ref{improvedbl}] We commute \fref{eqepsilon} with $H^L$ and take the scalar product with $\chi_{B_{\delta}}\Lambda Q$ for some small enough universal constant $0<\delta\ll 1$. This yields:
\bee
&&\frac{d}{ds}\left\{(H^L\e,\chi_{B^{\delta}}\Lambda Q)\right\}-(H^L\e,\Lambda Q\pa_s(\chi_{B^{\delta}}))\\
& = & -(H^{L+1}\e,\chi_{B^{\delta}}\Lambda Q)+\lsl(H^{L}\Lambda \e,\chi_{B_\delta}\Lambda Q)+(F-\widetilde{\mbox{Mod}},H^{L}\chi_{B_{\delta}}\Lambda Q).
\eee
The linear term is estimated by Cauchy Schwarz:
$$|(H^{L+1}\e,\chi_{B^{\delta}}\Lambda Q)|\lesssim C(M)\sqrt{|\log b_1|}\sqrt{\mathcal E_{2L+2}}.$$ We similarily estimate using \fref{parameters}:
\bee
&&\left|(H^L\e,\Lambda Q\pa_s(\chi_{B^{\delta}}))\right|+\left|\lsl(H^{L}\Lambda \e,\chi_{B_\delta}\Lambda Q)\right|\\
& \lesssim & C(M)\frac{|(b_1)_s|}{b_1}\frac{1}{b_1^{C\delta}}\sqrt{\mathcal E_{2L+2}}+\frac{b_1}{b_1^{C\delta}}C(M)\sqrt{\mathcal E_{2L+2}}\lesssim  \sqrt{|\log b_1|}\sqrt{\mathcal E_{2L+2}}.
\eee
The estimate on the error terms easily follows from the Hardy bounds:
$$|(L(\e),H^{L}\chi_{B^{\delta}}\Lambda Q)|+(N(\e),H^{L}\chi_{B^{\delta}}\Lambda Q)|\lesssim \frac{b_1}{b_1^{C\delta}}C(M)\sqrt{\mathcal E_{2L+2}}\lesssim  \sqrt{|\log b_1|}\sqrt{\mathcal E_{2L+2}}.$$ We further estimate from \fref{fluxcomputationonebis}:
$$|(H^L\e,\tilde{\Psi}_b)|\lesssim \frac{b_1^{L+3}}{b_1^{C\delta}}C(M)\sqrt{\mathcal E_{2L+2}}\lesssim  \sqrt{|\log b_1|}\sqrt{\mathcal E_{2L+2}}.$$
We now compute from \fref{parameters}, \fref{defmodtuntilde}:
\bee
&&-(\widetilde{\mbox{Mod}},H^{L}\chi_{B_{\delta}}\Lambda Q)=O\left(\frac{b_1^{L+\frac 32}}{b_1^{C\delta}}\right)\\
& + & \left[(b_L)_s+(2L-1+c_{b_1})b_1b_L\right]\left(H^L\tilde{T}_L+\Sigma_{j=L+1}^{L+2}H^L\left[\chi_{B_1}\frac{\partial S_j}{\pa_{b_L}}\right],\chi_{B_{\delta}}\Lambda Q\right)\\
& = & (-1)^L\left[(b_L)_s+(2L-1+c_{b_1})b_1b_L\right]\left[(\Lambda Q,\chi_{B_\delta}\Lambda Q)+O\left(b_1^{1-C\delta}\right)\right]+O\left(\frac{b_1^{L+\frac 32}}{b_1^{C\delta}}\right)\\
& = & (-1)^L\left[(b_L)_s+(2L-1+c_{b_1})b_1b_L\right]4\delta|\log b_1|+O\left(\sqrt{|\log b_1|}\sqrt{\mathcal E_{2L+2}}+b_1^{L+1}\right).
\eee
The collection of above bounds yields the preliminary estimate:
\bea
\label{calculinter}
\nonumber &&\left|\frac{d}{ds}\left\{(H^L\e,\chi_{B^{\delta}}\Lambda Q)\right\}+ (-1)^L\left[(b_L)_s+(2L-1+c_{b_1}))b_1b_L\right]4\delta|\log b_1|\right|\\
& \lesssim &  C(M)\sqrt{|\log b_1|}\left[\sqrt{\mathcal E_{2L+2}}+\frac{b_1^{L+1}}{|\log b_1|}\right]
\eea
We estimate in brute force from \fref{defbtildel}: $$|\tilde{b}_L-b_L|\lesssim |\log b_1|^Cb_1^{L+1-C\delta}\lesssim b_1^{L+\frac 12}$$ and we therefore rewrite \fref{calculinter} using \fref{rougboundpope}:
\bee
&&|(\tilde{b}_{L})_s+(2L-1+c_{b_1})b_1\tilde{b}_L|\\
& \lesssim&  |(H^L\e,\chi_{B_\delta}\Lambda Q)|\left|\frac{d}{ds}\left\{\frac{1}{4\delta\log b_1}\right\}\right|+ \frac{C(M)\sqrt{|\log b_1|}}{|\log b_1|}\left[\sqrt{\mathcal E_{2L+2}}+\frac{b_1^{L+1}}{|\log b_1|}\right]\\
& \lesssim & b_1^{1-C\delta}\sqrt{\matchal E_{2L+2}}+ \frac{C(M)}{\sqrt{|\log b_1|}}\left[\sqrt{\mathcal E_{2L+2}}+\frac{b_1^{L+1}}{|\log b_1|}\right]
\eee
and \fref{esteqbrildl} is proved.
 \end{proof}

\subsection{The Lyapounov monotonicity}


We now turn to the core of the argument which is the derivation of a suitable Lyapounov functional for the $\mathcal E_{2L+2}$ energy. 

\begin{proposition}[Lyapounov monotonicity]
\label{AEI2}
There holds:
\bea
\label{monoenoiencle}
\nonumber &&\frac{d}{dt} \left\{\frac{1}{\lambda^{4L+2}}\left[\mathcal E_{2L+2}+O\left(b^{\frac 45}_1\frac{b_1^{2L+2}}{|\log b_1|^2}\right)\right]\right\}\\
 & \leq & C\frac{ b_1}{\lambda^{4L+4}}\left[ \frac{\mathcal E_{2L+2}}{\sqrt{\log M}}+\frac{b_1^{2L+2}}{|\log b_1|^2}+\frac{b_1^{L+1}\sqrt{\mathcal E_{2L+2}}}{|\log b_1|} \right]
\eea
for some universal constant $C>0$ independent of $M$ and of the bootstrap constant $K$ in \fref{init2h}, \fref{init3h}.
\end{proposition}

\begin{proof}[Proof of Proposition \ref{AEI2}] {\bf step 1} Suitable derivatives. We define the derivatives of $w$ associated with the linearized Hamiltonian $H_\l$:  
$$w_1 = A_{\lambda}w, \ \ w_{k+1}=\left\{\begin{array}{ll} A^*_{\l}w_k\ \ \mbox{for k odd}\\ A_\l w_k\ \ \mbox{for k even}
\end{array}\right., \ \ 1\leq k \leq 2L+1$$
and its renormalized version $$\e_1 = A\e, \ \ \e_{k+1}=\left\{\begin{array}{ll} A^*\e_k\ \ \mbox{for k odd}\\ A \e_k\ \ \mbox{for k even}
\end{array}\right., \ \ 1\leq k \leq 2L+1.$$ 
We compute from \fref{eqenwini}:
\be
\label{eqtwol}
\partial_{t} w_{2L} + H_{\lambda} w_{2L} =  [\pa_t,H_\l^L]w + H^L_{\lambda} \left(\frac 1{\lambda^2} \mathcal F_{\lambda}\right)
\ee
\be
\label{eqw3}
\partial_{t} w_{2L+1}+ \tilde{H}_{\lambda} w_{2L+1} = \frac{\partial_{t} Z_{\lambda}}{r} w_{2L} + A_{\lambda}\left( [\pa_t,H_\l^L]w\right)+ A_\l H^L_\l  \left(\frac 1{\lambda^2} \mathcal F_{\lambda}\right).
\ee
We recall the action of time derivatives on rescaling: 
\be
\label{resaclingaction}
\partial_tv_{\lambda}=\frac{1}{\l^2}\left(\pa_sv-\lsl\Lambda v\right)_{\lambda}.
\ee

{\bf step 2} Modified energy identity. We compute the energy identity on \fref{eqw3} using \fref{resaclingaction}:
\par
\bea
\label{firstestimate}
&&\nonumber  \frac{1}{2}\frac{d}{dt} \mathcal E_{2L+2} =  \frac{1}{2} \frac{d}{d t}\left\{ \int \tilde{H}_{\lambda}w_{2L+1} w_{2L+1} \right\}= \int \tilde{H}_{\lambda} w_{2L+1} \partial_{t} w_{2L+1} + \int \frac{\partial_{t} \tilde{V}_\l}{2r^2} w_{2L+1}^2\\
\nonumber & = & - \int (\tilde{H}_\l w_{2L+1})^2+b_1\int \frac{(\Lambda \tilde{V})_\l}{2\l^2r^2} w_{2L+1}^2-\left(\lsl+b_1\right)\int \frac{(\Lambda \tilde{V})_\l}{2\l^2r^2} w_{2L+1}^2 \\
& + &\int \tilde{H}_\l w_{2L+1}\left[\frac{\partial_{t} Z_{\lambda}}{r} w_{2L} +  A_{\lambda}\left( [\pa_t,H_\l^L]w\right)+A_\l H^L_{\lambda}\left( \frac{1}{\lambda^2}\mathcal F_{\lambda}\right)\right].
\eea
We further compute from \fref{eqtwol}, \fref{eqw3}:
\bee
&&\frac{d}{dt}\left\{\int \frac{b_1(\Lambda Z)_\l}{\l^2r}w_{2L+1}w_{2L}\right\} =  \int \frac{d}{dt}\left(\frac{b_1(\Lambda Z)_\l}{\l^2r}\right)w_{2L+1}w_{2L}\\
& + & \int \frac{b_1(\Lambda Z)_\l}{\l^2r}w_{2L}\left[-\tilde{H}_\l w_{2L+1}+\frac{\partial_{t} Z_{\lambda}}{r} w_{2L} + A_{\lambda}\left( [\pa_t,H_\l^L]w\right)+A_\l H ^L_{\lambda}\left( \frac{1}{\lambda^2}\mathcal F_{\lambda}\right)\right]\\
& + & \int\frac{b_1(\Lambda Z)_\l}{\l^2r}w_{2L+1}\left[-A_\l^*w_{2L+1}+ [\pa_t,H_\l^L]w + H^L_{\lambda} \left(\frac 1{\lambda^2} \mathcal F_{\lambda}\right)\right]
\eee
We now integrate by parts to compute using \fref{defpotential}:
\bee
&&\int\frac{b_1(\Lambda Z)_\l}{\l^2r}w_{2L+1} A_\l^*w_{2L+1} =  \frac{b_1}{\l^{4L+4}}\int \frac{\Lambda Z}{y}\e_{2L+1}A^*\e_{2L+1}\\
& = & \frac{b_1}{\l^{4L+4}}\int\frac{2(1+Z)\Lambda Z-\Lambda^2Z}{2y^2}\e_{2L+1}^2=  \frac{b_1}{\l^{4L+4}}\int\frac{\Lambda\tilde{V}}{2y^2}\e_{2L+1}^2=b_1 \int \frac{(\Lambda \tilde{V})_\l}{2\l^2r^2} w_{2L+1}^2.
\eee
Injecting this into the energy identity \fref{firstestimate} yields the modified energy identity:
\bea
\label{modifeideenrgy}
\nonumber &&\frac{1}{2}\frac{d}{dt}\left\{ \mathcal E_{2L+2}+2\int \frac{b_1(\Lambda Z)_\l}{\l^2r}w_{2L+1}w_{2L}\right\} =  - \int (\tilde{H}_\l w_{2L+1})^2\\
\nonumber & - & \left(\lsl+b_1\right)\int \frac{(\Lambda \tilde{V})_\l}{2\l^2r^2} w_{2L+1}^2+\int \frac{d}{dt}\left(\frac{b_1(\Lambda Z)_\l}{\l^2r}\right)w_{2L+1}w_{2L}\\
\nonumber & + & \int \tilde{H}_\l w_{2L+1}\left[\frac{\partial_{t} Z_{\lambda}}{r} w_{2L} +A_{\lambda}\left( [\pa_t,H_\l^L]w\right)+A_\l H^L_{\lambda}\left( \frac{1}{\lambda^2}\mathcal F_{\lambda}\right)\right]\\
\nonumber & + & \int \frac{b_1(\Lambda Z)_\l}{\l^2r}w_{2L}\left[-\tilde{H}_\l w_{2L+1}+\frac{\partial_{t} Z_{\lambda}}{r} w_{2L} + A_{\lambda}\left( [\pa_t,H_\l^L]w\right)+A_\l H^L_{\lambda}\left( \frac{1}{\lambda^2}\mathcal F_{\lambda}\right)\right]\\
& + & \int\frac{b_1(\Lambda Z)_\l}{\l^2r}w_{2L+1}\left[ [\pa_t,H_\l^L]w + H^L_{\lambda} \left(\frac 1{\lambda^2} \mathcal F_{\lambda}\right)\right]
\eea
We now aim at estimating all terms in the RHS of \fref{modifeideenrgy}. All along the proof, we shall make an implicit use of the coercitivity estimates of Lemma \ref{coerchtilde} and Lemma \ref{lemmainterpolation}.\\

{\bf step 3} Lower order quadratic terms. We treat the lower order quadratic terms in \fref{modifeideenrgy} using dissipation. Indeed, we have from \fref{comportementz}, \fref{comportementv}, \fref{rougboundpope}  the bounds:
\bea
\label{nveknvneo}
|\partial_t Z_{\lambda}| + |\partial_t V_{\lambda}| \lesssim \frac{b_1}{\lambda^2} \left( |\Lambda Z| + |\Lambda V|\right)_{\lambda} \lesssim \frac{b_1}{\lambda^2}\frac{y^2}{1+y^4}.
\eea
We moreover claim the bound:
\be
\label{commutator}
\int\frac{\left([\pa_t,H_\l^L]w\right)^2}{\l^2(1+y^2)}+\int\left|A_{\lambda}\left( [\pa_t,H_\l^L]w\right)\right|^2\lesssim C(M)\frac{b_1^2}{\l^{4L+4}}\matchal E_{2L+2}
\ee
which is proved in Appendix \ref{proofcommutator}. We conclude from Cauchy Schwartz, the rough bound \fref{rougboundpope} and Lemma \ref{lemmainterpolation}:
\bee
&&\int\left|\tilde{H}_\l w_{2L+1}\left[\frac{\partial_{t} Z_{\lambda}}{r} w_{2L} + \int A_{\lambda}\left( [\pa_t,H_\l^L]w\right)\right]\right|+\int|\tilde{H}_\l w_{2L+1}|\left|\frac{b(\Lambda Z)_\l}{\l^2r}w_{2L}\right|\\
& \leq &\frac{1}{2}\int|\tilde{H}_\l w_{2L+1}|^2+\frac{b_1^2}{\l^{4L+4}}\left[\int \frac{\varepsilon_{2L}^{2}}{1+y^6}+C(M)\matchal E_{2L+2}\right]\\
& \leq &\frac{1}{2}\int|\tilde{H}_\l w_{2L+1}|^2+\frac{b_1}{\l^{4L+4}}C(M)b_1\matchal E_{2L+2}.
\eee
All other quadratic terms are lower order by a factor $b_1$ using again \fref{rougboundpope}, \fref{commutator}, \fref{parameters} and Lemma \ref{lemmainterpolation}:
\bee
&&\left|\lsl+b_1\right|\int \left|\frac{(\Lambda \tilde{V})_\l}{2\l^2r^2} w_{2L+1}^2\right|+\int\left|\frac{b_1(\Lambda Z)_\l}{\l^2r}w_{2L}\left[\frac{\partial_{t} Z_{\lambda}}{r} w_{2L} + A_{\lambda}\left( [\pa_t,H_\l^L]w\right)\right]\right|\\
& + & \int\left|\frac{b_1(\Lambda Z)_\l}{\l^2r}w_{2L+1}[\pa_t,H^L_\l]w\right|+\left|\int \frac{d}{dt}\left(\frac{b_1(\Lambda Z)_\l}{\l^2r}\right)w_{2L+1}w_{2L}\right|\\
& \lesssim & \frac{b_1^2}{\l^{4L+4}}\left[\int\frac{\e_{2L+1}^2}{1+y^4}+\int \frac{\varepsilon_{2L}^2}{1+y^6}+C(M)\mathcal E_{2L+2}\right]\lesssim \frac{b_1}{\l^{4L+4}}C(M)b_1\matchal E_{2L+2}.
\eee 
We similarily estimate the boundary term in time using \fref{interpolationboundlossy}: $$\left|\int \frac{b_1(\Lambda Z)_\l}{\l^2r}w_{2L+1}w_{2L}\right|\lesssim \frac{b_1}{\l^{4L+2}}\left[\int\frac{\e_{2L+1}^2}{1+y^2}+\int\frac{\e_{2L}^2}{1+y^4}\right]\lesssim \frac{b_1}{\l^{4L+2}}|\log b_1|^Cb_1^{2L+2}.$$
We inject these estimates into \fref{modifeideenrgy} to derive the preliminary bound:
\bea
\label{neoheohohe}
  &&\frac{1}{2}\frac{d}{dt}\left\{ \frac{1}{\l^{4L+2}}\left[\mathcal E_{2L+2}+O\left(b^{\frac 45}_1\frac{b^{2L+2}}{|\log b|^2}\right)\right]\right\}\leq -\frac12\int(\tilde{H}_\l w_{2L+1})^2\\
\nonumber  & + & \int \tilde{H}_\l w_{2L+1}A_\lambda H^L_\l\left( \frac{1}{\lambda^2}\mathcal F_{\lambda}\right)+\int H^L_{\lambda} \left(\frac 1{\lambda^2} \mathcal F_{\lambda}\right)\left[\frac{b_1(\Lambda Z)_\l}{\l^2r}w_{2L+1}+A^*_{\l}\left( \frac{b_1(\Lambda Z)_\l}{\l^2r}w_{2L}\right)\right]\\
 \nonumber & + &\frac{b_1}{\l^{4L+4}}\sqrt{b_1}b_1^{2L+2}
\eea
with constants independent of $M$ for $|b|<b^*(M)$ small enough.\\
We now estimate all terms in the RHS of \fref{neoheohohe}.\\

{\bf step 4} Further use of dissipation. Let us introduce the decomposition from \fref{eqepsilon}, \fref{defF}:
\be
\label{fprcingdecomp}
\mathcal F=\mathcal F_0+\mathcal F_1, \ \ \mathcal F_0=-\Psit_b-\widetilde{Mod}(t), \ \ \mathcal F_1=L(\e)-N(\e).
\ee
The first term in the RHS of \fref{neoheohohe} is estimated after an integration by parts:
\bea
\label{oneone}
\nonumber&& \left| \int \tilde{H}_\l w_{2L+1}A_\lambda H^L_\l\left( \frac{1}{\lambda^2}\mathcal F_{\lambda}\right)\right|\\
\nonumber & \leq & \frac{C}{\l^{4L+4}}\|A^*\e_{2L+1}\|_{L^2}\|H^{L+1}\mathcal F_0\|_{L^2}+\frac14\int|\tilde{H}_\l w_{2L+1}|^2+\frac{C}{\l^{4L+4}}\int|AH^L\mathcal F_1|^2\\
& \leq & \frac{C}{\l^{4L+4}}\left[\|H^{L+1}\mathcal F_0\|_{L^2}\sqrt{\mathcal E_{2L+2}}+\|AH^L\mathcal F_1\|_{L^2}^2\right]+\frac14\int|\tilde{H}_\l w_{2L+1}|^2
\eea
for some universal constant $C>0$ independent of $M$.\\
The last two terms in \fref{neoheohohe} can be estimated in brute force from Cauchy Schwarz:
\bea
\label{onetwo}
\nonumber \left|\int H^L_{\lambda} \left(\frac 1{\lambda^2} \mathcal F_{\lambda}\right)\frac{b_1(\Lambda Z)_\l}{\l^2r}w_{2L+1}\right| & \lesssim&  \frac{b_1}{\l^{4L+4}}\left(\int\frac{1+|\log y|^2}{1+y^4}|H^L\mathcal F|^2\right)^{\frac12}\left(\int\frac{\e_{2L+1}^2}{y^2(1+|\log y|^2)}\right)^{\frac12}\\
& \lesssim & \frac{b_1}{\l^{4L+4}}\sqrt{\mathcal E_{2L+2}}\left(\int\frac{1+|\log y|^2}{1+y^4}|H^L\mathcal F|^2\right)^{\frac12}
\eea
where constants are independent of $M$ thanks to the estimate \fref{coercwthree} for $\e_{2L+1}$. Similarily:
\bea
\label{onethree}
 &&\left|\int H^L_{\lambda} \left(\frac 1{\lambda^2} \mathcal F_{\lambda}\right)A^*_{\l}\left( \frac{b_1(\Lambda Z)_\l}{\l^2r}w_{2L}\right)\right|\\
\nonumber &  \lesssim &  \frac{b_1}{\l^{4L+4}}\left(\int\frac{1+|\log y|^2}{1+y^2}|AH^L\mathcal F|^2\right)^{\frac12}\left(\int\frac{\e_{2L}^2}{(1+y^4)(1+|\log y|^2)}\right)^{\frac12}\\
\nonumber & \lesssim & \frac{b_1}{\l^{4L+4}}C(M)\sqrt{\mathcal E_{2L+2}}\left(\int\frac{1+|\log y|^2}{1+y^2}|AH^L\mathcal F_0|^2+\int|AH^L\mathcal F_1|^2\right)^{\frac12}
\eea
We now claim the bounds:
\be
\label{crucialboundthree}
\int\frac{1+|\log y|^2}{1+y^4}|H^L\mathcal F|^2\lesssim \frac{b_1^{2L+2}}{|\log b_1|^2}+\frac{\mathcal E_{2L+2}}{\log M},
\ee
\be
\label{weigheivbiovheo}
\int\frac{1+|\log y|^2}{1+y^2}|AH^L\mathcal F_0|^2 \lesssim \delta(\alpha^*)\left[\frac{b_1^{2L+2}}{|\log b_1|^2}+\mathcal E_{2L+2}\right],
\ee
\be
\label{cnofooeeo}
\int|H^{L+1}\mathcal F_0|^2\lesssim b_1^2\left[\frac{b_1^{2L+2}}{|\log b_1|^2}+\frac{\mathcal E_{2L+2}}{\log M}\right],
\ee
\be
\label{crucialboundtwo}
\int |AH^L\mathcal F_1|^2\lesssim b_1\left[\frac{b_1^{2L+2}}{|\log b_1|^2}+\frac{\mathcal E_{2L+2}}{\log M}\right],
\ee
with all $\lesssim$ constants independent of $M$ for $|b|<\alpha^*(M)$ small enough, and where $$\delta(\alpha^*)\to 0\ \ \mbox{as} \ \ \alpha^*(M)\to 0.$$ Injecting these bounds together with \fref{oneone}, \fref{onetwo}, \fref{onethree} into \fref{neoheohohe} concludes the proof of \fref{monoenoiencle}. We now turn to the proof of \fref{crucialboundthree}, \fref{weigheivbiovheo}, \fref{cnofooeeo}, \fref{crucialboundtwo}.\\

{\bf step 5} $\Psit_b$ terms. The contribution of $\Psit_b$ terms to \fref{crucialboundthree}, \fref{weigheivbiovheo}, \fref{cnofooeeo} is estimated from \fref{estimathehtwobis}, \fref{controleh4erreursharptildebis} which are at the heart of the construction of $\tilde{Q}_b$ and yield the desired bounds.\\

{\bf step 6} $\widetilde{Mod(t)}$ terms. Recall \fref{defmodtuntilde}:
\bee
\widetilde{Mod}(t)& = & -\left(\lsl+b_1\right)\Lambda \qbt\\
\nonumber & + &   \Sigma_{i=1}^{L}\left[(b_i)_s+(2i-1+c_{b_1})b_1b_i-b_{i+1}\right]\left[\tT_i+\chi_{B_1}\Sigma_{j=i+1}^{L+2}\frac{\partial S_j}{\partial b_i}\right],
\eee
and the notation \fref{defut}.\\
{\it Proof of \fref{cnofooeeo} for $\widetilde{Mod}$}: We recall that $|b_k|\lesssim b_1^k$ and estimate from Lemma \ref{lemmaestimate}:
\bee
\int|H^{L+1}\Lambda \qbt|^2& \lesssim&  \Sigma_{i=1}^L\int|H^{L+1}b_i\Lambda \tt_i|^2+\Sigma_{i=2}^{L+2}\int|H^{L+1}\Lambda \tilde{S}_{i}|^2\\
& \lesssim & \Sigma_{i=1}^Lb_1^{2i}\int_{y\leq 2B_1}\left|\frac{(1+|\log y|^C)y^{2i-1}}{1+y^{2L+2}}\right|^2+\Sigma_{i=2}^{L+1}b_1^{2i}+\frac{b_1^{2L+4}}{b_1^2|\log b_1|^2}\\
& \lesssim & b_1^2.
\eee
We then use the cancellation $H^{L+1}T_i=0$ for $1\leq i\leq L$ to estimate:
$$\Sigma_{i=1}^L\int |H^{L+1}\tilde{T}_i|^2\lesssim \Sigma_{i=1}^L\int_{B_1\leq y\leq 2B_1}\left|\frac{y^{2i-1}}{y^{2L+2}}\right|^2\lesssim b_1^2.$$
Then using Lemma \ref{lemmaestimate} again\footnote{this is where we used the logarithmic gain \fref{esthooeh} induced by \fref{defgby}.}: for $1\leq i\leq L$, 
\bee
\Sigma_{j=i+1}^{L+2}\int\left|H^{L+1}\left[\chi_{B_1}\frac{\partial S_j}{\partial b_i}\right]\right|^2 & \lesssim & \Sigma_{j=i+1}^{L+1}b_1^{2(j-i)}+\frac{b_1^{2(L+2-i)}}{b_1^2|\log b_1|^2}\lesssim b_1^2.
\eee
We thus obtain from Lemma \ref{modulationequations} the expected bound:
$$\int|H^{L+1}\widetilde{Mod}|^2\lesssim b_1^2|D(t)|^2\lesssim b_1^2\left[\frac{\mathcal E_{2L+2}}{|\log M|}+\frac{b_1^{2L+2}}{|\log b_1|^2}\right].$$
{\it Proof of \fref{crucialboundthree} for $\widetilde{Mod}$}: We use Lemma \ref{lemmaestimate} to derive the rough bound:
\bee
\int\frac{1+|\log y|^2}{1+y^4}|H^{L}\Lambda \qbt|^2& \lesssim & \Sigma_{i=1}^L\int\frac{1+|\log y|^2}{1+y^4}|H^{L}b_i\Lambda \tt_i|^2+\Sigma_{i=2}^{L+2}\int\frac{1+|\log y|^2}{1+y^4}|H^{L}\Lambda \tilde{S}_{i}|^2\\
& \lesssim & \Sigma_{i=1}^Lb_1^{2i}\int_{y\leq 2B_1}\frac{1+|\log y|^C}{1+y^4}\left|\frac{y^{2i-1}}{1+y^{2L}}\right|^2+\Sigma_{i=2}^{L+1}b_1^{2i}|\log b_1|^3\\
& + & b_1^{2L+4}\int_{y\leq 2B_1}\frac{1+|\log y|^2}{1+y^4}\left|\frac{1+y^{2(L+2)-1}}{1+y^{2L}}\right|^2\\
& \lesssim & 1.
\eee
Next:
$$\Sigma_{j=1}^L\frac{1+|\log y|^2}{1+y^4}|H^L\tt_i|^2\lesssim \Sigma_{j=1}^L\int_{y\leq 2B_1}\frac{1+|\log y|^C}{1+y^4}\left|\frac{y^{2i-1}}{1+y^{2L}}\right|^2\lesssim 1,$$ and finally using Lemma \ref{lemmaestimate} again: for $1\leq i\leq L$,
\bee
&&\Sigma_{j=i+1}^{L+2}\int\frac{1+|\log y|^2}{1+y^4}\left|H^L\left[\chi_{B_1}\frac{\partial S_j}{\partial b_i}\right]\right|^2\\
& \lesssim & \Sigma_{j=i+1}^{L+1}b_1^{2(j-i)}|\log b_1|^2 +   b_1^{2(L-i)+4}\int_{y\leq 2B_1}\frac{1+|\log y|^2}{1+y^4}\left|\frac{1+y^{2(L+2)-1}}{1+y^{2L}}\right|^2\\
& \lesssim & 1.
\eee
We thus obtain from Lemma \ref{modulationequations} the expected bound:
$$\int\frac{1+|\log y|^2}{1+y^4}|H^{L}\widetilde{Mod}|^2\lesssim |D(t)|^2\lesssim \frac{\mathcal E_{2L+2}}{|\log M|}+\frac{b_1^{2L+2}}{|\log b_1|^2}.$$

{\it Proof of \fref{weigheivbiovheo} for $\widetilde{Mod}$}: We use Lemma \ref{lemmaestimate} to estimate:
\bee
&&\int\frac{1+|\log y|^2}{1+y^2}|AH^{L}\Lambda \qbt|^2\\
 &\lesssim & \Sigma_{i=1}^L\int\frac{1+|\log y|^2}{1+y^2}|H^{L}b_i\Lambda \tt_i|^2+\Sigma_{i=2}^{L+2}\int\frac{1+|\log y|^2}{1+y^2}|AH^{L}\Lambda \tilde{S}_{i}|^2\\
& \lesssim & \Sigma_{i=1}^Lb_1^{2i}\int_{y\leq 2B_1}\frac{1+|\log y|^2}{1+y^2}\left|\frac{y^{2i-1}}{1+y^{2L}}\right|^2+\Sigma_{i=2}^{L+1}b_1^{2i}|\log b_1|^3\\
& + & b_1^{2L+4}\int_{B_1\leq y\leq 2B_1}\frac{1+|\log y|^2}{1+y^2}\left|\frac{1+y^{2(L+2)-1}}{1+y^{2L+1}}\right|^2\\
& \lesssim & b_1^2.
\eee
Next using the cancellation $AH^LT_i=0$, $1\leq i\leq L$, and 
$$\Sigma_{j=1}^L\frac{1+|\log y|^2}{1+y^2}|AH^L\tt_i|^2\lesssim \Sigma_{j=1}^L\int_{B_1\leq y\leq 2B_1}\frac{1+|\log y|^C}{1+y^2}\left|\frac{y^{2i-1}}{1+y^{2L}}\right|^2\lesssim b_1|\log b_1|^C.$$ 
and finally using Lemma \ref{lemmaestimate} again: for $1\leq i\leq L$,
\bee
&&\Sigma_{j=i+1}^{L+2}\int\frac{1+|\log y|^2}{1+y^2}\left|AH^L\left[\chi_{B_1}\frac{\partial S_j}{\partial b_i}\right]\right|^2\\
& \lesssim & \Sigma_{j=i+1}^{L+1}b_1^{2(j-i)}|\log b_1|^C +   b_1^{2(L-i)+4}\int_{y\leq 2B_1}\frac{1+|\log y|^2}{1+y^2}\left|\frac{1+y^{2(L+2)-1}}{1+y^{2L+1}}\right|^2\\
& \lesssim & b_1.
\eee
We thus obtain from Lemma \ref{modulationequations} the desired bound:
$$\int\frac{1+|\log y|^2}{1+y^2}|AH^{L}\widetilde{Mod}|^2\leq \sqrt{b_1}|D(t)|^2\lesssim \delta(\alpha^*)\left[\mathcal E_{2L+2}+\frac{b_1^{2L+2}}{|\log b_1|^2}\right].$$

{\bf step 7} Nonlinear term $N(\e)$. {\it Control near the origin $y\leq 1$}: We rewrite from \fref{defNe} and a Taylor Lagrange formula 
\be
\label{formulanesilon}
N(\e)=zN_0(\e), \ \ z=y\left(\frac{\e}{y}\right)^2, \ \  N_0(\e)=\frac{1}{y}\int_0^1(1-\tau)f''(\qbt+\tau\e)d\tau.
\ee
First observe from \fref{Taylororigin} and the Taylor expansion at the origin of $T_i$ given by \fref{defthetak} that 
\be
\label{expansionz}
z=\frac{1}{y}\left[\Sigma_{i=1}^{L+1}c_{i}T_{L+1-i}+r_\e\right]^2=\Sigma_{i=0}^{L}\tilde{c}_iy^{2i+1}+\tilde{r}_{\e}
\ee with from \fref{boundci}, \fref{boundrestebis}:
 $$|\tilde{c}_{i}|\lesssim C(M)\mathcal E_{2L+2},$$
 \be
 \label{expsnionadpatedbis}
|\pa_y^k\tilde{r}_{\e}|\lesssim y^{2L+1-k}|\log y| C(M)\mathcal E_{2L+2}, \ \ 0\leq k\leq 2L+1.
 \ee
 We now let $\tau\in [0,1]$ and $$v_\tau=\qbt+\tau\e,$$ and obtain from Proposition \ref{consprofapproch} and \fref{Taylororigin} the Taylor expansion at the origin:
 \be
 \label{taylororoiign}
 v_\tau=\Sigma_{i=0}^{L}\hat{c}_iy^{2i+1}+\hat{r}_{\e}
 \ee
 with 
 \be
 \label{expsnionadpatedbisbis}
|\hat{c}_i|\lesssim 1, \ \ |\pa_y^k\hat{r}_{\e}|\lesssim y^{2L+1-k}|\log y|, \ \ 0\leq k\leq 2L+1.
 \ee
Recall that $f\in \mathcal C^{\infty}$ with $f^{2k}(0)=0$, $k\geq 0$. We therefore obtain a Taylor expansion $$f''(v_\tau)=\Sigma_{i=1}^{L+1}\frac{f^{(2i+1)}(0)}{i!}v_\tau^{2i-1}+\frac{v^{2L+2}}{(2L+1)!}\int_0^1(1-\sigma)^{2L+1}f^{(2L+4)}(\sigma v_\tau)d\sigma$$ which together with \fref{taylororoiign} ensures an expansion: $$N_0(\e)=\Sigma_{i=0}^{L}\hat{\hat{c}}_iy^{2i}+\hat{\hat{r}}_{\e},$$ $$|\hat{\hat{c}}_i|\lesssim 1, \ \ , \ \ |\pa_y^k\hat{\hat{r}}_{\e}|\lesssim y^{2L-k}|\log y| \ \ 0\leq k\leq 2L+1.$$ Combining this with \fref{expansionz} ensures the expansion: 
\be
\label{decompne}
N(\e)=zN_0(\e)=\Sigma_{i=0}^{L}\tilde{\tilde{c}}_iy^{2i+1}+\tilde{\tilde{r}}_{\e}
\ee with 
$$|\tilde{\tilde{c}}_{i}|\lesssim C(M)\mathcal E_{2L+2}, \ \ |\pa_y^k\tilde{\tilde{r}}_{\e}|\lesssim y^{2L+1-k}|\log y| C(M)\mathcal E_{2L+2}, \ \ 0\leq k\leq 2L+1.$$
Observe that this implies from direct check the bound: 
 \bee
 |\mathcal A^k\tilde{\tilde{r}}_{\e}| &\lesssim & \Sigma_{i=0}^k\frac{\pa_y^i\tilde{\tilde{r}}_{\e}}{y^{k-i}}\lesssim C(M)\mathcal E_{2L+2}\Sigma_{i=0}^k\frac{|\log y|y^{2L+1-i}}{y^{k-i}}\\
 & \lesssim & y^{2L+1-k}|\log y| C(M)\mathcal E_{2L+2}, \ \ 0\leq k\leq 2L+1
 \eee
 We now compute using a simple induction based on the expansions \fref{comportementz}, \fref{comportementv} and the cancellation $A(y)=O(y^2)$ that for $y\leq 1$: 
 \be
 \label{estak}
 \left\{\begin{array}{ll} \mathcal A^{2k+1}\left(\Sigma_{i=0}^{L}\tilde{\tilde{c}}_iy^{2i+1}\right)=\Sigma_{i=k+1}^{L}c_{i,2k+1}y^{2(i-k)}+O(y^{2(L-k)+2}),\\
\mathcal A^{2k+2}\left(\Sigma_{i=0}^{L}\tilde{\tilde{c}}_iy^{2i+1}\right)=\Sigma_{i=k+1}^{L}c_{i,2k+2}y^{2(i-k)-1}+O(y^{2(L-k)+1})
\end{array}\right.
\ee
We conclude from \fref{decompne}: 
\be
\label{controlrogi}
\|\mathcal A^kN(\e)\|_{L^{\infty}(y\leq 1)}\lesssim C(M)\mathcal E_{2L+2}, \ \ 0\leq k\leq 2L+1
\ee and thus in particular the control near the origin:
$$\int_{y\leq 1}\frac{1+|\log y|^2}{1+y^4}|H^LN(\e)|^2+\int_{y\leq 1}|AH^L\mathcal N(\e)|^2\lesssim C(M)\left(\mathcal E_{2L+2}\right)^2\lesssim b_1^2 b_1^{2L+2}.
$$
{\it Control for $y\geq 1$}: We give the detailed proof of \fref{crucialboundtwo}. The proof of \fref{crucialboundthree} follows the exact same lines -with in fact more room- and is left to the reader. Let 
\be
\label{cnekonronroenoer}
\zeta=\frac{\e}{y}, \ \ N_1(\e)=\int_0^1(1-\tau)f''(\qbt+\tau\e)d\tau\ \ \mbox{so that}\ \ N(\e)=\zeta^2N_1. 
\ee
We first estimate from \fref{weightbis}: for $(i,j)\in \Bbb N\times \Bbb N$ with $1\leq i+j\leq 2L+1$,
\bea
\label{zetpovctlinfty}
 \left\|\frac{\pa_y^i\zeta}{y^{j-1}}\right\|^{2}_{L^{\infty}(y\geq 1)} & \lesssim & \Sigma_{k=0}^i\left\|\frac{\pa_y^k\e}{y^{j+i-k}}\right\|^{2}_{L^{\infty}(y\geq 1)}\\
\nonumber  & \lesssim & |\log b_1|^C\left\{\begin{array}{lll}  b_1^{(i+j)\frac{2L}{2L-1}}\ \ \mbox{for}\ \ 1\leq i+j\leq 2L-1,\\ b_1^{2L+1} \ \ \mbox{for}\ \ i+j= 2L, \\  b_1^{2L+2} \ \ \mbox{for}\ \ i+j= 2L+1.\end{array}\right.
\eea
Similarily from \fref{estimatelossy}: for $(i,j)\in \Bbb N\times \Bbb N^*$ with $2\leq i+j\leq 2L+2$,
 \bea
\label{zetpovctlinftybis}
\nonumber \int_{y\geq 1 }\frac{1+|\log y|^C}{1+y^{2j-2}}|\pa_y^i\zeta|^2&\lesssim& \Sigma_{k=0}^i\int_{y\geq 1 }\frac{1+|\log y|^C}{1+y^{2j+2(i-k)}}|\pa^k_y\e|^2\\
& \lesssim&   |\log b_1|^C\left\{ \begin{array}{lll} b_1^{(i+j-1)\frac{2L}{2L-1}} \ \ \mbox{for}\ 2\leq  i+j\leq 2L\\
 b_1^{2L+1}  \ \ \mbox{for}\ \  i+j= 2L+1\\
  b_1^{2L+2}  \ \ \mbox{for}\   i+j=2L+2.
  \end{array}\right.
\eea
Moreover, from the energy bound \fref{init2h}: 
\be
\label{energyboundbis}
\int_{y\geq 1}|\zeta|^2\lesssim 1.
\ee
We now claim the pointwise bound for $y\geq 1$:
\be
\label{pointwisebound}
\forall 1\leq k\leq 2L+1, \ \ |\pa_y^kN_1(\e)|\lesssim |\log b_1|^C\left[\frac{1}{y^{k+1}}+b_1^{\frac{a_k}{2}}\right],
\ee
with 
\be
\label{defakkk}
 \ a_k= \left\{\begin{array}{lll}  k\frac{2L}{2L-1}\ \ \mbox{for}\ \ 1\leq k\leq 2L-1\\ 2L+1\ \ \mbox{for}\ \ k= 2L, \\ 2L+2 \ \ \mbox{for}\ \ k= 2L+1\end{array}\right..
\ee
which is proved below. For $k=0$, we simply need the obvious bound:
\be
\label{novioe}
\|N_1(\e)\|_{L^{\infty}(y\geq 1)}\lesssim 1.
\ee
We then estimate in brute force from \fref{cnekonronroenoer}, \fref{pointwisebound}, \fref{novioe}:
\bee
|AH^LN(\e)|&\lesssim& \Sigma_{k=0}^{2L+1}\frac{|\pa_y^kN(\e)|}{y^{2L+1-k}}\lesssim \Sigma_{k=0}^{2L+1}\frac{1}{y^{2L+1-k}}\Sigma_{i=0}^k|\pa_y^i\zeta^2||\pa_y^{k-i}N_1(\e)|\\
& \lesssim &  \Sigma_{k=0}^{2L+1}\frac{|\pa_y^k\zeta^2|}{y^{2L+1-k}}+\Sigma_{k=1}^{2L+1}\frac{1}{y^{2L+1-k}}\Sigma_{i=0}^{k-1}|\pa_y^i\zeta^2| |\log b_1|^C\left[\frac{1}{y^{k-i+1}}+b_1^{\frac{a_{k-i}}{2}}\right]\\
& \lesssim &  \Sigma_{k=0}^{2L+1}\frac{|\pa_y^k\zeta^2|}{y^{2L+1-k}}+|\log b_1|^C\Sigma_{i=0}^{2L}\frac{|\pa_y^i\zeta^2|}{y^{2L+2-i}} +|\log b_1|^C\Sigma_{k=1}^{2L+1}\Sigma_{i=0}^{k-1} b_1^{\frac{a_{k-i}}{2}}\frac{|\pa_y^i\zeta^2|}{y^{2L+1-k}}\\
& \lesssim & |\log b_1|^C\left[  \Sigma_{k=0}^{2L+1}\frac{|\pa_y^k\zeta^2|}{y^{2L+1-k}}+\Sigma_{k=1}^{2L+1}\Sigma_{i=0}^{k-1} b_1^{\frac{a_{k-i}}{2}}\frac{|\pa_y^i\zeta^2|}{y^{2L+1-k}}\right]
\eee
and hence:
\bee
\int_{y\geq 1}|AH^LN(\e)|^2&\lesssim & |\log b_1|^C \Sigma_{k=0}^{2L+1}\Sigma_{i=0}^k\int_{y\geq 1}\frac{|\pa_y^i\zeta|^2|\pa_y^{k-i}\zeta|^2}{y^{4L+2-2k}}\\
& + &  |\log b_1|^C \Sigma_{k=1}^{2L+1}\Sigma_{i=0}^{k-1} \Sigma_{j=0}^ib_1^{a_{k-i}}\int_{y\geq 1}\frac{|\pa_y^j\zeta|^2|\pa_y^{i-j}\zeta|^2}{y^{4L+2-2k}}.
\eee
We now claim the bounds 
\be
\label{boudntoneere}
\Sigma_{k=0}^{2L+1}\Sigma_{i=0}^k\int_{y\geq 1}\frac{|\pa_y^i\zeta|^2|\pa_y^{k-i}\zeta|^2}{y^{4L+2-2k}}\lesssim |\log b_1|^C b_1^{\delta(L)} b_1^{2L+3}
\ee
\be
\label{boudntoneerebis}
  |\log b_1|^C \Sigma_{k=1}^{2L+1}\Sigma_{i=0}^{k-1} \Sigma_{j=0}^ib_1^{a_{k-i}}\int_{y\geq 1}\frac{|\pa_y^j\zeta|^2|\pa_y^{i-j}\zeta|^2}{y^{4L+2-2k}}\lesssim |\log b_1|^C b_1^{\delta(L)} b_1^{2L+3}
 \ee
for some $\delta(L)>0$, and this concludes the proof of \fref{crucialboundtwo} for $N(\e)$.\\
{\it Proof of \fref{pointwisebound}}: We first extract from Proposition \ref{consprofapproch} the rough bound: 
\be
\label{esnfienneo}
|\pa_y^k\qbt|\lesssim |\log b_1|^C\left[\frac{1}{y^{k+1}}+\Sigma_{i=1}^{2L+2}b_1^iy^{2i-1-k}{\bf 1}_{y\leq 2B_1}\right]\lesssim  \frac{|\log b_1|^C}{y^{k+1}}.
\ee Let then $\tau\in [0,1]$ and $v_{\tau}=\qbt+\tau\e$, we conclude from \fref{esnfienneo}, \fref{weightbis}, \fref{defakkk}:
\be
\label{cneononepeiop}
|\pa_y^kv_{\tau}|\lesssim  |\log b_1|^C\left[\frac{1}{y^{k+1}}+b_1^{\frac{a_k}{2}}\right], \ \ 1\leq k\leq 2L+1, \ \ y\geq 1.
\ee
We therefore estimate $N_1$ through the formula \fref{cnekonronroenoer} using the rough bound $|\pa_v^if|\lesssim 1$ and the Faa di Bruno formula: for $1\leq k\leq 2L+1$,
\bee
|\pa_y^kN_1(\e)|& \lesssim & \int_0^1\Sigma_{m_1+2m_2+\dots+km_k=k}\left|\pa_v^{m_1+\dots+m_k}f(v_{\tau})\right|\Pi_{i=1}^k|\pa^i_yv_{\tau}|^{m_i}d\tau\\
& \lesssim &|\log b_1|^C \Sigma_{m_1+2m_2+\dots+km_k=k}|\Pi_{i=1}^k\left[\frac{1}{y^{i+1}}+b_1^{\frac{a_i}{2}}\right]^{m_i}\\
& \lesssim &  |\log b_1|^C\left[\frac{1}{y^{k+1}}+b_1^{\frac{\alpha_k}{2}}\right].
\eee
To estimate $\alpha_k$ from the definition \fref{defakkk}, we observe that for $k\leq 2L-1$, $i\leq 2L-1$ and thus $$\alpha_k\geq \Sigma_{i=0}^k \frac{2iL}{L-1}m_i=\frac{2kL}{L-1}=a_k.$$ For $k=2L$, we have to treat the boundary term $i=k$, $(m_1,\dots,m_{k-1},m_k)=(0,\dots,0,1)=1$ which yields: $$\alpha_{2L}\geq \min\{\frac{2L(2L)}{2L-1};2L+1\}=2L+1.$$ For $k=2L+1$, we have the the two boundary terms $(m_1,m_2,\dots,m_{k-2},m_{k-1},m_k)=(1,0, \dots,0,1,0)$, $(m_1,\dots,,m_{k-1},m_k)=(0,\dots,0,1)$ which yields:
$$\alpha_{2L+1}\geq \min \{\frac{2L(2L+1)}{2L-1}; 2L+1+\frac{2L}{2L-1};2L+2\}= 2L+2.$$
and \fref{pointwisebound} is proved.\\
\noindent{\it Proof of \fref{boudntoneere}}: Let $0\leq k\leq 2L+1$, $0\leq i\leq k$. Let $I_1=k-i$, $I_2=i$, then we can pick $J_2\in \Bbb N^*$ such that $$ \max\{1;2-i\}\leq J_2\leq \min\{2L+3-k;2L+2-i\}$$ and define $$J_1=2L+3-k-J_2.$$ Then from direct inspection,
$$
(I_1,J_1,I_2,J_2)\in \Bbb N^3\times \Bbb N^*, \ \ \left\{\begin{array}{ll}1\leq I_1+J_1\leq 2L+1, \ \ 2\leq I_2+J_2\leq 2L+2,\\ I_1+I_2+J_1+J_2=2L+3.\end{array}\right. .
$$
Thus 
\bee
A_i& = & \int_{y\geq 1}\frac{|\pa_y^i\zeta|^2|\pa_y^{k-i}\zeta|^2}{y^{4L+2-2k}}= \int_{y\geq 1}\frac{|\pa_y^{I_1}\zeta|^2|\pa_y^{I_2}\zeta|^2}{y^{2J_1-2+2J_2-2}}\lesssim \left\|\frac{\pa_y^{I_1}\zeta}{y^{J_1-1}}\right\|_{L^{\infty(y\geq 1)}}^2\int_{y\geq 1}\frac{|\pa_y^{I_2}\zeta|^2}{y^{2J_2-2}}\\
& \lesssim & |\log b_1|^Cb_1^{d_{i,k}}
\eee
where we now compute the exponent $d_{i,k}$ using \fref{zetpovctlinfty}, \fref{zetpovctlinftybis}:
\begin{itemize} 
\item for $I_1+J_1\leq 2L-1$, $I_2+J_2\leq 2L$, 
$$d_{i,k}= \frac{2L}{2L-1}(I_1+J_1+I_2+J_2-1)=\frac{2L(2L+2)}{2L-1}>2L+3;
$$
\item for $I_1+J_1=2L$, $I_2+J_2=3$,
$$d_{i,k}=2L+1+\frac{2L}{2L-1}(3-1)>2L+3;$$
\item for $I_1+J_1=2L+1$, $I_2+J_2=2$, $$d_{i,k}=2L+2+\frac{2L}{2L-1}>2L+3,$$
\item for $I_2+J_2=2L+1$, $I_1+J_1=2$,
$$d_{i,k}=\frac{2(2L)}{2L-1}+2L+1>2L+3,$$
\item for $I_2+J_2=2L+2$, $I_1+J_1=1$, $$d_{i,k}=2L+2+\frac{2L}{2L-1}>2L+3,$$
\end{itemize}
and \fref{boudntoneere} is proved.\\
{\it Proof of \fref{boudntoneerebis}}: Let $1\leq k\leq 2L+1$, $0\leq j\leq i\leq k-1$. For $k=2L+1$ and $0\leq i=j\leq 2L$, we use the energy bound \fref{energyboundbis} to estimate:
\bee
b_1^{a_{k-i}}\int_{y\geq 1}\frac{|\pa_y^j\zeta|^2|\pa_y^{i-j}\zeta|^2}{y^{4L+2-2k}} & = & b_1^{a_{2L+1-i}}\int_{y\geq 1}|\pa_y^i\zeta|^2|\zeta|^2\lesssim b_1^{a_{2L+1-i}}\|\zeta\|^2_{L^{\infty}(y\geq 1)}\int_{y\ge 1} |\pa_y^i\zeta|^2\\
& \lesssim & b_1^{d_{i,2L+1}}
\eee
with 
\bee
d_{i,2L+1}& =& \left\{\begin{array}{ll}\frac{2L}{2L-1}+2L+2\ \ \mbox{for}\ \ i=0,\\ \frac{2L}{2L-1}+2L+2+\frac{2L}{2L-1}\ \ \mbox{for}\ \ i=1,\\ \frac{2L}{2L-1}+\frac{2L}{2L-1}(i+1-1)+\frac{2L}{2L-1}(2L+1-i)\ \ \mbox{for}\ \ 2\leq i\le 2L\end{array}\right.\\
& >& 2L+3.
\eee
This exceptional case being treated, we let $I_1=j$, $I_2=i-j$ and pick $J_2\in \Bbb N^*$ with $$\max\{1;2-(i-j);2-(k-j)\}\leq J_2\leq \min\{2L+3-k;2L+2-(k-j);2L+2-(i-j)\}.$$ Let $$J_1=2L+3-k-J_2,$$ then from direct check: 
$$
(I_1,J_1,I_2,J_2)\in \Bbb N^3\times \Bbb N^*, \ \ \left\{\begin{array}{ll}1\leq I_1+J_1\leq 2L+1, \ \ 2\leq I_2+J_2\leq 2L+2,\\ I_1+I_2+J_1+J_2=2L+3-(k-i).\end{array}\right. .
$$
Hence:
\bee
b_1^{a_{k-i}}\int_{y\geq 1}\frac{|\pa_y^j\zeta|^2|\pa_y^{i-j}\zeta|^2}{y^{4L+2-2k}}&=&b_1^{a_{k-i}}\int_{y\geq 1}\frac{|\pa_y^{I_1}\zeta|^2|\pa_y^{I_2}\zeta|^2}{y^{2J_2-2+2J_1-2}}\lesssim b_1^{a_{k-i}}\left\|\frac{\pa_y^{I_1}\zeta}{y^{J_1-1}}\right\|^2_{L^{\infty}(y\ge1)}\int_{y\geq 1}\frac{|\pa_y^{I_2}\zeta|^2}{y^{2J_2-2}}\\
& \lesssim & |\log b_1|^Cb_1^{d_{i,j,k}}
\eee
where we now compute the exponent $d_{i,k}$ using \fref{zetpovctlinfty}, \fref{zetpovctlinftybis}, \fref{defakkk}:
\begin{itemize}
\item for $I_1+J_1\leq 2L-1$, $I_2+J_2\leq 2L$, $k-i\leq 2L-1$,
$$d_{i,j,k}=  (k-i)\frac{2L}{2L-1}+(2L+3-(k-i)-1)\frac{2L}{2L-1}=\frac{2L(2L+2)}{2L-1}>2L+3;$$
\item  for $I_1+J_1\leq 2L-1$, $I_2+J_2\leq 2L$, $k-i=2L$,  $$d_{i,j,k}= 2L+1 +(2L+3-2L-1)\frac{2L}{2L-1}>2L+3;$$
\item for $I_1+J_1=2L$, $I_2+J_2=3-(k-i)\geq 2$ and thus $k-i=1$, $I_2+J_2=2$, $$d_{i,j,k}=\frac{2L}{2L-1}+2L+1+\frac{2L}{2L_1}>2L+3;$$
\item for $I_2+J_2=2L+1$, $I_1+J_1=2-(k-i)\geq 1$ and thus $k-i=1$. $I_1+J_1=1$, $$d_{i,j,k}=\frac{2L}{2L-1}+2L+1+\frac{2L}{2L-1}>2L+3,$$
\end{itemize}
and this concludes the proof of \fref{boudntoneerebis}.\\

{\bf step 8} Small linear term $L(\e)$.\\

Let us rewrite from  a Taylor expansion: 
\be
\label{cnoheiohoe}
L(\e)=-\e N_2(\tilde{\alpha}_b),\ \ N_2(\tilde{\alpha}_b)=\frac{f'(Q+\tilde{\alpha}_b)-f'(Q)}{y^2}=\frac{\tilde{\alpha_b}}{y^2}\int_0^1f''(Q+\tau \tilde{\alpha}_b)d\tau.
\ee
{\it Control for $y\leq 1$}: We use a Taylor expansion with the cancellation $f^{2k}(0)=0$, $k\geq 0$. and Proposition \ref{consprofapproch} to ensure for $y\leq1 $ a decomposition $$N_2(\tilde{\alpha}_b)=b_1\left[\Sigma_{i=0}^L\tilde{c}_iy^{2i}+r\right], \ \ |\tilde{c}_i|\lesssim 1, \ \ |\pa_y^kr|\lesssim y^{2L+2-k}, \ \ 0\leq k\leq 2L+1.$$ We combine this with \fref{Taylororigin} and obtain the representation for $y\leq 1$:
\be
\label{represLorigin}
L(\e)=\left[\Sigma_{i=1}^{L+1}c_{i}T_{L+1-i}+r_\e\right]b_1\left[\Sigma_{i=0}^Lc_iy^{2i}+r\right]= b_1\left[\Sigma_{i=1}^{L}\hat{c}_{i}y^{2i-1}+\hat{r}_\e\right]
\ee
with bounds:
 \be
 \label{boundcibis}
 |\hat{c}_{i}|\lesssim C(M)\sqrt{\mathcal E_{2L+2}},
 \ee
 \be
 \label{boundreste}
|\pa_y^k \hat{r}_{\e}|\lesssim y^{2L+1-k}|\log y|C(M)\sqrt{\mathcal E_{2L+2}}, \ \ 0\leq k\leq 2L+1, \ \ y\leq 1.
\ee
We now apply $(\matchal A^k)_{0\leq k\leq 2L+1}$ to \fref{represLorigin} and conclude using \fref{estak}:
\be
\label{coneonienoeg}
\|\mathcal A^kL(\e)\|_{L^{\infty}(y\leq 1)}\lesssim b_1C(M)\sqrt{\mathcal E_{2L+2}}
\ee from which:
$$\int_{y\leq 1}\frac{1+|\log y|^2}{1+y^4}|H^LL(\e)|^2+\int_{y\leq 1}|AH^L L(\e)|^2\lesssim C(M)b_1^2\mathcal E_{2L+2}\lesssim C(M) b_1^2 b_1^{2L+2}.
$$
{\it Control for $y\geq 1$}: We give the detailed proof of \fref{crucialboundtwo}, the proof of \fref{crucialboundthree} follows the exact same lines and is left to the reader. We claim the pointwise bound for $y\geq 1$:
\be
\label{bmpbrpoitwisebound}
\forall0\leq k\leq 2L+1, \ \ |\pa_y^kN_2(\tilde{\alpha}_b)|\lesssim \frac{b_1|\log b_1|^C}{y^{k+1}}.
\ee
which is proved below. This yields from Leibniz rule 
\be
\label{nknnveo}
|\pa_y^k L(\e)|\lesssim \Sigma_{i=0}^k\frac{b_1|\log b_1|^C|\pa_y^i\e|}{y^{k-i+1}}
\ee
 and thus:
\bee
|AH^LL(\e)|&\lesssim &\Sigma_{k=0}^{2L+1}\frac{|\pa_y^kL(\e)|}{y^{2L+1-k}}\lesssim \Sigma_{k=0}^{2L+1}\frac{1}{y^{2L+1-k}}\Sigma_{i=0}^{k}\frac{b_1|\log b_1|^C|\pa_y^i\e|}{y^{k-i+1}}\\
& \lesssim & b_1|\log b_1|^C\Sigma_{i=0}^{2L+1} \frac{|\pa_y^i\e|}{y^{2L+2-i}}.
\eee
We therefore conclude from \fref{weight} with $k=L$:
$$\int_{y\geq 1} |AH^LL(\e)|^2\lesssim b_1^2|\log b_1|^C\Sigma_{i=0}^{2L+1}\int_{y\geq 1}\frac{|\pa_y^i\e|^2}{y^{4L+4-2i}}\lesssim |\log b_1|^Cb_1^{2L+4},$$ and \fref{crucialboundtwo} is proved.\\
{\it Proof of \fref{bmpbrpoitwisebound}}: Let $$N_3=\int_0^1f''(Q+\tau \tilde{\alpha}_b)d\tau.$$ Let $\tilde{v}_\tau=Q+\tau \tilde{\alpha}_b$, $0\leq\tau\leq 1$, we estimate from Proposition \ref{consprofapproch}:
$$|\pa_y^k\tilde{v}_{\tau}|\lesssim  \frac{|\log b_1|^C}{y^{k+1}}, \ \ 1\leq k\leq 2L+1, \ \ y\geq 1,$$ and hence using the Faa di Bruno formula:
\bee
|\pa_y^kN_3(\tilde{\alpha}_b)|& \lesssim & \int_0^1\Sigma_{m_1+2m_2+\dots+km_k=k}\left|\pa_v^{m_1+\dots+m_k}f(\tilde{v}_{\tau})\right|\Pi_{i=1}^k|\pa^i_y\tilde{v}_{\tau}|^{m_i}d\tau\\
& \lesssim &|\log b_1|^C \Sigma_{m_1+2m_2+\dots+km_k=k}\Pi_{i=1}^k\left[\frac{1}{y^{i+1}}\right]^{m_i}\lesssim \frac{|\log b_1|^C}{y^{k+1}}. 
\eee
This yields in particular the rough bound $$|\pa_y^kN_3(\tilde{\alpha}_b)|\lesssim \frac{|\log b_1|^C}{y^{k}}, \ \ y\geq 1,\ \ 0\leq k\leq 2L+1$$
 and hence from the Leibniz rule: 
 \be
 \label{vnonvnonori}
 \left|\pa_y^k\left(\frac{N_3(\tilde{\alpha}_b)}{y^2}\right)\right|\lesssim  \frac{|\log b_1|^C}{y^{k+2}}, \ \ y\geq 1,\ \ 0\leq k\leq 2L+1.
 \ee
 We extract from Proposition \ref{consprofapproch} the rough bound $$|\pa_y^k\tilde{\alpha}_b|\lesssim \frac{|\log b_1|^Cb_1}{y^{k-1}}, \ \ 0\leq k\leq 2L+1$$ and conclude from Leibniz rule: 
 $$|\pa_y^kN_2|\lesssim \Sigma_{i=0}^k|\log b_1|^C\frac{b_1}{y^{i+2}y^{k-i-1}}\lesssim \frac{b_1|\log b_1|^C}{y^{k+1}},$$
 and \fref{bmpbrpoitwisebound} is proved.\\
 
This concludes the proof of \fref{crucialboundthree}, \fref{weigheivbiovheo}, \fref{cnofooeeo}, \fref{crucialboundtwo} and thus of Proposition \ref{AEI2}.

\end{proof}

  
  \section{Closing the bootstrap and proof of Theorem \ref{thmmain}}
  \label{sectionfour}

We are now in position to close the bootstrap bounds of Proposition \ref{bootstrap}. The proof of Theorem \ref{thmmain} will easily follow.


\subsection{Proof of Proposition \ref{bootstrap}}
\label{sectionbootstrap}

Our aim is first to show that for $s<s^*$, the a priori bounds \fref{init2h}, \fref{init3h}, \fref{init3hb} can be improved, and then the unstables modes $(U_k)_{2\le k\le L}$ will be controlled through a standard topological argument.\\

{\bf step 1} Improved $\dot{H}^1$ bound. First observe from \fref{defukbis} and the a priori bound on $U_k$ form $s<s^*$ that 
\be
\label{cneocneneo}
|b_k(s)|\lesssim  |b_k(0)|.
\ee
The energy bound \fref{init2h} is now a straightforward consequence of the dissipation of energy and the bounds \fref{cneocneneo}, \fref{controlunstable}. Indeed, let $$\tilde{\e}=\e+\alphah,$$ then
\bea
\label{consebfie}
E_0 & = & \int|\pa_y(Q+\et)|^2+\int\frac{g^2(Q+\et)}{y^2}\\
\nonumber & = & E(Q)+(H\et,\et)+\int\frac{1}{y^2}\left[g^2(Q+\et)-2f(Q)\et-f'(Q)\et^2\right].
\eea
We first use the bound on the profile which is easily extracted from Proposition \ref{consprofapproch}
$$\int|\pa_y\alphah|^2+\int\frac{|\alphah|^2}{y^2}\lesssim b_1|\log b_1|^C\leq \sqrt{b_1(0)} $$ using \fref{cneocneneo}. This ensures using Lemma \ref{coerch} the coercivity:
\bee
(H\et,\et)&\geq &c(M)\left[\int|\pa_y\et|^2+\int\frac{|\et|^2}{y^2}\right]-\frac{1}{c(M)}(\et,\Phi_M)^2\\
& \geq & c(M)\left[\int|\pa_y\e|^2+\int\frac{|\e|^2}{y^2}\right]-\sqrt{b_1(0)}.
\eee
The nonlinear term is estimated from a Taylor expansion:
$$
\left|\int\frac{1}{y^2}\left[g^2(Q+\et)-2f(Q)\et-f'(Q)\et^2\right]\right|\lesssim  \int \frac{|\et|^3}{y^2}\lesssim  \left(\int|\pa_y\et|^2+\int\frac{|\et|^2}{y^2}\right)^{\frac{3}{2}}.
$$
where we used the Sobolev bound $$\|\et\|_{L^{\infty}}^2\lesssim \|\pa_y\et\|_{L^2}\|\frac{\et}{y}\|_{L^2}.$$
We inject these bounds into the dissipation of energy \fref{consebfie} together with the initial bound \fref{init2energy} to estimate: 
\be
\label{init2hbis}
\int|\pa_y\e|^2+\int\frac{|\e|^2}{y^2}  \lesssim  \int|\pa_y\et|^2+\int\frac{|\et|^2}{y^2}+b_1|\log b_1|^C\leq c(M)\sqrt{b_1(0)}\leq (b_1(0))^{\frac 14}
\ee
for $|b_1(0)|\leq b_1^*(M)$ small enough.\\

{\bf step 2} Integration of the scaling law. Let us compute explicitly the scaling parameter for $s<s^*$. From \fref{parameters}, \fref{controlunstable}, \fref{defuk}, \fref{approzimatesolution}, we have the rough bound: 
$$-\lsl=\frac{c_1}{s}-\frac{|d_1|}{\log s}+O\left(\frac{1}{s(\log s)^{\frac 54}}\right)$$ which we rewrite 
\be
\label{eneoeoe}
\left|\frac{d}{ds}\left\{\frac{s^{c_1}\l(s)}{(\log s)^{|d_1|}}\right\}\right|\lesssim \frac{1}{s(\log s)^{\frac 54}}.
\ee
We integrate this using the initial value $\l(0)=1$ and conclude:
\be
\label{valuel}
\frac{s^{c_1}\l(s)}{(\log s)^{|d_1|}}=\frac{s_0^{c_1}}{(\log s_0)^{|d_1|}}+O\left(\frac{1}{(\log s_0)^{\frac 14}}\right).
\ee
Together with the law for $b_1$ given by \fref{controlunstable}, \fref{defuk}, \fref{approzimatesolution}, this implies:
\be
\label{keyestimated}
b_1(0)^{c_1}|\log b_1(0)|^{|d_1|}\lesssim \frac{b_1^{c_1}(s)|\log b_1|^{|d_1|}}{\l(s)}\lesssim b_1(0)^{c_1}|\log b_1(0)|^{|d_1|}.
\ee

{\bf step 3} Improved control of $\mathcal E_{2L+2}$. We now improve the control of the high order $\mathcal E_{2L+2}$ energy \fref{init3hb} by reintegrating the Lyapounov monotonicity \fref{monoenoiencle} in the regime governed by \fref{valuel}, \fref{defuk}. Indeed, we inject the bootstrap bound \fref{init3hb} into the monotonicity formula \fref{monoenoiencle} and integrate in time $s$: $\forall s\in [s_0,s^*)$,
\bea
\label{monotnyintegree}
\nonumber \mathcal E_{2L+2}(s) & \leq & 2\left(\frac{\lambda(s)}{\lambda(0)}\right)^{4L+2}\left[\mathcal E_{2L+2}(0)+Cb^{\frac 45}_1(0)\frac{b_1^{2L+2}(0)}{|\log b_1(0)|^2}\right]+\frac{b_1^{2L+2}(s)}{|\log b_1(s)|^2}\\
 & + & C\left[1+\frac{K}{\log M}+\sqrt{K}\right]\lambda^{2L+4}(s)\int_{s_0}^{s} \frac{b_1}{\lambda^{4L+2}}\frac{b_1^{2L+2}}{|\log b_1|^2}d\sigma
\eea
for some universal constant $C>0$ independent of $M$. We now observe from \fref{keyestimated} that the integral in the RHS of \fref{monotnyintegree} is divergent since:
 \bee
 \frac{b_1}{\lambda^{4L+2}}\frac{b_1^{2L+2}}{|\log b_1|^2}\gtrsim C(b_0)\frac{b_1^{2L+3}}{b_1^{(4L+2)c_1}|\log b_1|^C}\gtrsim \frac{C(b_0)}{(\log s)^Cs^{2L+3-(4L+2)\frac{L}{2L-1}}}=\frac{C(b_0)}{(\log s)^Cs^{\frac{2L-3}{2L-1}}},
 \eee
  and therefore from \fref{keyestimated} and $\frac 1s\lesssim b_1\lesssim \frac 1s$:
 \bee
 \lambda^{4L+2}(s)\int_{s_0}^{s} \frac{b_1}{\lambda^{4L+2}}\frac{b_1^{2L+2}}{|\log b_1|^2}d\sigma\lesssim \frac{b_1^{2L+2}(s)}{|\log b_1(s)|^2}.
 \eee
We now estimate the contribution of the initial data using \fref{keyestimated} and the initial bounds \fref{init2}, \fref{intilzero}:
\bee
&&\left(\frac{\lambda(s)}{\lambda(0)}\right)^{4L+2}\left[\mathcal E_{2L+2}(0)+Cb^{\frac 45}_1(0)\frac{b_1^{2L+2}(0)}{|\log b_1(0)|^2}\right]\lesssim \lambda^{4L+2}(s)b^{\frac 45}_1(0)\frac{b_1^{2L+2}(0)}{|\log b_1(0)|^2}\\
& \lesssim & (b_1(s))^{(4L+2)\frac{L}{2L-1}}|\log b_1(s)|^C(b_1(0))^{\frac 45+2L+2-(4L+2)\frac{L}{2L-1}}|\log b_1(0)|^C\\
& \lesssim & \frac{b_1^{2L+2}(s)}{|\log b_1(s)|^2}
\eee
where we used the algebra for $L\geq 2$: $$0<\frac{L(4L+2)}{2L-1}-(2L+2)=\frac{2}{2L-1}<\frac 45.$$
Injecting these bounds into \fref{monotnyintegree} yields 
\be
\label{init3hbbis}\mathcal E_{2L+2}(s)\lesssim \frac{b_1^{2L+2}(s)}{|\log b_1(s)|^2}\left[1+\frac{K}{\log M}+\sqrt{K}\right]\leq \frac{K}{2} \frac{b_1^{2L+2}(s)}{|\log b_1(s)|^2}
\ee for $K$ large enough independent of $M$.\\

{\bf step 4} Improved control of $\mathcal E_{2k+2}$, $0\leq k\leq L-1$. We now claim the improved bound on the intermediate energies:
\be
\label{init3hbbisbis}
\mathcal E_{2k+2}\leq b_1^{(4k+2)\frac{2L}{2L-1}}|\log b_1|^{C+\sqrt{K}}.
\ee
This follows from the monotonicity formula: for $0\leq k\leq L-1$,
\bea
\label{monoenoiencleappouetpoute}
&& \frac{d}{dt} \left\{\frac{1}{\lambda^{4k+2}}\left[\mathcal E_{2k+2}+O\left(b_1^{\frac 12}b_1^{(4k+2)\frac{2L}{2L-1}}\right)\right]\right\}\\
\nonumber & \lesssim & \frac{|\log b_1|^{C}}{\l^{4k+4}}\left[b_1^{2k+3}+b_1^{1+\delta+(2k+1)\frac{2L}{2L-1}}+\sqrt{b_1^{2k+4}\matchal E_{2k+2}}\right]
 \eea
for some universal constants $C,\delta>0$ independent of the bootstrap constant $K$. The proof is similar to the one of \fref{init3hbbis} and in fact simpler since we allow for logarithmic losses, details are given in Appendix \ref{appendenergy}.\\
We now estimate using \fref{valuel}:
\bee
\l^{4k+2}(s)\int_{s_0}^s\frac{b_1^{2k+3}}{\l^{4k+2}}|\log b_1|^{C}\lesssim \frac{(\log s)^{|d_1|}}{s^{(4k+2)c_1}}\int_{s_0}^s\frac{(\log \sigma)^C}{\sigma^{2k+3-c_1(4k+2)}}d\sigma.
\eee
We compute from \fref{defalphai}: 
\be
\label{keyalgebra}
(2k+3)-c_1(4k+2)=1+\frac{2(L-k-1)}{2L-1}
\ee
and hence:
\bee
\l^{4k+2}(s)\int_{s_0}^s\frac{b_1^{2k+3}}{\l^{4k+2}}|\log b_1|^{C}\lesssim \frac{(\log s)^{|d_1|+C}}{s^{(4k+2)c_1}}\lesssim b_1^{4k+2}|\log b_1|^C.
\eee
Similarily, from \fref{keyestimated}: 
\bee
\l^{4k+2}(s)\int_{s_0}^s\frac{b_1^{1+\delta+(2k+1)\frac{2L}{2L-1}}}{\l^{4k+2}}|\log b_1|^Cd\sigma &\lesssim&  \frac{(\log s)^{|d_1|+C}}{s^{(4k+2)c_1}}\int_{s_0}^s\frac{(\log \sigma)^C}{\sigma^{1+\delta}}d\sigma\\
& \lesssim & b_1^{4k+2}|\log b_1|^C
\eee
and using \fref{keyalgebra}, \fref{init3h}:
\bee
&&\l^{4k+2}(s)\int_{s_0}^s\frac{|\log b_1|^{C}}{\l^{4k+2}}\sqrt{b_1^{2k+4}\matchal E_{2k+2}}d\sigma \lesssim  \frac{(\log s)^{|d_1|+C}}{s^{(4k+2)c_1}}\int_{s_0}^s\frac{(\log \sigma)^{C+\sqrt{K}}}{\sqrt{s^{2k+4-(2k+1)\frac{2L}{2L-1}}}}d\sigma\\
& \lesssim & |\log b_1|^{C+\sqrt{K}}b_1^{4k+2}\int_{s_0}^s\frac{d\sigma}{\sigma^{1+\frac{L-k-1}{2L-1}}}\lesssim  |\log b_1|^{C+\sqrt{K}}b_1^{4k+2}.
\eee
The time integration of \fref{monoenoiencleappouetpoute} from $s=s_0$ to $s$ therefore yields using also the initial smallness \fref{init2} and \fref{keyestimated}:
$$\matchal E_{2k+2}(s)\lesssim \l^{4k+2}(s)b_1(0)^{10L+4}+ |\log b_1(s)|^{C+\sqrt{K}}b_1^{4k+2}(s)\lesssim  |\log b_1(s)|^{C+\sqrt{K}}b_1^{4k+2}(s),$$ and \fref{init3hbbisbis} is proved.

\begin{remark}
\label{remarkreg} For $0\leq k\leq L-2$, the above argument shows the bound $$\matchal E_{2k+2}\lesssim \l^{4k+2}$$ which equivalently corresponds to a uniform high order Sobolev control $w$. The logarithmic loss for $k=L-1$ could be gained as well with a little more work, see \cite{RSc1} for the case $L=1$. This shows that the limiting excess of energy $u^*$ in \fref{concnenergy} enjoys some suitable high order Sobolev regularity.
\end{remark}

{\bf step 5} Contradiction through a topological argument. Let us consider $$\tilde{b}_k=b_k\ \ \mbox{for}\ \ 1\leq k\leq L, \ \ \tilde{b}_L\ \ \mbox{given by \fref{defbtildel}}$$ and the associated variables $$\bt_k=b_k^e+\frac{\tilde{U}_k}{s^k(\log s)^{\frac 54}},\ \ 1\leq k\leq L, \ \ \bt_{k+1}=\tilde{U}_{k+1}\equiv 0, \ \ \tilde{V}=P_L\tilde{U}.$$ From \fref{poitwidediff}:
\be
\label{estexit}
|V-\tilde{V}|\lesssim s^L|\log s|^Cb_1^{L+\frac 12}\lesssim \frac{1}{s^{\frac14}}.
\ee
Let the associated control of the unstable models be:
\be
\label{controlunstablebis}
|\tilde{V}_1(s)|\leq 2, \ \ \left(\tilde{V}_2(s),\dots,\tilde{V}_L(s)\right)\in \mathcal B_{L-1}(\frac 12).
\ee
and the slightly modified exit time:
\bee
\tilde{s}^*= \sup\{s\geq s_0\ \ \mbox{such that} \ \ \fref{init2h}, \fref{init3h}, \fref{init3hb}, \fref{controlunstablebis}\ \ \mbox{hold on }\ \ [s_0,s]\},
\eee
then \fref{estexit} and the assumption \fref{hypcontr} imply:
\be
\label{hypcontrbis}
\forall \left(\tilde{V}_2(0),\dots,\tilde{V}_L(0)\right)\in \mathcal B_{L-1}(\frac 12), \ \ \tilde{s}^*<+\infty.
\ee
We claim that this contradicts Brouwer fixed point theorem.\\
Indeed, we first estimate from \fref{reformulaion}
$$
(\bt_k)_s+\left(2k-1+\frac{2}{\log s}\right)\bt_1\bt_k-\bt_{k+1}=  \frac{1}{s^{k+1}(\log s)^{\frac 54}}\left[s(\tilde{U}_k)_s-(A_L\tilde{U})_k+O\left(\frac{1}{\sqrt{\log s}}\right)\right] 
$$
and thus from \fref{parameterspresicely}, \fref{esteqbrildl}, \fref{init3hbbis} and \fref{cbdb}:
\bee
\left |s(\tilde{U}_k)_s-(A_L\tilde{U})_k\right|&\lesssim &\frac{1}{\sqrt{\log s}}+s^{k+1}(\log s)^{\frac 54}\left|(\bt_k)_s+\left(2k-1+\frac{2}{\log s}\right)\bt_1\bt_k-\bt_{k+1}\right|\\
& \lesssim &\frac{1}{\sqrt{\log s}}+s^{k+1}(\log s)^{\frac 54}\left[b_1^{L+\frac 32}+\frac{1}{s^{k+1}(\log s)^2}+\frac{b_1^{L+1}}{|\log b_1|^{\frac 32}}\right]\\
& \lesssim & \frac{1}{(\log s)^{\frac 14}}
\eee
and hence using the diagonalization \fref{specta}:
\be
\label{eqvtilde}
s(\tilde V)_s=D_L\tilde{V}_s+O\left(\frac{1}{(\log s)^{\frac 14}}\right).
\ee
This first implies the control of the stable mode $\tilde{V}_1$ from \fref{specta}:
$$|(s\tilde{V}_1)_s|\lesssim \frac{1}{(\log s)^{\frac14}}$$ and thus from \fref{estintial}:
\be
\label{controlstable}
|\tilde{V}_1(s)|\lesssim \frac{1}{s}+\frac{1}{s}\int_{s_0}^s\frac{d\tau}{(\log \tau)^{\frac14}}\leq \frac{1}{10}.
\ee
Now from \fref{init2hbis}, \fref{init3hbbis}, \fref{init3hbbisbis}, \fref{controlstable} and a standard continuity argument, 
\be
\label{cneoneno}
\Sigma_{i=2}^L\tilde{V}_i^2(s^*)=\frac{1}{4}.
\ee 
We then compute from \fref{eqvtilde} the fundamental strict outgoing condition at the exit time $\tilde{s}^*$ defined by \fref{cneoneno}:
\bee
\frac 12\frac{d}{ds}\left\{\Sigma_{i=2}^L\tilde{V}_i^2\right\}_{|s=\tilde{s}^*}&=&\Sigma_{i=2}^L(\tilde{V}_i)_s\tilde{V}_i= \frac{1}{\tilde{s}^*}\left[\Sigma_{i=2}^{L}\frac{i}{2L-1}\tilde{V}_i^2(s^*)+O\left(\frac{1}{(\log \tilde{s}^*)^{\frac14}}\right)\right]\\
& \geq& \frac{1}{s^*}\left[\frac{2}{2L-1}\frac{1}{4}+O\left(\frac{1}{(\log \tilde{s}^*)^{\frac 14}}\right)\right]>0.
\eee
This implies from standard argument the continuity of the map $$(\tilde{V}_i)_{2\leq i\leq L}\in \mathcal B_{L-1}(\frac 12)\mapsto \tilde{s}^*\left[(\tilde{V}_i)_{2\leq i\leq L}\right]$$ and hence the continuous map $$ \begin{array}{ll} \mathcal B_{L-1}(\frac 12)\to \mathcal B_{L-1}(\frac 12)\\
(\tilde{V}_i)_{2\leq i\leq L}\mapsto \left\{\tilde{V}_i\left[\tilde{s}^*( (\tilde{V}_i)_{2\leq i\leq L})\right]\right\}
\end{array}$$
is the identity on the boundary sphere $\Bbb S_{L-1}(\frac 12)$, a contradiction to Brouwer's fixed point theorem. This concludes the proof of Proposition \ref{bootstrap}.


\subsection{Proof of Theorem \ref{thmmain}}


We pick an initial data satisfying the conclusions of Proposition \ref{bootstrap}.  In particular, \fref{eneoeoe} implies the existence of $c(u_0)>0$ such that $$\lambda(s)=c(u_0)\frac{(\log s)^{|d_1|}}{s^{c_1}}\left[1+O\left(\frac{1}{(\log s)^{\frac 14}}\right)\right]$$ and then from \fref{parameters}, \fref{defuk}:
$$-\l\l_t=-\lsl=b_1+O\left(\frac{1}{s^L}\right)=\frac{c_1}{s}\left[1+O\left(\frac{1}{\log s}\right)\right]=\frac{c(u_0)\l^{\frac 1{c_1}}}{|\log \l|^{\frac{|d_1|}{c_1}}}\left[1+O\left(\frac{1}{(\log s)^{\frac 14}}\right)\right]$$ and hence the pointwise differential equation: $$-\l^{\frac{-(L-1)}{L}}|\log \l|^{\frac{2}{2L-1}}\l_t=c(u_0)(1+o(1)).$$ We easily conclude that $\l$ touches zero at some finite time $T=T(u_0)<+\infty$ with near blow up time: $$\l(t)=c(u_0)(1+o(1))\frac{(T-t)^{L}}{|\log (T-t)|^{\frac{2L}{2L-1}}}(1+o(1)).$$ The strong convergence \fref{concnenergy} now follows as in \cite{RSc1}. This concludes the proof of Theorem \ref{thmmain}.


\begin{appendix}.



\section{Regularity in corotational symmetry}


We detail in this appendix the regularity of smooth maps with 1-corotational symmetry. 

\begin{lemma}[Regularity in corotational symmetry]
\label{regularitycorot}
Let $v$ be a smooth 1-corotational map 
\be
\label{corotv}
v(y,\theta)=\left|\begin{array}{ll}g(u(y))\cos\theta\\g(u(y))\sin\theta\\z(u(y))\end{array}\right., 
\ee 
with 
\be
\label{smoothboundary}
v(0)=e_z, \ \ \lim_{y\to +\infty} v(x)\to -e_z.
\ee
Assume that $v$ is smooth in Sobolev sense: $$\Sigma_{i=1}^N\int |(-\Delta)^{\frac i2} v|^2<+\infty,$$ for some $N\gg L$,
then:\\
(i) $u$ is a smooth function of $y$ with a Taylor expansion at the origin for $p\leq 10L$:
\be
\label{cnekoheoeo}
u(y)=\Sigma_{k=0}^{p}c_ky^{2k+1}+O(y^{2p+3}).
\ee
(ii) Assume that $u(y)=Q(y)+\e(y)$ with 
\be
\label{estlinfty}
\|\nabla \e\|_{L^2}+\|\frac{\e}{y}\|_{L^2}\ll 1,
\ee
and consider the sequence of suitable derivatives $\e_k=\mathcal A^k\e$, then: $\forall 1\leq k\leq L$
\bea
\label{cnojeooej}
&& \int|\e_{2k+2}|^2+\int\frac{|\e_{2k+1}|^2}{y^2(1+y^2)}\\
\nonumber & + & \Sigma_{p=0}^{k}\int\left[\frac{|\e_{2p-1}|^2}{y^6(1+|\log y|^2)(1+y^{4(k-p)})}+\frac{|\e_{2p}|^2}{y^4(1+|\log y|^2)(1+y^{4(k-p)})}\right]<+\infty.
\eea
\end{lemma}

\begin{proof}[Proof of Lemma \ref{regularitycorot}]

Let us consider the rotation matrix \be
 \label{defR}
 R=\left(\begin{array}{lll} 0 & -1 & 0 \\ 1 & 0 &0\\ 0 & 0& 0
\end{array}
\right ),
\ee and rewrite \fref{corotv}: 
\be
\label{coccooejeo}
v(r,\theta)=e^{\theta R}w\ \ \mbox{with}\ \ w(r)=\left|\begin{array}{lll}w_1=g(u)\\0\\w_3=z(u)\end{array}\right.
\ee
{\bf step 1} Control at the origin. We compute the energy density 
\be
\label{gradv}
|\nabla v|^2=|\pa_yw|^2+\frac{|w|^2}{y^2}
\ee which is bounded from the smoothess of $v$ from which 
\be
\label{firstbound}
\left|\frac{w_1}{y}\right|+|\pa_yw_1|\lesssim 1.
\ee Similarily, 
\be
\label{formulalaplacev}
\Delta v=e^{\theta R}\left(\Delta w+R^2\frac{w}{y^2}\right)=e^{\theta R}\left|\begin{array}{lll}-\Bbb Hw_1\\0\\\Delta w_3\end{array}\right.
\ee 
where $$\Bbb Hw_1=-\Delta w_1 +\frac{w_1}{y^2}=\Bbb A^*\Bbb Aw_1$$ with $$\Bbb A=-\pa_y+\frac{1}{y}, \ \ \Bbb A^*=\pa_y+\frac{2}{y}.$$ The regularity of $v$ implies $$|\Bbb Hw_1|\lesssim 1$$ near the origin which together with \fref{firstbound} yields: 
\be
\label{estzwitht}
\Bbb Aw_1(y)=\frac{1}{y^2}\int_0^r(\Bbb Hw_1)\tau^2d\tau=O(y).
\ee We now observe that $$\Bbb Hw_1=-\pa_{yy}w_1+\frac{1}{y}\Bbb Aw_1$$ and conclude $$|\pa^2_{yy}w_1|\lesssim 1.$$ We now iterate this argument once on \fref{formulalaplacev}. Indeed, at the origin, $$\left|\pa_y\Bbb H w_1\right|^2+\frac{|\Bbb Hw_1|^2}{y^2}\lesssim |\nabla \Delta v|^2\lesssim 1, \ \ |\Bbb H^2 w_1|\lesssim |\Delta^2 v|\lesssim 1$$ and hence
 $$|\Bbb Hw_1|\lesssim y, \ \ |\Bbb A\Bbb Hw_1|\lesssim y, \ \ |\Bbb H^2w_1|\lesssim 1.$$ This yields the $\mathcal C^3$ regularity of $w_1$ at the origin and the improved bound from \fref{estzwitht}: $$\Bbb Aw_1(y)=\frac{1}{y^2}\int_0^y(\Bbb Hw_1)\tau^2d\tau=O(y^2).$$ A simple induction now yields for all $k\geq 1$ the $\mathcal C^k$ regularity of $w_1$, and that the sequence $$(w_1)_0=w_1, \ \ (w_1)_{k+1}=\left\{\begin{array}{ll} \Bbb A^*(w_1)_k\ \ \mbox{for k odd}\\ \Bbb A (w_1)_k\ \ \mbox{for k even}
\end{array}\right., \ \ k\geq 1.$$ satisfies the bound 
\be
\label{exapnsionwkbounds}
|(w_1)_k|\lesssim \left\{\begin{array}{ll} y\ \ \mbox{for k even}\\ y^2 \ \ \mbox{for k odd}
\end{array}\right..
\ee 
We therefore let a Taylor expansion at the origin $$w_1(y)=\Sigma_{i=1}^{p}c_iy^i+O(y^{p+1}),$$ apply successively the operator $\Bbb A, \Bbb A^*$ and conclude from the relations $$\Bbb A(r^k)=-(k-1)y^{k-1}, \ \ \Bbb A^*(y^k)=(k+2)y^{k-1} $$ and \fref{exapnsionwkbounds} that $$c_{2k}=0, \ \ \forall k\geq 1.$$ We now recall from \fref{coccooejeo} that $w_1=g(u)$ and the Taylor expansion \fref{cnekoheoeo} now follows from the odd parity of $g$ at the origin.\\ We now claim that this implies the bound \fref{cnojeooej} at the origin. Indeed, $\e$ admits a Taylor expansion \fref{cnekoheoeo} from \fref{origin} to which we apply successively the operators $A,A^*$. We observe from \fref{comportementz} the cancellation $$A(y)=cy^2+O(y^3)$$ which ensures the bound near the origin 
\be
\label{devayorigin}
|\e_{2k}|\lesssim y, \ \ |\e_{2k+1}|\lesssim y^2,
\ee hence the finiteness of the norms \fref{cnojeooej} at the origin.\\

{\bf step 2}. Control for $r\geq 1$. We first claim:
\be
\label{derivativescontrol}
\int \frac{\e^2}{y^2}+\Sigma_{k=1}^{L+2}\int (\pa_y^k\e)^2<+\infty.
\ee
Indeed, from \fref{estlinfty}, 
\be
\label{cneneonve}
\|\e\|_{L^{\infty}}\lesssim \|\nabla \e\|_{L^2}+\|\frac{\e}{y}\|_{L^2}\ll 1.
\ee From \fref{gradv}, 
$$\int |\Delta g(u)|^2\lesssim \int |\Delta v|^2<+\infty.$$ Now $$\Delta g(u)=g'(u)\Delta u +(\pa_yu)^2g'(u)g''(u)$$ and we  estimate using Sobolev and the $L^{\infty}$ control \fref{cneneonve}: $$\int\left((\pa_y\e)^2g'(u)g''(u)\right)^2\lesssim \int |\Delta \e|^2 \int |\nabla\e|^2\ll \int |\Delta \e|^2.$$ Moreoever,  from the smallness \fref{cneneonve} and the structure of $Q$, $$|g'(u)|\gtrsim 1\ \ \mbox{as}\ \  r\to +\infty$$ from which: $$1+\int |\Delta v|^2\gtrsim \int |\Delta \e|^2.$$  The control of higher order Sobolev norms \fref{derivativescontrol} now follows similarly by induction using the Faa di Bruno formula for the computation of $\pa_y^kg(u)$, this is left to the reader. Now \fref{derivativescontrol} easily implies $$\Sigma_{i=1}^{L+2}\int_{y\geq 1}|\e_{k}|^2<+\infty$$ and the bound \fref{cnojeooej} is proved far out.\\
\end{proof}


\section{Coercitivity bounds}


Given $M\geq 1$, we let $\Phi_M$ be given by \fref{defdirection}. Let us recall the coercivity of the operator $H$ which is a standard consequence of the knowledge of the kernel of $H$ and the nondegeneracy \fref{nondgeener}.

\begin{lemma}[Coercivity of $H$]
\label{coerch}
Let $M\geq 1$ large enough, then there exists $C(M)>0$ such that for all radially symmetric $u$ with $$\int\left[|\pa_y u|^2+\frac{|u|^2}{y^2}\right]<+\infty $$ satisfying $$(u,\Phi_M)=0,$$ there holds: $$(Hu,u)\geq C(M)\left[\int|\pa_y u|^2+\frac{|u|^2}{y^2}\right].$$
\end{lemma}

We now recall the coercivity of $\tilde{H}$ which is a simple consequence of \fref{efinitei} and is proved in \cite{RaphRod}.

\begin{lemma}[Coericivity of $\tilde{H}$]
\label{coerchtilde}
Let $u$ with 
\be
\label{assmuttpo}
\int |\pa_yu|^2+\int\frac{|u|^2}{y^2}<+\infty,
\ee then 
\be
\label{coercwthree}
(\tilde{H}u,u)=\|A^*u\|_{L^2}^2\geq c_0\left[\int|\pa_yu|^2+\int\frac{|u|^2}{y^2(1+|\log y|^2)}\right]
\ee
for some universal constant $c_0>0$.
\end{lemma}

We now claim the following weighted coercivity bound on $H$.

\begin{lemma}[Coercivity of $\mathcal E_2$]
\label{ceorcetwobis}
There exists $C(M)>0$ such that for all radially symmetric $u$ with 
 \be
 \label{estinitialbisbisbis}
 \int\frac{|u|^2}{y^4(1+|\log y|^2)}+\int\frac{|Au|^2}{y^2(1+y^2)}<+\infty
 \ee 
 and $$(u,\Phi_M)=0,$$ there holds: 
 \be
 \label{coerciveboundbis}
  \int|Hu|^2 \geq  C(M)\left[\int\frac{|u|^2}{y^4(1+|\log y|^2)}+\int\frac{|Au|^2}{y^2(1+|\log y|^2)}\right].
  \ee
 \end{lemma}
 
 \begin{proof}[Proof of Lemma \ref{ceorcetwobis}] This lemma is proved in \cite{RaphRod} in the case of the sphere target. Let us briefly recall the argument for the sake of completeness.\\
 
 {\bf step 1} Conclusion using a subcoercivity lower bound. We claim the subcoercivity lower bound for any $u$ satisfying \fref{estinitialbisbis}:
 \bea
 \label{subceorovitybis}
\nonumber &&\int|H u|^2 \gtrsim  \int \frac{|\pa^2_{y}u|^2}{(1+|\log y|^2)}+ \int \frac{(\pa_y u)^2}{y^2(1+|\log y|^2)}+\int \frac{|u|^2}{y^4(1+|\log y|^2)}\\
 & - & C\left[ \int \frac{(\pa_y u)^2}{1+y^{3}}+ \int\frac{|u|^2}{1+y^{5}} \right].
\eea
Let us assume \fref{subceorovitybis} and conclude the proof of \fref{coerciveboundbis}. By contradiction, let $M>0$ fixed and consider a normalized sequence $u_n$  
\be
\label{normlaizationbis}
\int\frac{|u_n|^2}{y^4(1+|\log y|^2)}+\int\frac{|Au_n|^2}{y^2(1+|\log y|^2)}=1,
\ee
satisfying the orthogonality condition
\be\label{eq:orthogbis} 
(u_n,\Phi_M)=0
\ee and the smallness:
\be
\label{seeumporgpobis}
\int |Hu_n|^2\leq \frac1n.
  \ee
Note that the normalization condition implies 
$$\int\frac{|u_n|^2}{y^4(1+|\log y|^2)}+\int\frac{|\pa_yu_n|^2}{y^2(1+|\log y|^2)}\lesssim1$$ and thus from \fref{subceorovitybis}, the sequence $u_n$ is bounded in $H^2_{loc}$. Hence there exists $u_{\infty}\in H^2_{loc}$ such that up to a subsequence and for any smooth cut-off function $\zeta$ vanishing in a neighborhood of  $y=0$, the sequence $\zeta u_n$ is uniformly bounded in $H^2_{loc}$ and converges to $\zeta u_\infty$ in $H^1_{loc}$. Moreover, \fref{seeumporgpobis} implies $$Hu_{\infty}=0$$ and by lower semi continuity of the norm and \fref{normlaizationbis}: $$\int\frac{|u_\infty|^2}{y^4(1+|\log y|^2)}<+\infty$$ which implies from \fref{kernelh}: $$u=\alpha\Lambda Q\ \ \mbox{for some}\ \ \alpha\in \Bbb R.$$ We may moreover pass to the limit in  \fref{eq:orthogbis} from \fref{normlaizationbis} and the local compactness of Sobolev embeddings and thus $$(u_\infty,\Lambda Q)=0\ \ \mbox{from which} \ \ \alpha=0$$ where we used the nondegeneracy \fref{nondgeener}. Hence $u_\infty=0$. Now the subcoercivity lower bound \fref{subceorovitybis} together with \fref{normlaizationbis}, \fref{seeumporgpobis} and the $H^2_{loc}$ uniform bound imply the existence of $\e,c>0$ such that: $$\int_{\e\leq y\leq \frac1\e}\left[ \frac{|\pa_y u_\infty|^2}{1+y^3}+ \frac{|u_\infty|^2}{1+y^{5}}\right]\geq c>0$$ which contradicts the established identity $u_\infty=0$, and \fref{coerciveboundbis} is proved.\\ 

{\bf step 2} Proof of \fref{subceorovitybis}. Let us first apply Lemma \ref{coerchtilde} to $Au$ which satisfies \fref{assmuttpo} by assumption and estimate: 
\be
\label{firtgiefei}
\int (Hu)^2=(\tilde{H}Au,Au)\gtrsim \int|\pa_y(Au)|^2+\int \frac{|Au|^2}{y^2(1+y^2)}.
\ee
Near the origin, we now recall the logarithmic Hardy inequality: $$\int_{y\leq 1}\frac{|v|^2}{y^2(1+|\log y|^2)}\lesssim \int_{1\leq y\leq 2}|v|^2+\int_{y\leq 1}|\pa_y v|^2$$ and thus  using \fref{otherformula}:
$$\int\frac{|Au|^2}{y^2} =  \int\frac{1}{y^2}\left|\Lambda Q\pa_y\left(\frac{u}{\Lambda Q}\right)\right|^2\gtrsim\int_{y\leq 1}\left|\frac{u}{\Lambda Q}\right|^2\frac{dy}{y^2(1+|\log y|^2)}-\int_{y\leq 1}(|\pa_yu|^2+|u|^2)$$ which together with \fref{firtgiefei} yields 
$$\int|H u|^2 \gtrsim  \int_{y\leq 1}\left[\frac{(\pa_y u)^2}{y^2(1+|\log y|^2)}+\int \frac{|u|^2}{y^4(1+|\log y|^2)}\right]-C\int_{y\leq 1}(|\pa_yu|^2+|u|^2).$$
To control the second derivative, we rewrite near the origin: $$Hu=-\pa^2_yu+\frac{1}{y}\left(-\pa_y u+\frac{u}{y}\right)+\frac{V-1}{y^2}u=-\pa_y^2u+\frac{Au}{y}+\frac{(V-1)+(1-Z)}{y^2}u$$ and \fref{subceorovitybis} follows near the origin.\\
Away from the origin, let $\zeta(y)$ be a smooth cut-off function with support in $y\ge \frac12$ and equal to 1 for $y\ge 1$. We use the logarithmic Hardy inequality $$\int_{y\geq 1}\frac{|u|^2}{y^2(1+|\log y|^2)}\lesssim \int_{1\leq y\leq 2}|u|^2+\int|\pa_y u|^2$$ to conclude from \fref{firtgiefei}:
$$\int(Hu)^2\gtrsim \int\zeta\frac{|Au|^2}{y^2(1+|\log y|^2)}-C\int_{1\leq y\leq 2}(|u|^2+|\pa_yu|^2).$$ We now estimate from \fref{comportementz} :
\bee
\int\zeta\frac{|Au|^2}{y^2(1+|\log y|^2)}& = & \int\zeta\frac{|-\pa_yu-\frac{u}{y}|^2}{y^2(1+|\log y|^2)}-C\int\zeta\frac{|u|^2}{y^6(1+|\log y|^2)}\\
& \gtrsim & \int\frac{\zeta}{y^2(1+|\log y|^2)}\left[|\pa_yu|^2+\frac{|u|^2}{y^2}\right]-C\int\left[\frac{|u|^2}{y^5}+\frac{|\pa_yu|^2}{y^3}\right]
\eee
where we integrated by parts in the last step. The control of the second derivative follows from the explicit expression of $H$. This concludes the proof of \fref{subceorovitybis}.
\end{proof}

We now aim at generalizing the coercivity of the $\mathcal E_2$ energy of Lemma \ref{ceorcetwobis} to higher order energies. This first requires a generalization of the weighted estimate \fref{coerciveboundbis}.
 
 \begin{lemma}[Weighted coercitivity bound]
 \label{boundweight}
 Let $L\geq 1$, $0\leq k\leq L$ and $M\geq M(L)$ large enough. Then there exists $C(M)>0$ such that for all radially symmetric $u$ with 
 \be
 \label{estinitialbisbis}
 \int\frac{|u|^2}{y^4(1+|\log y|^2)(1+y^{4k+4})}+\int\frac{|Au|^2}{y^6(1+|\log y|^2)(1+y^{4k+4})}<+\infty
 \ee 
 and $$(u,\Phi_M)=0,$$ there holds: 
 \bea
 \label{coercivebound}
&& \int\frac{|Hu|^2}{y^4(1+|\log y|^2)(1+y^{4k})}\\
\nonumber & \geq & C(M)\left\{\int\frac{|u|^2}{y^4(1+|\log y|^2)(1+y^{4k+4})}+\int\frac{|Au|^2}{y^6(1+|\log y|^2)(1+y^{4k})}\right\}
  \eea
 \end{lemma}
 
 \begin{proof}[Proof of Lemma \ref{boundweight}]
 
 {\bf step 1} Subcoercivity lower bound. We claim the subcoercivity lower bound for any $u$ satisfying \fref{estinitialbisbis}:
 \bea
 \label{subceorovity}
 &&\int \frac{|H u|^2}{y^4(1+|\log y|)^2(1+y^{4k})}\\
\nonumber &  \gtrsim & \int \frac{|\pa^2_{y}u|^2}{y^4(1+|\log y|^2)(1+y^{4k})}+ \int \frac{(\pa_y u)^2}{y^2(1+|\log y|)^2(1+y^{4k+4})}\\
\nonumber & + & \int \frac{|u|^2}{y^4(1+|\log y|^2)(1+y^{4k+4})} -  C\left[ \int \frac{(\pa_y u)^2}{1+y^{4k+8}}+ \int\frac{|u|^2}{1+y^{4k+10}} \right].
\eea
\underline{{\it Control near the origin}}. Recall from the finitness of the norm \fref{estinitialbisbis} the formula \fref{formulaau}:
 $$Au(y)=\frac{1}{y\Lambda Q(y)}\int_0^y \tau\Lambda Q(\tau)Hu(\tau)d\tau.$$
 We then estimate from Cauchy Schwarz and Fubbini:
\bee
&&\int_{y\le 1}\frac{|Au|^2}{y^5(1+|\log y|^2)}dy  \lesssim \int_{0\le y\leq 1}\int_{0\le\tau\le y}\frac{y^5}{y^9(1+|\log y|^2)}|Hu(\tau)|^2dyd\tau\\
& \lesssim & \int_{0\le \tau\le 1}|Hu(\tau)|^2\left[\int_{\tau\leq y\leq 1}\frac{dy}{y^4(1+|\log y|^2)} \right]d\tau\lesssim  \int_{\tau\le 1}\frac{|Hu(\tau)|^2}{\tau^3(1+|\log \tau|^2)} d\tau
\eee
and thus: 
\be
\label{cnoceooehoe}
\int_{y\le 1}\frac{|Au|^2}{y^6(1+|\log y|^2)}\lesssim  \int_{y\le 1}\frac{|Hu|^2}{y^4(1+|\log y|^2)}.
\ee
We now invert $A$ and get from \fref{otherformula} the existence of $c(u)$ such that $$u(y)=c(u)\Lambda Q(y)-\Lambda Q(y)\int_0^y\frac{Au(\tau)}{\Lambda Q(\tau)}d\tau.$$ We estimate from Cauchy Schwarz and \fref{cnoceooehoe} for $1\leq y\leq 1$:
\bee
\left|\int_0^y\frac{Au(\tau)}{\Lambda Q(\tau)}d\tau\right|^2 & \lesssim & y^4(1+|\log y|^2)\int_{0}^y\frac{|Au|^2}{\tau^5(1+|\log \tau|^2)}d\tau\\
& \lesssim & y^3 \int_{y\le 1}\frac{|Hu|^2}{y^4(1+|\log y|^2)}
\eee
 from which $$|c(u)|^2\lesssim \int_{y\leq 1}|u|^2+\int_{y\le 1}\frac{|Hu|^2}{y^4(1+|\log y|^2)}$$ and 
 \be
 \label{contorou}
 \int_{y\le 1}\frac{|u|^2}{y^4(1+|\log y|^2)}\lesssim  \int_{y\le 1}\frac{|Hu|^2}{y^4(1+|\log y|^2)}+\int_{1\leq y\leq 2}|u|^2.
 \ee
The control of one derivative follows from \fref{cnoceooehoe}, \fref{contorou} and the definition of $A$:
\bee
\int_{y\le 1}\frac{|\pa_yu|^2}{y^2(1+|\log y|^2)} & \lesssim & \int_{y\le 1}\frac{|Au|^2}{y^2(1+|\log y|^2)}+\int_{y\le 1}\frac{|u|^2}{y^4(1+|\log y|^2)}\\
& \lesssim & \int_{y\le 1}\frac{|Hu|^2}{y^4(1+|\log y|^2)}+\int_{1\leq y\leq 2}|u|^2.
\eee To control the second derivative, we rewrite near the origin: $$Hu=-\pa^2_yu+\frac{1}{y}\left(-\pa_y u+\frac{u}{y}\right)+\frac{V-1}{y^2}u=-\pa_y^2u+\frac{Au}{y}+\frac{(V-1)+(1-Z)}{y^2}u$$ which implies using \fref{cnoceooehoe}, \fref{contorou} and \fref{comportementz}, \fref{comportementv}:
$$\int_{y\leq 1}\frac{|\pa^2_yu|^2}{y^4(1+|\log y|^2)}\lesssim \int_{y\le 1}\frac{|Hu|^2}{y^4(1+|\log y|^2)}+\int_{1\leq y\leq 2}|u|^2.
$$
This concludes the proof of \fref{subceorovity} near the origin.\\
 \underline{{\it Control away from the origin}}. Let $\zeta(y)$ be a smooth cut-off function with support in $y\ge \frac12$ and equal to 1 for $y\ge 1$. We compute:
\bea
\label{inoeiheioyeo}
\nonumber &&\int \zeta \frac{|H u|^2}{y^{4k+4}(1+|\log y|)^2} = \int \zeta \frac{|- \pa_y (y\pa_y u) + \frac Vy u|^2}{y^{4k+6}(1+|\log y|)^2}\\ 
\nonumber &= & \int \zeta \frac{|\pa_y (y\pa_y u)|^2}{y^{4k+6}(1+|\log y|)^2} -2  \int \zeta \frac{\pa_y (y\pa_y u)\cdot V u}{y^{4k+7}(1+|\log y|)^2}+ \int \zeta \frac{V^2 |u|^2}{y^{4k+8}(1+|\log y|)^2}\\ 
\nonumber  &= & \int \zeta \frac{|\pa_y (y\pa_y u)|^2}{y^{4k+6}(1+|\log y|)^2} +
 2  \int \zeta \frac{ V (\pa_y u)^2}{y^{4k+6}(1+|\log y|)^2}+ \int  \zeta \frac{V^2 |u|^2}{y^{4k+8}(1+|\log y|)^2}\\ 
 &-& 
 \int |u|^2 \Delta\left( \frac{\zeta V }{y^{4k+6}(1+|\log y|)^2}\right) 
\eea
We now use the two dimensional logarithmic Hardy inequality with best constant\footnote{which can be obtained by a simple integration by parts, see \cite{MRR}}: $\forall \gamma>0$,
\bea
\label{harfylog'}
&&\frac{\gamma^2}{4} \int_{y\geq 1} \frac{|v|^2}{y^{2+\gamma}(1+|\log y|)^2}\\
\nonumber & \leq &  C_\gamma\int_{1\leq y\leq 2} |v|^2+
 \int_{y\geq 1} \frac{|\pa_y v|^2}{y^\gamma(1+|\log y|)^2},
\eea
with $\gamma=4k+6$ and estimate:
\bee
 \int \zeta \frac{|\pa_y (y\pa_y u)|^2}{y^{4k+6}(1+|\log y|)^2}& \geq & \frac{(4k+6)^2}{4}\int_{y\geq 1}\frac{|\pa_y u|^2}{y^{4k+6}(1+|\log y|)^2}-C_k\int_{1\leq y\leq 2}|\pa_yu|^2\\
 & \geq & \frac{(4k+6)^4}{16}\int_{y\geq 1}\frac{|u|^2}{y^{4k+8}(1+|\log y|)^2}-C_k\int_{1\leq y\leq 2}\left[|\pa_yu|^2+|u|^2\right].
 \eee
We now observe that for $k\ge 0$ and $y\ge 1$:
$$
\pa_y^k V(y)= \pa_y^k (1)+ O(y^{-2-k})
$$
and compute $$\Delta\left(\frac{1}{y^{4k+6}}\right)=\frac{(4k+6)^2}{y^{4k+8}}.$$ Injecting these bounds into \fref{inoeiheioyeo} yields the lower bound: 
\bee
\int \zeta \frac{|H u|^2}{y^{4k+4}(1+|\log y|)^2} & \geq & \left[\frac{(4k+6)^4}{16}-(4k+6)^2\right]\int_{y\geq 1}\frac{|u|^2}{y^{4k+8}(1+|\log y|)^2}\\
& - & C_k\int\left[\frac{|\pa_yu|^2}{1+y^{4k+8}}+\frac{|u|^2}{1+y^{4k+10}}\right].
\eee
Note that we can always keep the control of the first two derivatives in these estimates, and the control \fref{subceorovity} follows away from the origin.\\

{\bf step 2} Proof of \fref{coercivebound}. By contradiction, let $M>0$ fixed and consider a normalized sequence $u_n$  
\be
\label{normlaization}
\int\frac{|u_n|^2}{y^4(1+|\log y|^2)(1+y^{4k+4})}+\int\frac{|Au_n|^2}{y^6(1+|\log y|^2)(1+y^{4k})}=1,
\ee
satisfying the orthogonality condition
\be\label{eq:orthog} 
(u_n,\Phi_M)=0
\ee 
and the smallness:
\be
\label{seeumporgpo}
\int\frac{|Hu_n|^2}{y^4(1+|\log y|^2)(1+y^{4k})}\leq \frac1n.
  \ee
Note that the normalization condition implies 
\be
\label{nvbeoonvoe}
\int\frac{|u_n|^2}{y^4(1+|\log y|^2)(1+y^{4k+4})}+\int\frac{|\pa_yu_n|^2}{y^2(1+|\log y|^2)(1+y^{4k+4})}\lesssim1
\ee and thus from \fref{subceorovity}, the sequence $u_n$ is bounded in $H^2_{loc}$. Hence there exists $u_{\infty}\in H^2_{loc}$ such that up to a subsequence and for any smooth cut-off function $\zeta$ vanishing in a neighborhood of  $y=0$, the sequence $\zeta u_n$ is uniformly bounded in $H^2_{loc}$ and converges to $\zeta u_\infty$ in $H^1_{loc}$. Moreover, \fref{seeumporgpo} implies $$Hu_{\infty}=0$$ and by lower semi continuity of the norm and \fref{normlaization}: $$\int\frac{|u_\infty|^2}{y^4(1+|\log y|^2)(1+y^{4k+4})}<+\infty$$ which implies from \fref{kernelh}: $$u=\alpha\Lambda Q\ \ \mbox{for some}\ \ \alpha\in \Bbb R.$$ We may moreover pass to the limit in  \fref{eq:orthog} from \fref{normlaization} and the local compactness embedding and thus $$(u_\infty,\Lambda Q)=0\ \ \mbox{from which} \ \ \alpha=0$$ where we used the nondegeneracy \fref{nondgeener}. Hence $u_\infty=0$.\\
Now from \fref{cnoceooehoe},  \fref{contorou}, \fref{seeumporgpo} and \fref{normlaization}: $$\int_{y\geq 1}\frac{|u_n|^2}{y^4(1+|\log y|^2)(1+y^{4k+4})}+\int_{y\geq 1}\frac{|\pa_yu_n|^2}{y^6(1+|\log y|^2)(1+y^{4k})}\gtrsim 1$$ and hence from \fref{subceorovity}, \fref{seeumporgpo}: $$\frac{|\pa_yu_n|^2}{1+y^{4k+8}}+\frac{|u_n|^2}{1+y^{4k+10}}\gtrsim 1$$ which from the local compactness of Soboelv embeddings and the a priori bound \fref{nvbeoonvoe} ensures: $$\int\frac{|\pa_yu_{\infty}|^2}{1+y^{4k+8}}+\int\frac{|u_{\infty}|^2}{1+y^{4k+10}}\gtrsim 1$$ which contradicts the established identity $u_{\infty}=0.$
\end{proof}

We are now position to prove the coercivity of the higher order $(\mathcal E_{2k+2})_{0\leq k\leq L}$ energies under suitable orthogonality conditions.
Given a radially symmetric function $\e$, we recall the definition of suitable derivatives: $$\e_{-1}=0, \ \ \e_0=\e, \ \ \e_{k+1}=\left\{\begin{array}{ll} A^*\e_k\ \ \mbox{for k odd}\\ A \e_k\ \ \mbox{for k even}
\end{array}\right., \ \ 0\leq k \leq 2L+1.$$ 

\begin{lemma}[Coercivity of $\mathcal E_{2k+2}$]
\label{propcorc}
Let $L\geq 1$, $0\leq k\leq L$ and $M\geq M(L)$ large enough. Then there exists $C(M)>0$ such that for all $\e$ with:
\bea
\label{cnojeooejbis}
&& \int|\e_{2k+2}|^2+\int\frac{|\e_{2k+1}|^2}{y^2(1+y^2)}\\
\nonumber & + & \Sigma_{p=0}^{k}\int\left[\frac{|\e_{2p-1}|^2}{y^6(1+|\log y|^2)(1+y^{4(k-p)})}+\frac{|\e_{2p}|^2}{y^4(1+|\log y|^2)(1+y^{4(k-p)})}\right]<+\infty
\eea satisfying 
\be
\label{orthoappendix}
 (\e,H^{p}\Phi_M)=0, \ \ 0\leq p\leq k,
\ee
 there holds:
 \bea
 \label{coerciviityley}
 &&\mathcal E_{2k+2}(\e)=\int(H^{k+1}\e)^2 \geq  C(M)\left\{\int\frac{|\e_{2k+1}|^2}{y^2(1+|\log y|^2)}\right.\\
\nonumber  & + &\left. \Sigma_{p=0}^{k}\int\left[\frac{|\e_{2p-1}|^2}{y^6(1+|\log y|^2)(1+y^{4(k-p)})}+\frac{|\e_{2p}|^2}{y^4(1+|\log y|^2)(1+y^{4(k-p)})}\right]\right\}.
 \eea
 \end{lemma}

\begin{proof}[Proof of Lemma \ref{coerch}]

We argue by induction on $k$. The case $k=0$ is Lemma \ref{ceorcetwobis}. We assume the claim for $k$ and prove it for $1\leq k+1\leq L$. Indeed, let $v=H\e$, then $v_{p}=\e_{p+2}$ and thus $v$ satisfies \fref{cnojeooejbis} and\footnote{from $k\leq L+1$}: $$\forall 0\leq p\leq k, \ \ (v,H^p\Phi_M)=(\e,H^{p+1}\Phi_M)=0.$$ We may thus apply the induction claim for $k$ to $v$ and estimate:
\bea
\label{cneoeofofee}
\nonumber &&\int(H^{k+2}\e)^2=\int(H^{k+1}v)^2\geq C(M)\left\{\int\frac{|\e_{2k+3}|^2}{y^2(1+|\log y|^2)}\right.\\
 \nonumber & + &\left. \Sigma_{p=0}^{k}\int\left[\frac{|\e_{2p+1}|^2}{y^6(1+|\log y|^2)(1+y^{4(k-p)})}+\frac{|\e_{2p+2}|^2}{y^4(1+|\log y|^2)(1+y^{4(k-p)})}\right]\right\}\\
& \geq & C(M)\left\{\int\frac{|\e_{2k+3}|^2}{y^2(1+|\log y|^2)}\right.\\
 \nonumber & + &\left. \Sigma_{p=1}^{k+1}\int\left[\frac{|\e_{2p-1}|^2}{y^6(1+|\log y|^2)(1+y^{4(k+1-p)})}+\frac{|\e_{2p}|^2}{y^4(1+|\log y|^2)(1+y^{4(k+1-p)})}\right]\right\}
 \eea
The orthogonality condition $(\e,\Phi_M)=0$ and \fref{cnojeooejbis} allow us to use Lemma \ref{boundweight} and to deduce from the weighted bound \fref{coercivebound} the control:
$$\int\frac{|\e_2|^2}{y^4(1+|\log y|^2)(1+y^{4k})}\gtrsim \int\frac{|\e|^2}{y^4(1+|\log y|^2)(1+y^{4k+4})}$$ which together with \fref{cneoeofofee} concludes the proof of Lemma \ref{coerch}.
\end{proof}


\section{Interpolation bounds}


We derive in this section interpolation bounds on $\e$ in the setting of the bootstrap Proposition \ref{bootstrap}, and which are a consequence of the coercivity Property of Lemma \ref{propcorc}.

\begin{lemma}[Interpolation bounds]
\label{lemmainterpolation}
{\em (i) Weighted Sobolev bounds for $\e_k$}: for $0\leq k\leq L$:
\be
\label{interpolationbound}
\Sigma_{i=0}^{2k+1}\int\frac{|\e_{i}|^2}{y^2(1+y^{4k-2i+2})(1+|\log y|^2)}+\int|\e_{2k+2}|^2\leq C(M) \mathcal E_{2k+2}.
\ee
{\em (ii) Development near the origin}: $\e$ admits a Taylor Lagrange like expansion 
\be
\label{Taylororigin}
\e=\Sigma_{i=1}^{L+1}c_{i}T_{L+1-i}+r_\e
 \ee
 with bounds:
 \be
 \label{boundci}
 |c_{i}|\lesssim C(M)\sqrt{\mathcal E_{2L+2}},
 \ee
 \be
 \label{boundrestebis}
|\pa_y^k r_{\e}|\lesssim y^{2L+1-k}|\log y|C(M)\sqrt{\mathcal E_{2L+2}}, \ \ 0\leq k\leq 2L+1, \ \ y\leq 1.
\ee
\noindent{\em (iii) Bounds near the origin for $\e_k$}: for $|y|\leq \frac 12$,
\be
\label{poutwisepair}
|\e_{2k}|\lesssim C(M)y|\log y|\sqrt{\mathcal E_{2L+2}}, \ \ 0\leq k\leq L,
\ee
\be
\label{poutwiseimpair}
|\e_{2k-1}|\lesssim C(M)y^2|\log y|\sqrt{\mathcal E_{2L+2}}, \ \ 1\leq k\leq L,
\ee
\be
\label{poutwiseimpairlast}
|\e_{2L+1}|\lesssim C(M) \sqrt{\mathcal E_{2L+2}}.
\ee
{\em (iv) Bounds near the origin for $\pa_y^k\e$}: for $|y|\leq \frac 12$,
\be
\label{poutwisepairbis}
|\pa_y^{2k}\e|\lesssim C(M)y|\log y|\sqrt{\mathcal E_{2L+2}}, \ \ 0\leq k\leq L,
\ee
\be
\label{poutwiseimpairbis}
|\pa_y^{2k-1}\e|\lesssim C(M)|\log y|\sqrt{\mathcal E_{2L+2}}, \ \ 1\leq k\leq L+1.
\ee
\noindent{\em (v) Lossy bound}:  
\be
\label{interpolationboundlossy}
\Sigma_{i=0}^{2k+1}\int\frac{1+|\log y|^C}{1+y^{4k-2i+4}}|\e_{i}|^2\lesssim |\log b_1|^C\left\{\begin{array}{ll}b_1^{(4k+2)\frac{L}{2L-1}}, \ \  0\leq k\leq L-1,\\ b_1^{2L+2}\ \ \mbox{for}\ \ k=L,\end{array}\right.
\ee
\be
\label{weight}
\Sigma_{i=0}^{2k+1}\int\frac{1+|\log y|^C}{1+y^{4k-2i+4}}|\pa_y^i\e|^2\lesssim |\log b_1|^C\left\{\begin{array}{ll}b_1^{(4k+2)\frac{L}{2L-1}}, \ \  0\leq k\leq L-1,\\ b_1^{2L+2}\ \ \mbox{for}\ \ k=L,\end{array}\right.
\ee
\noindent{\em (vi) Generalized lossy bound}: Let $(i,j)\in \Bbb N\times \Bbb N^*$ with $2\leq i+j\leq 2L+2$, then:
\be
\label{estimatelossy}
\int\frac{1+|\log y|^C}{1+y^{2j}}|\pa_y^i\e|^2\lesssim |\log b_1|^C\left\{ \begin{array}{lll} b_1^{(i+j-1)\frac{2L}{2L-1}} \ \ \mbox{for}\ 2\leq  i+j\leq 2L\\
 b_1^{2L+1}  \ \ \mbox{for}\ \  i+j= 2L+1\\
  b_1^{2L+2}  \ \ \mbox{for}\   i+j=2L+2.
  \end{array}\right.
\ee
Moreoever: 
\be
\label{inteextremale}
\int \frac{|\pa_y^i\e|^2}{1+|\log y|^2}\lesssim |\log b_1|^C\left\{\begin{array}{ll}b_1^{(i-1)\frac{2L}{2L-1}}, \ \  2\leq i\leq 2L+1,\\ b_1^{2L+2}\ \ \mbox{for}\ \ i=2L+2,\end{array}\right..
\ee
\noindent{\em (vii) Pointwise bound far away}: let $(i,j)\in \Bbb N\times \Bbb N$ with $1\leq i+j\leq 2L+1$, then:
\be
\label{weightbis}
\left\|\frac{\pa_y^i\e}{y^{j}}\right\|^2_{L^{\infty}(y\geq 1)}\lesssim |\log b_1|^C\left\{\begin{array}{lll}  b_1^{(i+j)\frac{2L}{2L-1}}\ \ \mbox{for}\ \ 1\leq i+j\leq 2L-1\\ b_1^{2L+1} \ \ \mbox{for}\ \ i+j= 2L, \\  b_1^{2L+2} \ \ \mbox{for}\ \ i+j= 2L+1\end{array}\right..
\ee
\end{lemma}

\begin{proof} {\bf step1 } Proof of (i). The estimate \fref{interpolationbound} follows from \fref{coerciviityley} with $0\leq k\leq L$.\\

{\bf step 2}. Adapted Taylor expansion.\\
{\em Initialization}: Recall the boundary condition origin at the origin \fref{devayorigin} which implies $|\e_{2L+1}(y)|\leq C_{\e_{2L+1}}y^2$ as $y\to 0$.Together with \fref{otherformula} and the behavior $\Lambda Q\sim y$ near the origin, this implies: 
\be
\label{cneineoneon}
r_1=\e_{2L+1}(y)=\frac{1}{y\Lambda Q}\int_0^y\e_{2L+2}\Lambda Q xdx, 
\ee 
and this yields the pointwise bound for $y\leq 1$:
\bea
\label{boundrone}
\nonumber |r_1(y)|&\lesssim &\frac{1}{y^2}\left(\int_{y\leq 1}|\e_{2L+2}|^2xdx\right)^{\frac 12}\left(\int_{0}^yx^2xdx\right)^{\frac 12}\\
& \lesssim & C(M)\sqrt{\mathcal E_{2L+2}}.
\eea
We now remark that there exists $\frac 12<a<2$ such that $$|\e_{2L+1}(a)|^2\lesssim \int_{|y|\leq 1}|\e_{2L+1}|^2\lesssim C(M)\mathcal E_{2L+2}$$ from \fref{interpolationbound}. We then define $$r_2=-\Lambda Q\int_a^y\frac{r_{1}}{\Lambda Q}dx$$ and obtain from \fref{boundrone} the pointwise bound for $y\leq 1$: 
\be
\label{boundrtwo}
|r_2|\lesssim y|\log y|C(M)\sqrt{\mathcal E_{2L+2}}.
\ee
Now observe that by construction using \fref{otherformula}: 
\be
\label{relationsinti}
Ar_2=r_1=\e_{2L+1}, \ \ Hr_2=A^*\e_{2L+1}=\e_{2L+2}=H\e_{2L}
\ee
Now from \fref{cneoeofofee}, $$\int_{y\leq 1}\frac{|\e_{2L}|^2}{y^4(1+|\log y|^2)}ydy<+\infty$$ and hence $|\e_{2L}(y_n)|<+\infty$ on some sequence $y_n\to 0$, and from \fref{boundrtwo}, \fref{relationsinti}, the explicit knowledge of the kernel of $H$ and the singular behavior \fref{Gamma}, we conclude that there exists $c_2\in \Bbb R$ such that 
\be
\label{fnevnoneo}
 \e_{2L}=c_2\Lambda Q+r_{2}.
 \ee
 Moreover, there exists $\frac 12<a<2$ such that $$|\e_{2L}(a)|^2\lesssim \int_{|y|\leq 1}|\e_{2L}|^2\lesssim C(M)\mathcal E_{2L+2}$$ from \fref{interpolationbound}, and thus from \fref{boundrtwo}, \fref{fnevnoneo}: 
 \be
 \label{eusvbeibie}
 |c_2|\lesssim C(M)\sqrt{\mathcal E_{2L+2}}, \ \ |\e_{2L}|\lesssim y|\log y|C(M)\sqrt{\mathcal E_{2L+2}}.
 \ee
{\em Induction}. We now build by induction the sequence:
$$r_{2k+1}=\frac{1}{y\Lambda Q}\int_0^yr_{2k}\Lambda Q xdx, \ \ r_{2k+2}=-\Lambda Q\int_0^y\frac{r_{2k+1}}{\Lambda Q}dx, \ \  1\leq k\leq L.$$ We claim by induction that for all $1\leq k\le L+1$, $\e_{2L+2-2k}$ admits a Taylor expansion at the origin 
\be
\label{nonneoneoe}
\e_{2L+2-2k}=\Sigma_{i=1}^kc_{i,k}T_{k-i}+r_{2k}, \ \ 1\leq k\leq L+1
\ee
with the bounds for $|y|\leq 1$:
\be
\label{boundcibisbisbibs}
|c_{i,k}|\lesssim C(M)\sqrt{\matchal E_{2L+2}},
\ee
\be
\label{vnoevnoneov}
|\matchal A^ir_{2k}|\lesssim |\log y|y^{2k-1-i}C(M)\sqrt{\matchal E_{2L+2}}, \ \ 0\leq i\leq 2k-1.
\ee
This follows from \fref{fnevnoneo}, \fref{eusvbeibie}, \fref{boundrtwo}, \fref{relationsinti} for $k=1$. We now let $1\leq k\leq L$, assume the claim for $k$ and prove it for $k+1$.\\
By construction using \fref{otherformula}, 
\be
\label{nwoorg}
Ar_{2k+2}=r_{2k+1}, \ \ Hr_{2k+2}=r_{2k}
\ee
and thus $\mathcal A^{i}r_{2k}=r_{2k-i}.$ In particular, for $i\geq 2$, $\mathcal A^{i-2}r_{2k+2}=r_{2k-i}$ and therefore the bounds \fref{vnoevnoneov} for $k+1$ and $2\leq i\leq 2k+1$ follow from the induction claim. We now estimate by definition and induction for $|y|\leq 1$: 
$$
|Ar_{2k+2}|=|r_{2k+1}(y)|=\left|\frac{1}{y\Lambda Q}\int_0^yr_{2k}\Lambda Q xdx\right|\lesssim \frac{C(M)\sqrt{\matchal E_{2L+2}}}{y^2}y^{3+2k-1}
$$
$$|r_{2k+2}|=\left|\Lambda Q\int_0^y\frac{r_{2k+1}}{\Lambda Q}dx\right|\lesssim  y y^{2k}C(M)\sqrt{\matchal E_{2L+2}}
$$
and \fref{vnoevnoneov} is proved for $k=1$ and $i=0,1$. From the regularity at the origin \fref{devayorigin}, \fref{nwoorg}, the relation $$H\e_{2L+2-2(k+1)}=\e_{2L+2-2k}=\Sigma_{i=1}^kc_{i,k}T_{k-i}+r_{2k}$$ and the bound \fref{vnoevnoneov}, there exists $c_{2k+2}$ such that $$\e_{2L+2-2(k+1)}=\Sigma_{i=1}^kc_{i,k}T_{k+1-i}+c_{2k+2}\Lambda Q+r_{2k+2}.$$ We now observe that there exists $\frac 12<a<2$ such that $$|\e_{2L-2k}(a)|^2\lesssim \int_{|y|\leq 1}|\e_{2L-2k}|^2\lesssim C(M)\mathcal E_{2L+2}$$ from \fref{interpolationbound}, and thus using \fref{vnoevnoneov}: $$|c_{2k+2}|\lesssim C(M)\sqrt{\mathcal E_{2L+2}}.$$ This completes the induction claim.\\

{\bf step 3}. Proof of (ii), (iii), (iv). We obtain from \fref{nonneoneoe}, \fref{boundci} with $k=L+1$ the Taylor expansion $$\e=\Sigma_{i=1}^{L+1}c_{i,k}T_{k-i}+r_\e, \ \ r_\e=r_{2L+2}, \ \ |c_{i,k}|\lesssim C(M)\sqrt{\matchal E_{2L+2}},$$ with from \fref{vnoevnoneov}
$$|\matchal A^ir_{\e}|\lesssim |\log y|y^{2L+1-i}C(M)\sqrt{\matchal E_{2L+2}}, \ \ 0\leq i\leq 2L+1.$$ A brute force computation using the expansions \fref{comportementz}, \fref{comportementv} near the origin ensure that for any function $f$:
\be
 \label{expansion}
 \pa_y^kf=\Sigma_{i=0}^kP_{i,k}\mathcal A^if, \ \ |P_{i,k}|\lesssim \frac{1}{y^{k-i}},
 \ee  
and we therefore estimate for $0\leq k\leq 2L+1$:
$$|\pa_y^kr_{\e}|\lesssim C(M)\sqrt{\matchal E_{2L+2}}\Sigma_{i=0}^k \frac{|\log y|y^{2L+1-i}}{y^{k-i}}\lesssim y^{2L+1-k}|\log y| C(M)\sqrt{\matchal E_{2L+2}}.$$ This concludes the proof of (ii). The estimates of (iii), (iv) now directly follow from (ii) using the Taylor expansion of $T_i$ at the origin given by Lemma \ref{lemmaradiation}, and \fref{boundrone} for \fref{poutwiseimpairlast}.\\

{\bf step 4} Proof of (v). We first claim: for $0\leq k\leq L$,
\be
\label{interpolationvnone}
\Sigma_{i=0}^{2k+2}\int\frac{|\pa_y^i\e|^2}{(1+|\log y|^2)(1+y^{4k-2i+4})}\lesssim C(M)(\mathcal E_{2k+2}+\mathcal E_{2L+2}).
\ee
Observe that this implies \fref{inteextremale} by taking $i=2k+2$.\\
Indeed we estimate from \fref{poutwisepairbis}, \fref{poutwiseimpairbis}:
\be
\label{estneineo}
\Sigma_{i=0}^{2k+1}\int_{y\leq 1}\frac{1+|\log y|^C}{1+y^{4k-2i+4}}|\pa_y^i\e|^2\lesssim C(M)\mathcal E_{2L+2}.
\ee
For $y\geq 1$, we recall from the brute force computation \fref{expansion} 
\be
\label{eoneneeno}
|\pa_y^k\e|\lesssim \Sigma_{i=0}^k\frac{|\e_i|}{y^{k-i}}
\ee
and thus using \fref{interpolationbound}: for $0\leq k\leq L$,
\bee
&&\Sigma_{i=0}^{2k+2}\int_{y\geq 1} \frac{|\pa_y^i\e|^2}{(1+|\log y|^2)(1+y^{4k-2i+4})}\\
& \lesssim & \Sigma_{i=0}^{2k+2} \Sigma_{j=0}^{i}\int_{y\geq 1}\frac{|\e_j|^2}{(1+|\log y|^2)(1+y^{4k-2i+4+2i-2j})}\\
& \lesssim & \Sigma_{j=0}^{2k+2}\int\frac{|\e_j|^2}{(1+|\log y|^2)(1+y^{4k+4-2j})}\lesssim C(M)\mathcal E_{2k+2},
\eee
and \fref{interpolationvnone} is proved. In particular, this yields together with the energy bound \fref{init2h} the rough Sobolev bound: $$ \int\frac{|\e|^2}{y^2}+\Sigma_{k=1}^{2L+2}\int \frac{|\pa_y^k\e|^2}{1+|\log y|^2}\lesssim 1.$$ We therefore estimate using \fref{interpolationvnone} again:
\bea
\label{cnkneoneov}
\nonumber &&\Sigma_{i=0}^{2k+1}\int \frac{1+|\log y|^C}{1+y^{4k-2i+4}}|\pa_y^i\e|^2\\
\nonumber & \lesssim & \Sigma_{i=0}^{2k+1}\left[\int_{y\leq B_0^{100L}} \frac{1+|\log y|^C}{1+y^{4k-2i+4}}|\pa_y^i\e|^2+\int_{y\ge B_0^{100L}}\frac{1+|\log y|^C}{y^2}|\pa_y^i\e|^2\right]\\
& \lesssim & |\log b_1|^C\matchal E_{2k+2}+\frac{1}{B_0^{10L}}
\eea
and \fref{weight} follows. The estimate \fref{interpolationboundlossy} now follows from \fref{poutwisepair}, \fref{poutwiseimpair}, \fref{poutwiseimpairlast} for $y\leq 1$ with also \fref{comprsion},  
and \fref{weight} for $y\geq 1$.\\

{\bf step 5} Proof of (vi). Let $i\geq 0$, $j\geq 1$ with $2\leq i+j\leq 2L+2$.\\
{\it case $i+j$ even}: we have $i+j=2(k+1)$, $0\leq k\leq L$. For $k\leq L-1$, we estimate from \fref{weight} and $0\leq i=2k+2-j\leq 2k+1$:
$$\int \frac{1+|\log y|^C}{1+y^{2j}}|\pa_y^i\e|^2=\int\frac{1+|\log y|^C}{1+y^{4k+4-2i}}|\pa_y^i\e|^2\lesssim b_1^{(4k+2)\frac{L}{2L-1}}|\log b_1|^C\lesssim b_1^{(i+j-1)\frac{2L}{2L-1}}|\log b_1|^C.$$
For $k=L$, we have from \fref{weight}: $$\int \frac{1+|\log y|^C}{1+y^{2j}}|\pa_y^i\e|^2=\int\frac{1+|\log y|^C}{1+y^{4k+4-2i}}|\pa_y^i\e|^2\lesssim b_1^{2L+2}|\log b_1|^C.$$

{\it case $i+j$ odd}: we have $i+j=2k+1$, $1\leq k\leq L$. Assume $k\leq L-1$. If $j\geq 2$, then $i\leq 2k+1-j\leq 2(k-1)+1$ and thus from \fref{weight}:
\bee
&&\int \frac{1+|\log y|^C}{1+y^{2j}}|\pa_y^i\e|^2=\int\frac{1+|\log y|^C}{1+y^{4k+2-2i}}|\pa_y^i\e|^2\\
&  \lesssim & \left(\int\frac{1+|\log y|^C}{1+y^{4k+4-2i}}|\pa_y^i\e|^2\right)^{\frac 12}\left(\int\frac{1+|\log y|^C}{1+y^{4(k-1)+4-2i}}|\pa_y^i\e|^2\right)^{\frac 12}\\
& \lesssim &|\log b_1|^Cb_1^{\frac{L}{2(2L-1)}(4k+2+4(k-1)+2)}=b_1^{(i+j-1)\frac{2L}{2L-1}}|\log b_1|^C.
\eee
For the extremal case $j=1$, $i=2k$, $1\leq k \leq L-1$, we estimate from \fref{interpolationboundlossy}, \fref{interpolationvnone}:
\bee
\int \frac{1+|\log y|^C}{1+y^2}|\pa_y^{2k}\e|^2&\lesssim &\left(\int \frac{1+|\log y|^C}{1+y^4}|\pa_y^{2k}\e|^2\right)^{\frac 12}\left(\int \frac{|\pa_y^{2k}\e|^2}{1+|\log y|^2}\right)^{\frac 12}\\
& \lesssim & |\log b_1|^Cb_1^{\frac{L}{2(2L-1)}(4k+2+4(k-1)+2)}=b_1^{(i+j-1)\frac{2L}{2L-1}}|\log b_1|^C
\eee
For $k=L$, then for $j\geq 2$, $i\leq 2k+1-j\leq 2(k-1)+1$ and thus from \fref{weight}:
\bee
&&\int \frac{1+|\log y|^C}{1+y^{2j}}|\pa_y^i\e|^2=\int\frac{1+|\log y|^C}{1+y^{4k+2-2i}}|\pa_y^i\e|^2\\
&  \lesssim & \left(\int\frac{1+|\log y|^C}{1+y^{4k+4-2i}}|\pa_y^i\e|^2\right)^{\frac 12}\left(\int\frac{1+|\log y|^C}{1+y^{4(k-1)+4-2i}}|\pa_y^i\e|^2\right)^{\frac 12}\\
& \lesssim &|\log b_1|^Cb_1^{\frac{1}{2}(2L+2+(4(k-1)+2)\frac{L}{2L-1})}=b^{2L+1}_1|\log b_1|^C,
\eee
and for $j=1$, $i=2L$ from \fref{interpolationboundlossy}, \fref{inteextremale}:
\bee
\int \frac{1+|\log y|^C}{1+y^2}|\pa_y^{2L}\e|^2&\lesssim &\left(\int \frac{1+|\log y|^C}{1+y^4}|\pa_y^{2L}\e|^2\right)^{\frac 12}\left(\int \frac{|\pa_y^{2L}\e|^2}{1+|\log y|^2}\right)^{\frac 12}\\
& \lesssim & |\log b_1|^Cb_1^{\frac{1}{2}(2L+2+(4(L-1)+2)\frac{L}{2L-1})}=b^{2L+1}_1|\log b_1|^C
\eee

{\bf step 6} Proof of (vii). We estimate from Cauchy Schwarz: $$\left\|\frac{\e}{y}\right\|_{L^{\infty}(y\geq 1)}^2\lesssim \int_{y\geq 1}|\e\pa_y\e|dy\lesssim \int\frac{(1+|\log y|^2)|\e|^2}{y^2}+\int \frac{|\pa_y\e|^2}{1+|\log y|^2}.$$ Let then $i,j\geq 0$ with $1\leq i+j\leq 2L+1$, then $2\leq i+j+1\leq 2L$ and we conclude from \fref{estimatelossy}, \fref{inteextremale}:
\bee
\left\|\frac{\pa_y^i\e}{y^j}\right\|_{L^{\infty}(y\geq 1)}^2& \lesssim& \int_{y\geq 1}\frac{(1+|\log y|^2)|\pa_y^{i}\e|^2}{y^{2j+2}}+\int_{y\geq 1}\frac{|\pa_y^{i+1}\e|^2}{y^{2j}(1+|\log y|^2)}\\
& \lesssim & |\log b_1|^C\left\{\begin{array}{lll}  b_1^{(i+j)\frac{2L}{2L-1}}\ \ \mbox{for}\ \ 2\leq i+j+1\leq 2L\\ b_1^{2L+1} \ \ \mbox{for}\ \ i+j+1= 2L+1 \\  b_1^{2L+2} \ \ \mbox{for}\ \ i+j+1= 2L+2\end{array}\right..
\eee
\end{proof}


\section{Leibniz rule for $H^k$}


Given a smooth function $\Phi$, we prove the following Leibniz rule:

\begin{lemma}[Leibniz rule for $H^k$]
\label{leibnizrule}
Let $k\geq1$, then: 
\bea
\label{ebeboboneo}
&&\mathcal A^{2k-1}(\Phi\e)= \Sigma_{i=0}^{k-1}\Phi_{2i,2k-1}\e_{2i}+\Sigma_{i=1}^{k}\Phi_{2i-1,2k-1}\e_{2i-1},\\
\nonumber && \mathcal A^{2k}(\Phi\e)= \Sigma_{i=0}^{k}\Phi_{2i,2k}\e_{2i}+\Sigma_{i=1}^{k}\Phi_{2i-1,2k}\e_{2i-1}
\eea
where $\Phi_{i,k}$ is computed through the recurrence relation:
\bea
\label{intialisa}
&&\Phi_{0,1}=-\pa_y\Phi, \ \ \Phi_{1,1}=\Phi\\
\nonumber && \Phi_{0,2}=-\pa_{yy}\Phi-\frac{1+2Z}{y}\pa_y\Phi,\ \ \Phi_{1,2}=2\pa_y\Phi,\ \ \Phi_{2,2}=\Phi
 \eea
\bea
\label{recurrence}
&& \left\{\begin{array}{llll} \Phi_{2k+2,2k+2}=\Phi_{2k+1,2k+1}\\ \Phi_{2i,2k+2}=\Phi_{2i-1,2k+1}+\pa_y\Phi_{2i,2k+1}+\frac{1+2Z}{y}\Phi_{2i,2k+1}\, \ \ 1\leq i\leq k\\ \Phi_{0,2k+2}=\pa_y\Phi_{0,2k+1}+\frac{1+2Z}{y}\Phi_{0,2k+1}\\ \Phi_{2i-1,2k+2}=-\Phi_{2i-2,2k+1}+\pa_y\Phi_{2i-1,2k+1}, \ \ 1\leq i\leq k+1 \end{array} \right.\\
\label{recurrencebis}&&\left\{\begin{array}{llll} \Phi_{2k+1,2k+1}=\Phi_{2k,2k}\\ \Phi_{2i-1,2k+1}=\Phi_{2i-2,2k}+\frac{1+2Z}{y}\Phi_{2i-1,2k}-\pa_y\Phi_{2i-1,2k}, \ \  1\leq i\leq k\\
\Phi_{2i,2k+1}=-\pa_y\Phi_{2i,2k}-\Phi_{2i-1,2k}, \ \ 1\leq i\leq k\\ \Phi_{0,2k+1}=-\pa_y\Phi_{0,2k} \end{array} \right.
\eea
\end{lemma}

\begin{proof}
We compute:
$$A(\Phi\e)=\Phi\e_1-(\pa_y\Phi)\e,$$ 
\bee
 H(\Phi\e)& =& A^*A\e=\Phi\e_2+\pa_y\Phi\e_1-\left(-A+\frac{1+2Z}{y}\right)(\pa_y\Phi\e)\\
 & = & \Phi\e_2+2\pa_y\Phi\e_1+\left[-\pa_{yy}\Phi-\frac{1+2Z}{y}\pa_y\Phi\right]\e.
\eee
\bee
&&\mathcal A^{2k+1}(\Phi\e)=\Sigma_{i=0}^{k}A\left[\Phi_{2i,2k}\e_{2i}\right]+\Sigma_{i=1}^{k}\left(-A^*+\frac{1+2Z}{y}\right)\Phi_{2i-1,2k}\e_{2i-1}\\
& = & \Sigma_{i=0}^{k}\left\{\Phi_{2i,2k}\e_{2i+1}-\pa_y\Phi_{2i,2k}\e_{2i}\right\}\\
& + & \Sigma_{i=1}^k\left\{-\Phi_{2i-1,2k}\e_{2i}+\left[\frac{1+2Z}{y}\Phi_{2i-1,2k}-\pa_y\Phi_{2i-1,2k}\right]\e_{2i-1}\right\}\\
& = & -\pa_y\Phi_{0,2k}\e+\Sigma_{i=1}^k(-\pa_y\Phi_{2i,2k}-\Phi_{2i-1,2k})\e_{2i}\\
& + & \Sigma_{i=1}^k\left\{\Phi_{2i-2,2k}+\frac{1+2Z}{y}\Phi_{2i-1,2k}-\pa_y\Phi_{2i-1,2k}\right\}\e_{2i-1}+\Phi_{2k,2k}\e_{2k+1},
\eee
this is \fref{recurrencebis},
\bee
\mathcal A^{2k+2}(\Phi\e)& = &\Sigma_{i=0}^k\left[-A+\frac{1+2Z}{y}\right]\left\{\Phi_{2i,2k+1}\e_{2i}\right\} +\Sigma_{1=i}^{k+1}A^*\left(\Phi_{2i-1,2k+1}\e_{2i-1}\right)\\
& = & \Sigma_{i=0}^k\left\{-\Phi_{2i,2k+1}\e_{2i+1}+\left[\pa_y\Phi_{2i,2k+1}+\frac{1+2Z}{y}\Phi_{2i,2k+1}\right]\e_{2i}\right\}\\
& + & \Sigma_{i=1}^{k+1}\left\{\Phi_{2i-1,2k+1}\e_{2i}+\pa_y\Phi_{2i-1,2k+1}\e_{2i-1}\right\}\\
& = & \left[\pa_y\Phi_{0,2k+1}+\frac{1+2Z}{y}\Phi_{0,2k+1}\right]\e+\Phi_{2k+1,2k+1}\e_{2k+2}\\
& +& \Sigma_{i=1}^{k}\left[\Phi_{2i-1,2k+1}+\pa_y\Phi_{2i,2k+1}+\frac{1+2Z}{y}\Phi_{2i,2k+1}\right]\e_{2i}\\
& + & \Sigma_{i=1}^{k+1}\left[-\Phi_{2i-2,2k+1}+\pa_y\Phi_{2i-1,2k+1}\right]\e_{2i-1},
\eee
this is \fref{recurrence}. 
\end{proof}


\section{Proof of \fref{commutator}}
\label{proofcommutator}


A simple induction argument ensures the formula: $$[\pa_t,H_\l^L]w=\Sigma_{k=0}^{L-1}H_{\l}^k[\pa_t,H_\l]H_{\l}^{L-(k+1)}w.$$ We therefore renormalize and compute explicitely: 
\be
\label{formuacommut}
[\pa_t,H_\l^L]w=\frac{1}{\l^{2L+2}}\Sigma_{k=0}^{L-1}H^k\left(-\lsl \frac{\Lambda V}{y^2}H^{L-(k+1)}\e\right).
\ee
We now apply the Leibniz rule Lemma \ref {leibnizrule} with $\Phi=\frac{\Lambda V}{y^2}$. In view of the expansion \fref{comportementv} and the recurrence formula \fref{recurrence}, we have an expansion at the origin 
to all orders for even $k\geq 2$:
$$
\left\{\begin{array}{ll}
\Phi_{2i,2k}(y)=\Sigma_{p=0}^Nc_{i,k,p}y^{2p}+O(y^{2N+2}), \ \ 0\leq i\leq k\\
 \Phi_{2i+1,2k}(y)=\Sigma_{p=0}^Nc_{i,k,p}y^{2p+1}+O(y^{2N+3}), \ \ 1\leq i\leq k-1
 \end{array}\right .
 $$
 and for odd $k\geq 1$:
 $$
  \left\{\begin{array}{ll}
\Phi_{2i-1,2k+1}(y)=\Sigma_{p=0}^Nc_{i,k,p}y^{2p}+O(y^{2N+2}), \ \ 1\leq i\leq k+1\\
 \Phi_{2i,2k+1}(y)=\Sigma_{p=0}^Nc_{i,k,p}y^{2p+1}+O(y^{2N+3}), \ \ 1\leq i\leq k-1
 \end{array}\right . 
$$
and a bound for $y\geq 1$:
$$|\Phi_{i,k}|\lesssim \frac{1}{1+y^{4+(2k-i)}}, \ \ 0\leq i\leq 2k,
$$ and we therefore estimate from \fref{ebeboboneo}:
\be
\label{finalbound}
\forall k\geq 0, \ \ \left|H^k\left(\frac{\Lambda V}{y^2}\e\right)\right|\leq \Sigma_{i=0}^{2k}c_{i,k}\frac{|\e_i|}{1+y^{4+(2k-i)}}.
\ee
Similarily:
\bee
\left|AH^k\left(\frac{\Lambda V}{y^2}\e\right)\right|& \lesssim &  \Sigma_{i=0}^{2k}c_{i,k}\frac{1}{1+y^{4+(2k-i)}}\left[|\pa_y\e_i|+\frac{|\e_i|}{y}\right]\\
& \lesssim & \Sigma_{i=0}^{2k+1}c_{i,k}\frac{|\e_i|}{y(1+y^{4+(2k-i)})}.
\eee
We now inject \fref{finalbound} into \fref{formuacommut} and obtain using \fref{parameters} the pointwise bound on the commutator:
\bee
\left|[\pa_t,H_\l^L]w\right| &\lesssim & \frac{|b_1|}{\l^{2L+2}}\Sigma_{k=0}^{L-1}\Sigma_{i=0}^{2k}c_{i,k}\frac{|\e_{2(L-k-1)+i}|}{1+y^{4+(2k-i)}}\\
& \lesssim & \frac{|b_1|}{\l^{2L+2}}\Sigma_{m=0}^{2L-2}\frac{|\e_{2L-2-m}|}{1+y^{4+m}}=\frac{|b_1|}{\l^{2L+2}}\Sigma_{m=0}^{2L-2}\frac{|\e_{m}|}{1+y^{2+2L-m}}.
\eee
Hence after a change of variables in the integral and using \fref{interpolationbound}: 
\bee
\int \frac{|[\pa_t,H_\l^L]w|^2}{\l^2(1+y^2)}\lesssim \frac{|b_1|^2}{\l^{4L+4}}\Sigma_{m=0}^{2L-2}\int\frac{\e^2_{m}}{(1+y^2)(1+y^{4+4L-2m})}\lesssim \frac{C(M)b_1^2}{\l^{4L+4}}\mathcal E_{2L+2},
\eee
and similarily:
\bee
\int |A_\l[\pa_t,H_\l^L]w|^2\lesssim \frac{|b_1|^2}{\l^{4L+4}}\Sigma_{m=0}^{2L-1}\int\frac{\e^2_{m}}{y^2(1+y^{4+4L-2m})}\lesssim \frac{C(M)b_1^2}{\l^{4L+4}}\mathcal E_{2L+2},
\eee
this is \fref{commutator}.


\section{Proof of \fref{monoenoiencleappouetpoute}}
\label{appendenergy}


We claim the following Lyapounov monotonicity functional for the $\mathcal E_{2k+2}$ energy. 

\begin{proposition}[Lyapounov monotonicity for $\mathcal E_{2k+2}$]
\label{AEI2bis}
Let $0\le k\leq L-1$, then there holds:
\bea
\label{monoenoiencleapp}
&& \frac{d}{dt} \left\{\frac{1}{\lambda^{4k+2}}\left[\mathcal E_{2k+2}+O\left(b_1^{\frac 12}b_1^{(4k+2)\frac{2L}{2L-1}}\right)\right]\right\}\\
\nonumber & \lesssim & \frac{|\log b_1|^{C}}{\l^{4k+4}}\left[b_1^{2k+3}+b_1^{1+\delta+(2k+1)\frac{2L}{2L-1}}+\sqrt{b_1^{2k+4}\matchal E_{2k+2}}\right]
 \eea
for some universal constants $C,\delta>0$ independent of $M$ and of the bootstrap constant $K$ in \fref{init2h}, \fref{init3h}.
\end{proposition}

\begin{proof}[Proof of Proposition \ref{AEI2bis}] {\bf step 1} Modified energy identity. We follow verbatim the algebra of Propositon \ref{monoenoiencle} with $L\to k$ and obtain the modified energy identity:
\bea
\label{modifeideenrgyapp}
\nonumber &&\frac{1}{2}\frac{d}{dt}\left\{ \mathcal E_{2k+2}+2\int \frac{b_1(\Lambda Z)_\l}{\l^2r}w_{2k+1}w_{2k}\right\} =  - \int (\tilde{H}_\l w_{2k+1})^2\\
\nonumber & - & \left(\lsl+b_1\right)\int \frac{(\Lambda \tilde{V})_\l}{2\l^2r^2} w_{2k+1}^2+\int \frac{d}{dt}\left(\frac{b_1(\Lambda Z)_\l}{\l^2r}\right)w_{2k+1}w_{2k}\\
\nonumber & + & \int \tilde{H}_\l w_{2k+1}\left[\frac{\partial_{t} Z_{\lambda}}{r} w_{2k} +A_{\lambda}\left( [\pa_t,H_\l^k]w\right)+A_\l H^k_{\lambda}\left( \frac{1}{\lambda^2}\mathcal F_{\lambda}\right)\right]\\
\nonumber & + & \int \frac{b_1(\Lambda Z)_\l}{\l^2r}w_{2k}\left[-\tilde{H}_\l w_{2k+1}+\frac{\partial_{t} Z_{\lambda}}{r} w_{2k} + A_{\lambda}\left( [\pa_t,H_\l^k]w\right)+A_\l H^k_{\lambda}\left( \frac{1}{\lambda^2}\mathcal F_{\lambda}\right)\right]\\
& + & \int\frac{b_1(\Lambda Z)_\l}{\l^2r}w_{2L+1}\left[ [\pa_t,H_\l^k]w + H^k_{\lambda} \left(\frac 1{\lambda^2} \mathcal F_{\lambda}\right)\right]
\eea
We now estimate all terms in the RHS of \fref{modifeideenrgyapp}.\\

{\bf step 3} Lower order quadratic terms. We treat the lower order quadratic terms in \fref{modifeideenrgyapp} using dissipation. The bound 
\be
\label{commutatork}
\int\frac{\left([\pa_t,H_\l^k]w\right)^2}{\l^2(1+y^2)}+\int\left|A_{\lambda}\left( [\pa_t,H_\l^k]w\right)\right|^2\lesssim C(M)\frac{b_1^2}{\l^{4k+4}}\matchal E_{2k+2}
\ee
follows from \fref{commutator} with $L\to k$. We estimate from \fref{nveknvneo}, the rough bound \fref{rougboundpope} and Lemma \ref{lemmainterpolation}:
\bee
&&\int\left|\tilde{H}_\l w_{2k+1}\left[\frac{\partial_{t} Z_{\lambda}}{r} w_{2k} + \int A_{\lambda}\left( [\pa_t,H_\l^k]w\right)\right]\right|+\int|\tilde{H}_\l w_{2k+1}|\left|\frac{b(\Lambda Z)_\l}{\l^2r}w_{2k}\right|\\
& \leq &\frac{1}{2}\int|\tilde{H}_\l w_{2k+1}|^2+\frac{b_1^2}{\l^{4k+4}}\left[\int \frac{\varepsilon_{2k}^{2}}{1+y^6}+C(M)\matchal E_{2k+2}\right]\\
& \leq &\frac{1}{2}\int|\tilde{H}_\l w_{2k+1}|^2+\frac{b_1}{\l^{4k+4}}C(M)b_1\matchal E_{2k+2}.
\eee
All other quadratic terms are lower order by a factor $b_1$ using again \fref{rougboundpope}, \fref{commutator}, \fref{parameters} and Lemma \ref{lemmainterpolation}:
\bee
&&\left|\lsl+b_1\right|\int \left|\frac{(\Lambda \tilde{V})_\l}{2\l^2r^2} w_{2k+1}^2\right|+\int\left|\frac{b_1(\Lambda Z)_\l}{\l^2r}w_{2k}\left[\frac{\partial_{t} Z_{\lambda}}{r} w_{2k} + A_{\lambda}\left( [\pa_t,H_\l^k]w\right)\right]\right|\\
& + & \int\left|\frac{b_1(\Lambda Z)_\l}{\l^2r}w_{2k+1}[\pa_t,H^L_\l]w\right|+\left|\int \frac{d}{dt}\left(\frac{b_1(\Lambda Z)_\l}{\l^2r}\right)w_{2k+1}w_{2k}\right|\\
& \lesssim & \frac{b_1^2}{\l^{4k+4}}\left[\int\frac{\e_{2k+1}^2}{1+y^4}+\int \frac{\varepsilon_{2k}^2}{1+y^6}+C(M)\mathcal E_{2k+2}\right]\lesssim \frac{b_1}{\l^{4k+4}}C(M)b_1\matchal E_{2k+2}.
\eee 
We similarily estimate the boundary term in time using \fref{interpolationboundlossy}: $$\left|\int \frac{b_1(\Lambda Z)_\l}{\l^2r}w_{2k+1}w_{2k}\right|\lesssim \frac{b_1}{\l^{4k+2}}\left[\int\frac{\e_{2k+1}^2}{1+y^2}+\int\frac{\e_{2k}^2}{1+y^4}\right]\lesssim \frac{b_1}{\l^{4k+2}}|\log b_1|^Cb_1^{(4k+2)\frac{2L}{2L-1}}.$$
We inject these estimates into \fref{modifeideenrgy} to derive the preliminary bound:
\bea
\label{neoheohoheapp}
  &&\frac{1}{2}\frac{d}{dt}\left\{ \frac{1}{\l^{4k+2}}\left[\mathcal E_{2k+2}+O\left(b_1^{\frac 12}b_1^{(4k+2)\frac{2L}{2L-1}}\right)\right]\right\}\leq -\frac12\int(\tilde{H}_\l w_{2k+1})^2\\
\nonumber  & + & \int \tilde{H}_\l w_{2k+1}A_\lambda H^k_\l\left( \frac{1}{\lambda^2}\mathcal F_{\lambda}\right)+\int H^k_{\lambda} \left(\frac 1{\lambda^2} \mathcal F_{\lambda}\right)\left[\frac{b_1(\Lambda Z)_\l}{\l^2r}w_{2k+1}+A^*_{\l}\left( \frac{b_1(\Lambda Z)_\l}{\l^2r}w_{2k}\right)\right]\\
 \nonumber & + &\frac{b_1}{\l^{4k+4}}b_1^{\delta}\mathcal E_{2k+2}
\eea
with constants independent of $M$ for $|b|<b^*(M)$ small enough. We now estimate all terms in the RHS of \fref{neoheohoheapp}.\\

{\bf step 4} Further use of dissipation. Recall the decomposition \fref{fprcingdecomp}. 
The first term in the RHS of \fref{neoheohoheapp} is estimated after an integration by parts:
\bea
\label{oneoneapp}
\nonumber&& \left| \int \tilde{H}_\l w_{2k+1}A_\lambda H^k_\l\left( \frac{1}{\lambda^2}\mathcal F_{\lambda}\right)\right|\\
\nonumber & \leq & \frac{C}{\l^{4k+4}}\|A^*\e_{2k+1}\|_{L^2}\|H^{k+1}\mathcal F_0\|_{L^2}+\frac14\int|\tilde{H}_\l w_{2k+1}|^2+\frac{C}{\l^{4k+4}}\int|AH^k\mathcal F_1|^2\\
& \leq & \frac{C}{\l^{4k+4}}\left[\|H^{k+1}\mathcal F_0\|_{L^2}\sqrt{\mathcal E_{2k+2}}+\|AH^k\mathcal F_1\|_{L^2}^2\right]+\frac14\int|\tilde{H}_\l w_{2k+1}|^2
\eea
for some universal constant $C>0$ independent of $M$. The last two terms in \fref{neoheohoheapp} can be estimated in brute force from Cauchy Schwarz:
\bea
\label{onetwoapp}
\nonumber \left|\int H^k_{\lambda} \left(\frac 1{\lambda^2} \mathcal F_{\lambda}\right)\frac{b_1(\Lambda Z)_\l}{\l^2r}w_{2k+1}\right| & \lesssim&  \frac{b_1}{\l^{4k+4}}\left(\int\frac{1+|\log y|^2}{1+y^4}|H^k\mathcal F|^2\right)^{\frac12}\left(\int\frac{\e_{2k+1}^2}{y^2(1+|\log y|^2)}\right)^{\frac12}\\
& \lesssim & \frac{b_1}{\l^{4k+4}}\sqrt{\mathcal E_{2k+2}}\left(\int\frac{1+|\log y|^2}{1+y^4}|H^k\mathcal F|^2\right)^{\frac12}
\eea
where constants are independent of $M$ thanks to the estimate \fref{coercwthree} for $\e_{2k+1}$. Similarily:
\bea
\label{onethreeapp}
 &&\left|\int H^k_{\lambda} \left(\frac 1{\lambda^2} \mathcal F_{\lambda}\right)A^*_{\l}\left( \frac{b_1(\Lambda Z)_\l}{\l^2r}w_{2k}\right)\right|\\
\nonumber &  \lesssim &  \frac{b_1}{\l^{4k+4}}\left(\int\frac{1+|\log y|^2}{1+y^2}|AH^k\mathcal F|^2\right)^{\frac12}\left(\int\frac{\e_{2k}^2}{(1+y^4)(1+|\log y|^2)}\right)^{\frac12}\\
\nonumber & \lesssim & \frac{b_1}{\l^{4k+4}}C(M)\sqrt{\mathcal E_{2k+2}}\left(\int\frac{1+|\log y|^2}{1+y^2}|AH^k\mathcal F_0|^2+\int|AH^k\mathcal F_1|^2\right)^{\frac12}
\eea
We now claim the bounds:
\be
\label{crucialboundthreeapp}
\int\frac{1+|\log y|^2}{1+y^4}|H^k\mathcal F|^2\leq b_1^{2k+2}|\log b_1|^C
\ee
\be
\label{weigheivbiovheoapp}
\int\frac{1+|\log y|^2}{1+y^2}|AH^k\mathcal F_0|^2 \leq b_1^{2k+2}|\log b_1|^C,
\ee
\be
\label{cnofooeeoapp}
\int|H^{k+1}\mathcal F_0|^2\leq b_1^{2k+4}|\log b_1|^C,
\ee
\be
\label{crucialboundtwoapp}
\int |AH^k\mathcal F_1|^2\leq b_1^{2k+3}|\log b_1|^C+b_1^{1+\delta+(2k+1)\frac{2L}{2L-1}}
\ee
for some universal constants $\delta,C>0$ independent of $M$ and of the bootstrap constant $K$ in \fref{init2h}, \fref{init3h}. Injecting these bounds together with \fref{oneoneapp}, \fref{onetwoapp}, \fref{onethreeapp} into \fref{neoheohoheapp} concludes the proof of \fref{monoenoiencleapp}. We now turn to the proof of \fref{crucialboundthreeapp}, \fref{weigheivbiovheoapp}, \fref{cnofooeeoapp}, \fref{crucialboundtwoapp}.\\

{\bf step 5} $\Psit_b$ terms. We estimate from \fref{controleh4erreurtilde}:
$$\int\frac{1+|\log y|^2}{1+y^4}|H^k\Psit_b|^2\lesssim \int |H^k\Psit_b|^2\lesssim b_1^{2k+2}|\log b_1|^C,$$ 
$$\int \frac{1+|\log y|^2}{1+y^2}|AH^k\Psit_b|^2\lesssim \int |AH^k\Psit_b|^2=\int H^k\Psit_bH^{k+1}\Psit_b\lesssim b_1^{2k+3} |\log b_1|^C,$$
$$\int |H^{k+1}\Psit_b|^2\lesssim b_1^{2(k+1)+2}|\log b_1|^C$$
and \fref{crucialboundthreeapp}, \fref{weigheivbiovheoapp}, \fref{cnofooeeoapp} is proved for $\Psit_b$.\\

{\bf step 6} $\widetilde{Mod(t)}$ terms. Recall \fref{defmodtuntilde}:
\bee
\widetilde{Mod}(t)& = & -\left(\lsl+b_1\right)\Lambda \qbt\\
\nonumber & + &   \Sigma_{i=1}^{L}\left[(b_i)_s+(2i-1+c_{b_1})b_1b_i-b_{i+1}\right]\left[\tT_i+\chi_{B_1}\Sigma_{j=i+1}^{L+2}\frac{\partial S_j}{\partial b_i}\right],
\eee
and the notation \fref{defut}. We will need only the rough bound for $b_1$ admissible functions \fref{roughbound}.\\

{\it Proof of \fref{cnofooeeoapp} for $\widetilde{Mod}$}: We estimate from \fref{roughbound} for $y\leq 2B_1$:
\bee
\nonumber |H^{k+1}S_i|+|H^{k+1}\Lambda S_i|+|H^{k+1}b_i\Lambda \tt_i|&\lesssim &b_1^i(1+y)^{2i-1-(2k+2)}\lesssim  b_1b_1^{i-1}(1+y)^{2i-2k-3}\\
& \lesssim &\frac{b_1|\log b_1|^C}{1+y^{2k+1}}
\eee
and thus using $H\Lambda Q=0$: $$\int|H^{k+1} \Lambda \qbt|^2\lesssim  \int_{y\leq 2B_1}\frac{b^2_1|\log b_1|^C}{1+y^{4k+2}}\lesssim b_1^2|\log b_1|^C.
$$ We also have the rough bound for $1\leq i\leq L$, $i+1\leq j\leq L_2$, $y\le 2B_1$:
\be
\label{reougbound}
|\tT_i|+|\chi_{B_1}\Sigma_{j=i+1}^{L+2}\frac{\partial S_j}{\partial b_i}|\lesssim |\log b_1|^C\left[y^{2i-1}+y^{2j-1}b_1^{j-i}|\log b_1|^C\right]\lesssim |\log b_1|^Cy^{2i-1}
\ee
and similarly for suitable derivatives, and hence the bound:
\bee
&&\Sigma_{i=1}^L\int\left|H^{k+1}\left[\tT_i+\chi_{B_1}\Sigma_{j=i+1}^{L+2}\frac{\partial S_j}{\partial b_i}\right]\right|^2\\
& \lesssim &|\log b_1|^C \int_{y\leq 2B_1}|y^{2L-1-(2k+2)}|^2\lesssim |\log b_1|^CB_1^{4(L-k)-4}\lesssim \frac{|\log b_1|^C}{b_1^{2(L-k)-2}}.
\eee
We therefore obtain from Lemma \ref{modulationequations} the control:
\bee
\int |H^{k+1}\widetilde{Mod(t)}|^2&\lesssim &C(K)|\log b_1|^Cb_1^{2L+2}\left[b_1^2+\frac{1}{b_1^{2(L-k)-2}}\right]\lesssim C(K)b_1^{2k+4}|\log b_1|^C\\
& \lesssim & |\log b_1|^Cb_1^{2k+4}
\eee
for $b_1<b_1^*(M)$ small enough.\\
{\it Proof of \fref{crucialboundthreeapp}, \fref{weigheivbiovheoapp}}: We estimate:
\bee
 |H^{k}S_i|+|H^{k}\Lambda S_i|+|H^{k}b_i\Lambda \tt_i|&\lesssim &b_1^i(1+y)^{2i-1-2k}\lesssim  \frac{|\log b_1|^C}{1+y^{2k+1}}
 \eee
 and thus
 \bee
&&\int \frac{1+|\log y|^2}{1+y^4}|H^k\Lambda \qbt|^2+\int \frac{1+|\log y|^2}{1+y^2}|AH^k\Lambda \qbt|^2\lesssim |\log b_1|^C\int \frac{1}{1+y^{4k+2}}\lesssim |\log b_1|^C.
\eee
We then estimate from \fref{reougbound}:
\bee
&&\Sigma_{i=1}^L\int\frac{1+|\log y|^2}{1+y^4}\left|H^{k}\left[\tT_i+\chi_{B_1}\Sigma_{j=i+1}^{L+2}\frac{\partial S_j}{\partial b_i}\right]\right|^2\\
& + & \Sigma_{i=1}^L\int\frac{1+|\log y|^2}{1+y^2}\left|AH^{k}\left[\tT_i+\chi_{B_1}\Sigma_{j=i+1}^{L+2}\frac{\partial S_j}{\partial b_i}\right]\right|^2\\
& \lesssim & |\log b_1|^C\int _{y\leq 2B_1}|y^{2L-1-2k-2}|^2\lesssim \frac{|\log b_1|^C}{b_1^{2(L-k)-2}}
\eee
and hence the bound using Lemma \ref{modulationequations}:
\bee
&&\int\frac{1+|\log y|^2}{1+y^4}|H^k\widetilde{Mod}|^2+\int\frac{1+|\log y|^2}{1+y^2}|AH^k\widetilde{Mod}|^2\\
& \lesssim &|\log b_1|^C C(K)b_1^{2L+2}\left[1+\frac{1}{b_1^{2(L-k)-2}}\right]\lesssim b_1^{2k+2}.
\eee

{\bf step 7} Nonlinear term $N(\e)$. {\it Control near the origin $y\leq 1$}: The control near the origin follows directly from \fref{controlrogi}.\\
 {\it Control for $y\geq 1$}:  We detail the proof of the most delicate bound \fref{crucialboundtwoapp}, the proof of \fref{crucialboundthreeapp}, \fref{weigheivbiovheoapp}  follows similar lines and is left to the reader.\\
 Recall the notations \fref{cnekonronroenoer} and the bounds \fref{zetpovctlinfty}, \fref{zetpovctlinftybis}, \fref{energyboundbis} on $\zeta$. We then have the bounds \fref{pointwisebound}, \fref{defakkk}, \fref{novioe} on $N_1(\e)$ which yield:
\bee
|AH^kN(\e)|&\lesssim& \Sigma_{p=0}^{2k+1}\frac{|\pa_y^pN(\e)|}{y^{2k+1-p}}\lesssim \Sigma_{p=0}^{2k+1}\frac{1}{y^{2k+1-p}}\Sigma_{i=0}^p|\pa_y^i\zeta^2||\pa_y^{p-i}N_1(\e)|\\
& \lesssim &  \Sigma_{p=0}^{2k+1}\frac{|\pa_y^k\zeta^2|}{y^{2k+1-p}}+\Sigma_{p=1}^{2k+1}\frac{1}{y^{2k+1-p}}\Sigma_{i=0}^{p-1}|\pa_y^i\zeta^2| |\log b_1|^C\left[\frac{1}{y^{p-i+1}}+b_1^{\frac{a_{p-i}}{2}}\right]\\
& \lesssim &  \Sigma_{p=0}^{2k+1}\frac{|\pa_y^p\zeta^2|}{y^{2k+1-p}}+|\log b_1|^C\Sigma_{i=0}^{2k}\frac{|\pa_y^i\zeta^2|}{y^{2k+2-i}} +|\log b_1|^C\Sigma_{p=1}^{2k+1}\Sigma_{i=0}^{p-1} b_1^{\frac{a_{p-i}}{2}}\frac{|\pa_y^i\zeta^2|}{y^{2k+1-p}}\\
& \lesssim & |\log b_1|^C\left[  \Sigma_{p=0}^{2k+1}\frac{|\pa_y^p\zeta^2|}{y^{2k+1-p}}+\Sigma_{p=1}^{2k+1}\Sigma_{i=0}^{p-1} b_1^{\frac{a_{p-i}}{2}}\frac{|\pa_y^i\zeta^2|}{y^{2k+1-p}}\right]
\eee
and hence:
\bee
\int_{y\geq 1}|AH^kN(\e)|^2&\lesssim & |\log b_1|^C \Sigma_{p=0}^{2k+1}\Sigma_{i=0}^p\int_{y\geq 1}\frac{|\pa_y^i\zeta|^2|\pa_y^{p-i}\zeta|^2}{y^{4L+2-2p}}\\
& + &  |\log b_1|^C \Sigma_{p=1}^{2k+1}\Sigma_{i=0}^{p-1} \Sigma_{j=0}^ib_1^{a_{p-i}}\int_{y\geq 1}\frac{|\pa_y^j\zeta|^2|\pa_y^{i-j}\zeta|^2}{y^{4L+2-2p}}.
\eee
We now claim the bounds 
\be
\label{boudntoneereapp}
\Sigma_{p=0}^{2k+1}\Sigma_{i=0}^p\int_{y\geq 1}\frac{|\pa_y^i\zeta|^2|\pa_y^{p-i}\zeta|^2}{y^{4k+2-2p}}\leq b_1b_1^{\delta}b_1^{(2k+1)\frac{2L}{2L-1}}
\ee
\be
\label{boudntoneerebisapp}
  |\log b_1|^C \Sigma_{p=1}^{2k+1}\Sigma_{i=0}^{p-1} \Sigma_{j=0}^ib_1^{a_{p-i}}\int_{y\geq 1}\frac{|\pa_y^j\zeta|^2|\pa_y^{i-j}\zeta|^2}{y^{4k+2-2p}}\leq b_1b_1^{\delta}b_1^{(2k+1)\frac{2L}{2L-1}}
 \ee
for some $\delta>0$, and this concludes the proof of \fref{crucialboundtwoapp} for $N(\e)$.\\
\noindent{\it Proof of \fref{boudntoneere}}: Let $0\leq k\leq L-1$, $0\leq p\leq 2k+1$, $0\leq i\leq p$. Let $I_1=p-i$, $I_2=i$, then we can pick $J_2\in \Bbb N^*$ such that $$ \max\{1;2-i\}\leq J_2\leq \min\{2k+3-p;2k+2-i\}$$ and define $$J_1=2k+3-p-J_2.$$ Then from direct inspection,
$$
(I_1,J_1,I_2,J_2)\in \Bbb N^3\times \Bbb N^*, \ \ \left\{\begin{array}{ll}1\leq I_1+J_1\leq 2k+1\leq 2L-1, \ \ 2\leq I_2+J_2\leq 2k+2\leq 2L,\\ I_1+I_2+J_1+J_2=2k+3.\end{array}\right. .
$$
Hence from \fref{zetpovctlinfty}, \fref{zetpovctlinftybis}:
\bee
\int_{y\geq 1}\frac{|\pa_y^i\zeta|^2|\pa_y^{p-i}\zeta|^2}{y^{4k+2-2p}}&\lesssim& \left\|\frac{\pa_y^{I_1}\zeta}{y^{J_1-1}}\right\|_{L^{\infty(y\geq 1)}}^2\int_{y\geq 1}\frac{|\pa_y^{I_2}\zeta|^2}{y^{2J_2-2}}\\
& \lesssim & |\log b_1|^{C(K)}b_1^{(I_1+J_1+I_2+J_2-1)\frac{2L}{2L-1}}=|\log b_1|^{C(K)}b_1^{(2k+2)\frac{2L}{2L-1}}\\
& \leq & b_1b_1^{\delta}b_1^{(2k+1)\frac{2L}{2L-1}}.
\eee
{\it Proof of \fref{boudntoneerebis}}: Let $0\leq k\leq L-1$, $1\leq p\leq 2k+1$, $0\leq j\leq i\leq p-1$. For $p=2k+1$ and $0\leq i=j\leq 2k$, we use the energy bound \fref{energyboundbis} to estimate:
\bee
b_1^{a_{p-i}}\int_{y\geq 1}\frac{|\pa_y^j\zeta|^2|\pa_y^{i-j}\zeta|^2}{y^{4k+2-2p}} & = & b_1^{a_{2k+1-i}}\|\zeta\|^2_{L^{\infty}(y\geq 1)}\int_{y\ge 1} |\pa_y^i\zeta|^2\\
& \lesssim & b_1^{\frac{2L}{2L-1}((2k+1-i)+1+i)}|\log b_1|^{C(K)}\leq b_1b_1^{\delta}b_1^{(2k+1)\frac{2L}{2L-1}}.
\eee
This exceptional case being treated, we let $I_1=j$, $I_2=i-j$ and pick $J_2\in \Bbb N^*$ with $$\max\{1;2-(i-j);2-(p-j)\}\leq J_2\leq \min\{2k+3-p;2k+2-(p-j);2k+2-(i-j)\}.$$ Let $$J_1=2k+3-p-J_2,$$ then from direct check: 
$$
(I_1,J_1,I_2,J_2)\in \Bbb N^3\times \Bbb N^*, \ \ \left\{\begin{array}{ll}1\leq I_1+J_1\leq 2k+1, \ \ 2\leq I_2+J_2\leq 2k+2,\\ I_1+I_2+J_1+J_2=2k+3-(p-i).\end{array}\right. .
$$
and thus:
\bee
b_1^{a_{p-i}}\int_{y\geq 1}\frac{|\pa_y^j\zeta|^2|\pa_y^{i-j}\zeta|^2}{y^{4k+2-2p}}&\lesssim &b_1^{a_{p-i}}\left\|\frac{\pa_y^{I_1}\zeta}{y^{J_1-1}}\right\|^2_{L^{\infty}(y\ge1)}\int_{y\geq 1}\frac{|\pa_y^{I_2}\zeta|^2}{y^{2J_2-2}}\\
 & \lesssim & |\log b_1|^{C(K)}b_1^{(p-i+I_1+J_1+I_2+J_2-1)\frac{2L}{2L-1}}=|\log b_1|^{C(K)}b_1^{(2k+2)\frac{2L}{2L-1}}\\
& \leq & b_1b_1^{\delta}b_1^{(2k+1)\frac{2L}{2L-1}}.\eee

{\bf step 8} Small linear term $L(\e)$. We recall the decomposition \fref{cnoheiohoe}.\\
{\it Control for $y\leq 1$}: The  control near the origin directly follows from \fref{coneonienoeg}.\\
{\it Control for $y\geq 1$}: We give the detailed proof of \fref{crucialboundtwoapp} and leave \fref{crucialboundthreeapp} to the reader. We recall the bound \fref{nknnveo}
$$|\pa_y^k L(\e)|\lesssim \Sigma_{i=0}^k\frac{b_1|\log b_1|^C|\pa_y^i\e|}{y^{k-i+1}}$$ which implies:
\bee
|AH^kL(\e)|&\lesssim &\Sigma_{p=0}^{2k+1}\frac{|\pa_y^pL(\e)|}{y^{2k+1-p}}\lesssim \Sigma_{p=0}^{2k+1}\frac{1}{y^{2k+1-p}}\Sigma_{i=0}^{p}\frac{b_1|\log b_1|^C|\pa_y^i\e|}{y^{p-i+1}}\\
& \lesssim & b_1|\log b_1|^C\Sigma_{i=0}^{2k+1} \frac{|\pa_y^i\e|}{y^{2k+2-i}}.
\eee
We therefore conclude from \fref{weight}:
\bee
\int_{y\geq 1} |AH^kL(\e)|^2&\lesssim& b_1^2|\log b_1|^C\Sigma_{i=0}^{2k+1}\int_{y\geq 1}\frac{|\pa_y^i\e|^2}{y^{4k+4-2i}}\lesssim |\log b_1|^{C(K)}b_1^{2+(2k+1)\frac{2L}{2L-1}}\\
& \leq & b_1^{1+\delta+(2k+1)\frac{2L}{2L-1}}
\eee
 and \fref{crucialboundtwo} is proved.\\ 
This concludes the proof of \fref{crucialboundthreeapp}, \fref{weigheivbiovheoapp}, \fref{cnofooeeoapp}, \fref{crucialboundtwoapp} and thus of Proposition \ref{AEI2}.

\end{proof}

\end{appendix}

\frenchspacing
\bibliographystyle{plain}

\begin{thebibliography}{99}

\bibitem{matano} Angenent, S.B.; Hulshof, J.; Matano, H., The radius of vanishing bubbles in equivariant harmonic map flow from $D^2$ to $\Bbb S^2$, SIAM J. Math. Anal, vol 41, no 3, pp 1121-1137.

\bibitem{BT} Bejenaru, I.; Tataru, D., Near soliton evolution for equivariant Schr\"odinger Maps in two spatial dimensions, arXiv:1009.1608.

 \bibitem{heatflow} Van den Berg, G.J.B.; Hulshof, J.; King, J., Formal asymptotics of bubbling in the harmonic map heat flow, SIAM J. Appl. Math. vol 63, o5. pp 1682-1717.

\bibitem{CDY} Chang, K-C.; Ding, W.Y.; Ye, R.,  Finite-time blow-up of the heat flow of harmonic maps
from surfaces, J. Differential Geom., 36 (1992), pp. 507–515.

\bibitem{CD} Coron, J.M.; Ghidaglia, J-M., Explosion en temps fini pour le flot des applications
harmoniques, C. R. Acad. Sci. Paris S\'er. I Math., 308 (1989), pp. 339–344.

\bibitem{cote} C\^ote, R., Instability of nonconstant harmonic maps for the (1+2)-dimensional equivariant wave map system, Int. Math. Res. Not. 2005, no. 57, 3525–3549.

\bibitem{CMM}  Côte, R.; Martel, Y.; Merle, F., Construction of multi-soliton solutions for the $L^2$-supercritical gKdV and NLS equations, Rev. Mat. Iberoam. 27 (2011), no. 1, 273–302.

\bibitem{CZ} C\^ote, R.; Zaag, H., Construction of a multi-soliton blow-up solution to the semilinear wave equation in one space dimension,  arXiv:1110.2512

\bibitem{dingtian} Ding, W.-Y.; Tian, G.,  Energy identity for a class of approximate harmonic maps from surfaces, Comm. Anal. Geom. 3, 543–554 (1995)

\bibitem{DK} Donninger, R.; Krieger, J., Nonscattering solutions and blowup at infinity for the critical wave equation, preprint, arXiv:1201.3258

\bibitem{FHV} Filippas, S.; Herrero, M.A.; Velázquez, Juan J. L., Fast blow-up mechanisms for sign-changing solutions of a semilinear parabolic equation with critical nonlinearity, R. Soc. Lond. Proc. Ser. A Math. Phys. Eng. Sci. 456 (2000), no. 2004, 2957–2982. 

\bibitem{GHL} Gallot, S.; Hulin, D.; Lafontaine, J., Riemannian geometry, Third edition. Universitext. Springer-Verlag, Berlin, 2004.

\bibitem{NT1} Guan, M.; Gustafson, S.; Tsai, T-P, Global existence and blow-up for harmonic map heat flow. J. Differential Equations 246 (2009), no. 1, 1–20.

\bibitem{NT2} Gustafson, S.; Nakanishi, K.; Tsai, T-P.; Asymptotic stability, concentration and oscillations in harmonic map heat flow, Landau Lifschitz and Schr\"odinger maps on $\Bbb R^2$, Comm. Math. Phys. (2010), 300, no 1, 205-242.

\bibitem{Velas} Herrero, M.A.; Vel\'azquez, J.J.L., Explosion de solutions des équations paraboliques semilin\'eaires
supercritiques, C. R. Acad. Sci. Paris 319, 141–145 (1994).

\bibitem{HR} Hillairet, M.; Rapha\"el, P., Smooth type II blow up solutions to the energy critical focusing wave equation in dimension four, to appear Analysis and PDE.


\bibitem{KMR} Krieger, J.; Martel, Y.; Raphaël, P., Two-soliton solutions to the three-dimensional gravitational Hartree equation, Comm. Pure Appl. Math. 62 (2009), no. 11, 1501–1550.

\bibitem{KSNLS}  Krieger, J.; Schlag, W., Non-generic blow-up solutions for the critical focusing NLS in 1-D, J. Eur. Math. Soc. (JEMS) 11 (2009), no. 1, 1–125.

\bibitem{KST} Krieger, J.; Schlag, W.; Tataru, D. Renormalization and blow up for charge one equivariant critical wave maps. Invent. Math. 171 (2008), no. 3, 543--615.

\bibitem{qingtian} Qing, J., Tian, G., Bubbling of the heat flows for harmonic maps from surfaces, Comm. Pure Appl. Math. 50, 295–310 (1997).

\bibitem{MMJMPA} Martel, Y.; Merle, F., A Liouville theorem for the critical generalized Korteweg-de Vries equation, J. Math. Pures Appl. (9) 79 (2000), no. 4, 339–425.

\bibitem{MMR3} Martel, Y.; Merle, F.; Rapha\"el, P., Blow up for the critical gKdV equation III: exotic regimes, preprint (2012), submitted.

\bibitem{MMR2} Martel, Y.; Merle, F.; Rapha\"el, P., Blow up for the critical gKdV equation II: minimal mass blow up,  preprint (2012), submitted.

\bibitem{MMR1} Martel, Y.; Merle, F.; Rapha\"el, P., Blow up for the critical gKdV equation I: dynamics near the solitary wave,  preprint (2012), submitted.

\bibitem {MaM1}  Matano, H.; Merle, F., Classification of type I and type II behaviors for a supercritical nonlinear heat equation, J. Funct. Anal. 256 (2009), no. 4, 992–1064.

\bibitem{MaM2}  Matano, H.; Merle, F., On nonexistence of type II blowup for a supercritical nonlinear heat equation, Comm. Pure Appl. Math. 57 (2004), no. 11, 1494–1541.

\bibitem{Mm} Merle, F., Construction of solutions with exactly $k$ blow-up points for the Schrödinger equation with critical nonlinearity, Comm. Math. Phys. 129 (1990), no. 2, 223–240. 

\bibitem{MR1} Merle, F.; Rapha\"el, P., Blow up dynamic and upper bound on the blow up rate for critical nonlinear Schr\"odinger equation, Ann. Math. 161 (2005), no. 1, 157--222.

\bibitem{MR2} Merle, F.; Rapha\"el, P., Sharp upper bound on the blow-up rate for the critical nonlinear Schrödinger equation, Geom. Funct. Anal. 13 (2003), no. 3, 591–642

\bibitem{MR4} Merle, F.; Rapha\"el, P., Sharp lower bound on the blow up rate for critical nonlinear Schr\"odinger equation, J. Amer. Math. Soc. 19 (2006), no. 1, 37--90.

\bibitem{MR5} Merle, F.; Rapha\"el, P.,  Profiles and quantization of the blow up mass for critical nonlinear Schr\"odinger equation, Comm. Math. Phys.  253  (2005),  no. 3, 675--704.

\bibitem{MRR} Merle, F.; Rapha\"el, P.; Rodnianski, I., Blow up dynamics for smooth solutions to the energy critical Schr\"odinger map, preprint 2011.

\bibitem{MRS2010} Merle, F.; Rapha\"el, P.; Szeftel, J., Instability of Bourgain Wang solutions for the $L^2$ critical NLS, to appear in Amer. Jour. Math.

\bibitem{Mizo}  Mizoguchi, N., Rate of type II blowup for a semilinear heat equation, Math. Ann. 339 (2007), no. 4, 839–877. 
  
  \bibitem{textzurich} Rapha\"el, P., On the singularity formation for the nonlinear Schr\"odinger equation, Lecture notes of the 2008 Clay Summer school, Zurich, preprint (2012).
  
  \bibitem{RaphRod} Rapha\"el, P.; Rodnianksi, I., Stable blow up dynamics for the critical corotational wave maps and equivariant Yang Mills problems, to appear in Prep. Math. IHES.
    
  \bibitem{RSc1} Rapha\"el, P.; Schweyer, R., Stable blow up dynamics for the 1-corotational energy critical harmonic heat flow, to appear in Comm. Pure App. Math.
  
  \bibitem{RSc2} Rapha\"el, P.; Schweyer, R., On the stability of critical chemotaxis aggregation, prerpint 2012, submitted.
  
  \bibitem{Sc} Schweyer, R., Type II blow up for the four dimensional energy critical semi linear heat equation, J. Funct. Anal., 263 (2012), pp. 3922-3983.
  
  \bibitem{Struwe} Struwe, M., On the evolution of harmonic mappings of Riemannian surfaces, Comment. Math. Helv. 60 (1985), no. 4, 558–581. 
  
  \bibitem{Topping} Topping, P., Winding behaviour of finite-time singularities of the harmonic map heat flow, Math. Z. 247 (2004), no. 2, 279–302.
  
 
  
  \end{thebibliography}

\end{document}